\documentclass[Threeside,11pt]{article}
\usepackage{etoolbox}
\makeatletter
\patchcmd{\thebibliography}
{\list}
{\small          
	\setstretch{1.2}
	\list}
{}{}
\makeatother
\usepackage{setspace}
\usepackage{tikz}
\usepackage{float}
\definecolor{myPurple}{RGB}{109,90,207}
\definecolor{myGreen}{RGB}{4,130,67}
\definecolor{myGold}{RGB}{253,177,71}
\definecolor{myBurgundy}{RGB}{63,1,44}
\definecolor{myTeal}{RGB}{8,119,127}
\definecolor{myCream}{RGB}{255,216,177}
\definecolor{myOrange}{RGB}{225,119,1}
\usetikzlibrary{arrows.meta,bending,calc}
\tikzset{
	ball3Donly/.style={
		circle, minimum size=1.4cm,
		shading=ball,
		ball color=myPurple,
		draw=myPurple
	}
}
\pgfmathsetmacro{\wNear}{2.0pt}
\pgfmathsetmacro{\wMid}{1.3pt}
\pgfmathsetmacro{\wFar}{0.7pt}
\pgfmathsetmacro{\bend}{6mm}
\usepackage[latin1]{inputenc}
\usepackage[english]{babel}
\usepackage{amsmath}
\usepackage{amsthm}
\usepackage{algorithm}
\usepackage{cite}
\usepackage{amsfonts}
\usepackage{amssymb}
\usepackage{mathrsfs}
\usepackage{graphicx}
\usepackage{colortbl,dcolumn}
\usepackage{marvosym}
\usepackage{ifsym}
\usepackage{paralist}
\usepackage{xcolor}
\usepackage{subcaption}
\usepackage{adjustbox}
\usepackage{enumitem}
\usepackage{booktabs,tabularx}
\usepackage[colorlinks,citecolor=blue,
linkcolor=blue,hypertexnames=false]{hyperref} 
\allowdisplaybreaks
\DeclareFontFamily{U}{mathx}{\hyphenchar\font45}
\DeclareFontShape{U}{mathx}{m}{n}{
	<5> <6> <7> <8> <9> <10>
	<10.95> <12> <14.4> <17.28> <20.74> <24.88>
	mathx10
}{}
\DeclareSymbolFont{mathx}{U}{mathx}{m}{n}
\DeclareMathAccent{\widecheck}{0}{mathx}{"71}
\usepackage{graphicx}
\usepackage{caption}
\usepackage{amsthm}

\newtheorem{lemma}{\bf Lemma}[section]
\newtheorem{mainthm}{Theorem}

\newtheorem{remark}{\bf Remark}[section]

\numberwithin{equation}{section}

\newtheorem{innercustomrem}{Remark}
\newenvironment{customrem}[1]
{\renewcommand\theinnercustomrem{#1}\innercustomrem}
{\endinnercustomrem}

\definecolor{PlanarRevisionGreen}{RGB}{0,100,50}

\begin{document}

\title{{\sl Complete characterization of planar $C^1$ linearizability and optimal conjugacy regularity}}
	
	\author{Zhicheng Tong\thanks{School of Mathematics, Jilin University, Changchun 130012, P. R. China. Email: \url{tongzc25@jlu.edu.cn}} 
		\and 
		Yong Li\thanks{Corresponding author. School of Mathematics, Jilin University, Changchun 130012, P. R. China; Center for Mathematics and Interdisciplinary Sciences, Northeast Normal University, Changchun 130024, P. R. China. Email: \url{liyong@jlu.edu.cn}}
	}
	
	\renewcommand{\thefootnote}{}
	
	\date{}
	\maketitle

	\begin{abstract}
We establish a complete characterization of local $C^1$ linearizability for planar diffeomorphisms. For every prescribed hyperbolic real Jordan form, we obtain necessary and sufficient integrability conditions on the modulus of continuity of the derivative for local $C^1$ linearizability. These criteria reduce to three canonical thresholds: the classical, polynomially weighted, and squared-logarithmically weighted Dini conditions. Each threshold is optimal: whenever the corresponding condition fails, we construct a diffeomorphism with the identical linear part that is not locally $C^1$ linearizable. For every non-hyperbolic linear part, explicit polynomial counterexamples show that smoothness alone cannot guarantee even local topological linearization. Beyond $C^1$ existence, we determine the optimal regularity for the derivative of the linearizing conjugacy in every hyperbolic case, with the optimality established by counterexamples.
		\vspace{2mm}
	\end{abstract}
	
	\tableofcontents
	
	\section{Introduction}\label{SEC1}
	\setcounter{footnote}{0}
	\renewcommand{\thefootnote}{{\arabic{footnote}}}

	\subsection{Background}
	Preserving finer dynamical properties under linearization is a longstanding problem dating back to the work of Poincar\'e. This naturally motivates the central question of lowering the required regularity of the system while achieving higher regularity for the linearizing conjugacy, which has repeatedly been raised explicitly along this line; see, for instance, Hartman \cite{Har64}, Stowe \cite{Sto86}, Guysinsky--Hasselblatt--Rayskin~\cite{GHR03}, Newhouse \cite{New17}, and  Lu--Zhang--Zhang \cite{LZZ20}. The decisive problem to be settled in this direction is:
	\begin{center}
		\textit{\textbf{What minimal regularity of the system guarantees $C^1$ linearizability?}}
	\end{center}		
	
	Linearization is a foundational problem in the local theory of dynamical systems: near a stationary state, the principal aim is to replace a nonlinear mapping by its derivative, thereby reducing the local orbit geometry and phase portrait to spectral data. In the hyperbolic regime, the celebrated Hartman--Grobman theorem achieves this reduction at the topological level, asserting that any $C^1$ diffeomorphism is locally conjugate to its linear part near a hyperbolic fixed point via a homeomorphism \cite{Gro59,Har64,Pug69}.

	In many applications, however, topological equivalence is far from adequate \cite{GHR03}. A $C^1$ conjugacy is substantially more informative, preserving tangent directions, differentiable invariant manifolds, and first-order geometric structures that are lost under homeomorphisms. Such $C^1$ linearization plays an essential role across a broad spectrum of problems, ranging from homoclinic bifurcations and tangencies \cite{Den89,AAIS99,HW99}, mixing of hyperbolic flows \cite{FMT07}, and the dynamics of Lorenz attractors \cite{HM07}, to semilinear elliptic equations \cite{Bud89,DFG10}, entropy rigidity for Anosov flows \cite{DLVY20}, and length-spectrum rigidity in billiards \cite{HKS18,CKZ23}.

	The passage from $C^0$ to $C^1$ conjugacy is nevertheless delicate: even very simple hyperbolic mappings may be $C^0$, but not $C^1$, conjugate to their linear parts. The analytic theory was initiated by Poincar\'e \cite{Poi90}, while $C^r$ ($r>1$) linearization was developed systematically by Sternberg \cite{Ste57b,Ste58}. Important $C^r$ linearization and spectral-gap results were obtained by Belitskii \cite{Bel73} and Sell \cite{Sel85}, and extensions to Banach spaces were developed independently by ElBialy \cite{ElB01} and Rodrigues--Sol\`a-Morales \cite{RS04}. Van Strien \cite{van90} pioneered the investigation of $C^0$ linearizations that are differentiable at the fixed point, while sharp regularity problems for hyperbolic diffeomorphisms were addressed by Stowe \cite{Sto86} and Zhang--Zhang--Jarczyk \cite{ZZJ14}. In general, this broad theory is governed by an intricate interplay among nonresonance conditions, spectral separation, and regularity assumptions. Over the past decades, smooth linearization\footnote{Not necessarily $C^1$, but at least differentiable at the fixed point.} has been extended to nonautonomous and random dynamics by Dragi\v{c}evi\'c--Zhang--Zhang \cite{DZZ19,DZZ20}, Lu--Zhang--Zhang \cite{LZZ20}, and Zhang--Lu--Zhang \cite{ZLZ23}. In the deterministic setting, this line of inquiry has also witnessed continuous advances; see, for instance, Newhouse \cite{New17}, Casta\~neda--Jara \cite{CJ24},  Eynard-Bontemps--Navas \cite{EN25}, Lu--Xia \cite{LX25}, Lu--Xia--Zhang \cite{LXZ25}, Casta\~neda--Huerta--Robledo \cite{CHR26}, Lu--Xia--Zhang--Zhang \cite{LXZZ26}, and references therein.

	However, the optimal regularity condition ensuring $C^1$ linearizability remains elusive and continues to draw considerable interest. For instance, Guysinsky--Hasselblatt--Rayskin~\cite{GHR03} underscored the limitations of purely topological conjugacies, motivating the question of whether the smoothness of the linearizing conjugacy can be improved. Along this line, Newhouse~\cite{New17} raised the fundamental problem of whether one can ``\textit{reduce the regularity assumption on $T$ and still get a $C^1$ linearization.}''  More recently, Lu--Zhang--Zhang~\cite{LZZ20} explicitly emphasized that a basic problem in this theory is to find optimal regularity conditions under which a diffeomorphism admits a $C^1$ linearization. While questions of this nature have a long history, their explicit formulation in terms of sharp regularity is a more recent development.

	Nevertheless, a definitive picture of  $C^1$ linearizability has remained out of reach, even in dimension two. In dimension one, at a hyperbolic fixed point, $C^1$ regularity alone
	does not guarantee local $C^1$ linearizability \cite{Ste57a}, whereas
	$C^{1,\alpha}$ regularity suffices for every $\alpha\in(0,1]$
	\cite{Kuc68}. The \textit{plane} is therefore the minimal setting where distinct spectral rates and nontrivial Jordan blocks give rise to new regularity thresholds.
	\subsection{Complete Characterization of Planar $C^1$ Linearizability}
	
	\subsubsection{Classification of Fixed Points in the Plane and Motivations}
	Let $O=(0,0)^\top$ denote the origin in $\mathbb{R}^2$, and endow $\mathbb{R}^2$ (and more generally $\mathbb{R}^n$) with the standard supremum norm $|\cdot|$ unless explicitly stated otherwise. Following Arnold's terminology for mappings \cite[Section~25.A, p.~192]{Arn88}, for a planar diffeomorphism $F:(\mathbb{R}^2,O)\to(\mathbb{R}^2,O)$ the eigenvalues of the derivative $\Lambda:=DF(O)\in\mathrm{GL}(2,\mathbb{R})$ fall into one of the following cases. Moreover, upon a suitable linear change of coordinates, $\Lambda$ can be brought into the corresponding real Jordan canonical form:
	\begin{enumerate}[label=(\Alph*), ref=(\Alph*), leftmargin=*]
		\setcounter{enumi}{15}
		\item\label{case-P} The eigenvalues belong to the
		\textit{Poincar\'e domain}, meaning that they both lie strictly
		inside or strictly outside the unit circle $S^1$ (the fixed point being a hyperbolic attractor or a hyperbolic repeller). Since a repeller becomes an
		attractor upon passing to the inverse mapping, it suffices to
		consider the contracting case. According to the real Jordan
		canonical form, this case splits into the following three mutually
		exclusive subcases (all parameters are understood to be real-valued):
		\begin{enumerate}[
			label=\textup{($\mathrm{P}_{\arabic*}$)},
			ref=\textup{($\mathrm{P}_{\arabic*}$)},
			leftmargin=*
			]
			\item\label{case-P-attractor}
			$\Lambda$ is diagonalizable over $\mathbb{R}$:
			\[
			\Lambda
			=
			\begin{pmatrix}
				\lambda_1 & 0\\
				0 & \lambda_2
			\end{pmatrix},
			\quad
			0<|\lambda_1|\leqslant|\lambda_2|<1,
			\]
			which corresponds to a hyperbolic attractor
			with a real diagonalizable linear part; see Figure~\ref{fig:fixed-point-classification}\subref{fig:phase-P1};
			
			\item\label{case-P-focus}
			$\Lambda$ has a pair of nonreal conjugate eigenvalues:
			\[
			\Lambda
			=
			\begin{pmatrix}
				\lambda\cos\theta & -\lambda\sin\theta\\
				\lambda\sin\theta & \lambda\cos\theta
			\end{pmatrix},
			\quad
			0<\lambda<1,\quad
			\theta\notin\pi\mathbb{Z},
			\]
			which corresponds to a hyperbolic attracting focus;  see Figure~\ref{fig:fixed-point-classification}\subref{fig:phase-P2};
			
			\item\label{case-P-degenerate-node}
			$\Lambda$ admits a repeated real eigenvalue and fails to be diagonalizable:
			\[
			\Lambda
			=
			\begin{pmatrix}
				\lambda & 1\\
				0 & \lambda
			\end{pmatrix},
			\quad
			0<|\lambda|<1,
			\]
			which corresponds to a hyperbolic degenerate node;  see Figure~\ref{fig:fixed-point-classification}\subref{fig:phase-P3}.
		\end{enumerate}
		
		\setcounter{enumi}{18}
		\item\label{case-S}
		The eigenvalues belong to the
		\textit{Siegel domain}, understood here as the complement of the
		Poincar\'e domain among all pairs of nonzero eigenvalues.
		Thus, either one eigenvalue lies strictly inside the unit circle
		$S^1$ and the other strictly outside $S^1$, or at least one
		eigenvalue lies on $S^1$.
		According to the real Jordan canonical form, this case splits
		into the following five mutually exclusive subcases
		(all parameters are understood to be real-valued):
		\begin{enumerate}[
			label=\textup{($\mathrm{S}_{\arabic*}$)},
			ref=\textup{($\mathrm{S}_{\arabic*}$)},
			leftmargin=*
			]
			\item\label{case-S-saddle}
			$\Lambda$ has one eigenvalue strictly inside $S^1$
			and the other strictly outside $S^1$:
			\[
			\Lambda
			=
			\begin{pmatrix}
				\lambda_1 & 0\\
				0 & \lambda_2
			\end{pmatrix},
			\quad
			0<|\lambda_1|<1<|\lambda_2|,
			\]
			which corresponds to a hyperbolic saddle point;  see  Figure~\ref{fig:fixed-point-classification}\subref{fig:phase-S1};
			
			\item\label{case-S-mixed}
			$\Lambda$ has exactly one eigenvalue on $S^1$:
			\[
			\Lambda
			=
			\begin{pmatrix}
				\lambda_1 & 0\\
				0 & \lambda_2
			\end{pmatrix},\quad 
			\lambda_1\in\{-1,1\},
			\quad
			\lambda_2\neq 0,
			\quad
			|\lambda_2|\neq 1,
			\]
			which corresponds to a non-hyperbolic fixed point; see  Figure~\ref{fig:fixed-point-classification}\subref{fig:phase-S2};
			
			\item\label{case-S-real-unit}
			$\Lambda$ is diagonalizable over $\mathbb{R}$
			and both eigenvalues lie on $S^1$:
			\[
			\Lambda
			=
			\begin{pmatrix}
				\lambda_1 & 0\\
				0 & \lambda_2
			\end{pmatrix},
			\quad
			(\lambda_1,\lambda_2)
			\in\{(1,1),(1,-1),(-1,-1)\},
			\]
			which corresponds to a non-hyperbolic fixed point; see  Figure~\ref{fig:fixed-point-classification}\subref{fig:phase-S3};
			
			\item\label{case-S-elliptic}
			$\Lambda$ has a pair of nonreal conjugate eigenvalues
			on $S^1$:
			\[
			\Lambda
			=
			\begin{pmatrix}
				\cos\theta & -\sin\theta\\
				\sin\theta & \cos\theta
			\end{pmatrix},
			\quad
			0<\theta<\pi,
			\]
			which corresponds to a non-hyperbolic fixed point; see  Figure~\ref{fig:fixed-point-classification}\subref{fig:phase-S4};
			
			\item\label{case-S-unit-jordan}
			$\Lambda$ admits a repeated real eigenvalue on $S^1$
			and fails to be diagonalizable:
			\[
			\Lambda
			=
			\begin{pmatrix}
				\lambda & 1\\
				0 & \lambda
			\end{pmatrix},
			\quad
			\lambda\in\{-1,1\},
			\]
			which corresponds to a non-hyperbolic fixed point; see  Figure~\ref{fig:fixed-point-classification}\subref{fig:phase-S5}.
		\end{enumerate}
	\end{enumerate}

	\textit{The main purpose of the present paper is to establish a complete
		characterization of the regularity assumptions guaranteeing local $C^1$
		linearizability for all these cases, thereby providing a definitive
		resolution of this longstanding problem in dimension two.}
	In dimensions one and two, our results settle the questions raised
	in~\cite{GHR03,New17,LZZ20}. Under the corresponding $C^1$ linearizability criteria,
\textit{we further determine the optimal uniform moduli of continuity of the conjugacy derivatives in every hyperbolic case.}
See Tables~\ref{tab:planar-linearization-summary},
\ref{tab:one-dimensional-linearization-summary},
and~\ref{tab:linearization-sharp-regularity} in Section~\ref{seccon}
for an overview of the complete linearizability criteria and the optimal conjugacy regularity.

	Here and throughout, this characterization is understood in a uniform sense with respect to a prescribed linear part. For each hyperbolic real Jordan form, the corresponding integrability condition guarantees local $C^1$ linearizability for every mapping in the regularity class under consideration; conversely, if the condition fails, there exists a counterexample with that same linear part. In other words, the linear part---including the eigenvalues (their signs or rotation angles) and the Jordan block structure---is fixed prior to the construction of the counterexample. In the non-hyperbolic case, a real-analytic counterexample exists for every prescribed linear part.

	\begin{figure}[p]
		\centering
	
		\begin{adjustbox}{max totalsize={\textwidth}{0.95\textheight}}
			\begin{minipage}{\textwidth}
				\centering
				\setlength{\parindent}{0pt}
				\setlength{\parskip}{0pt}
				\captionsetup{font=small,skip=5pt}
				\captionsetup[subfigure]{
					font=footnotesize,
					labelfont=bf,
					justification=centering,
					singlelinecheck=false,
					skip=2pt
				}

				\newcommand{\FixedPointPanel}[4]{
					\begin{subfigure}[t]{0.315\linewidth}
						\centering
						\includegraphics[
						width=\linewidth,
						height=0.25\textheight,
						keepaspectratio
						]{#1}
						\caption{Case~\ref{#3}: #2}
						\label{#4}
					\end{subfigure}
				}

				\IfFileExists{P1_phase_portrait.png}
				{\def\PonePortraitFile{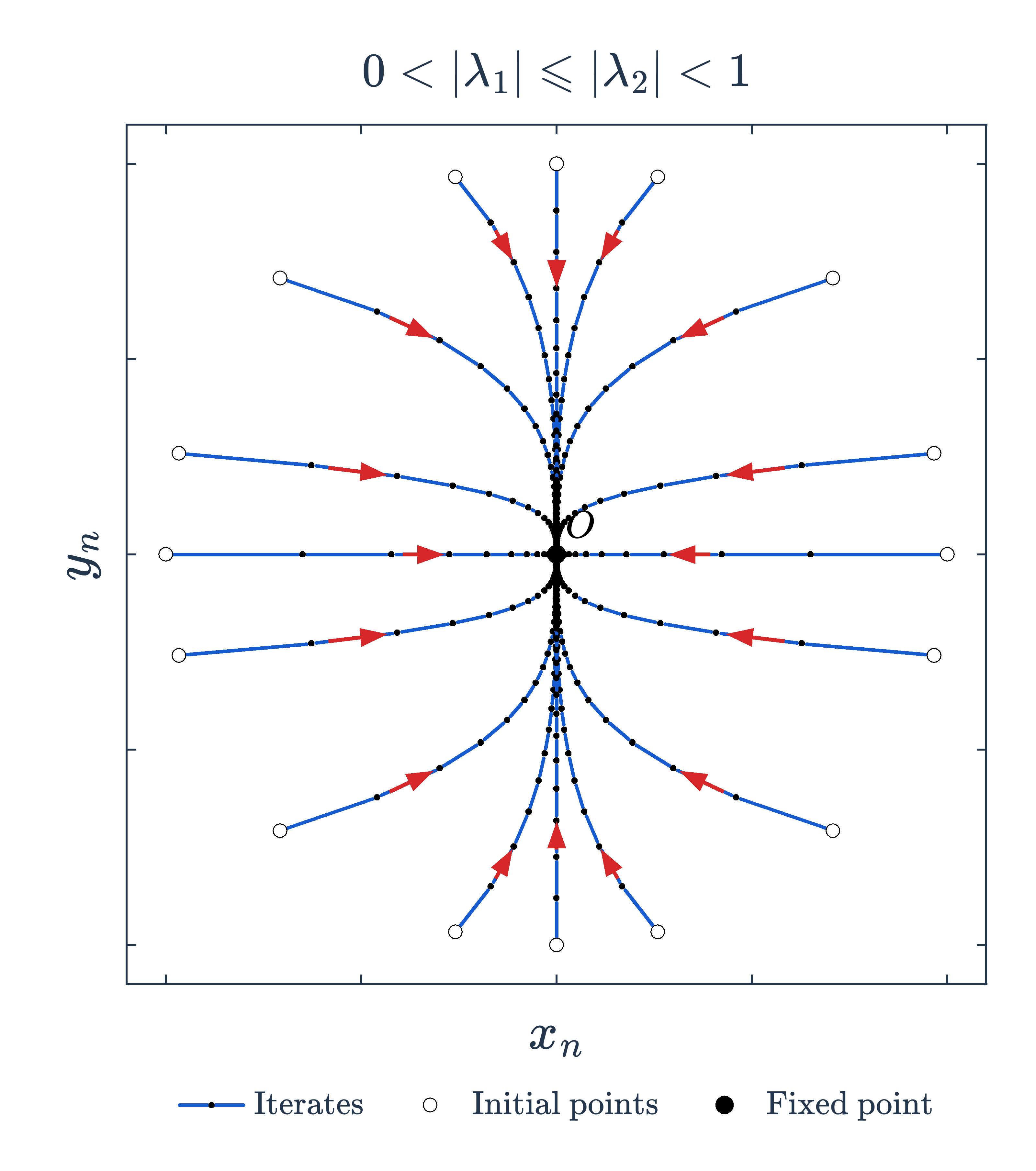}}
				{\def\PonePortraitFile{P1_attracting_node_v2.png}}
				
				\FixedPointPanel{\PonePortraitFile}
				{Hyperbolic attractor with a real diagonalizable linear part}
				{case-P-attractor}{fig:phase-P1}
				\hfill
				\FixedPointPanel{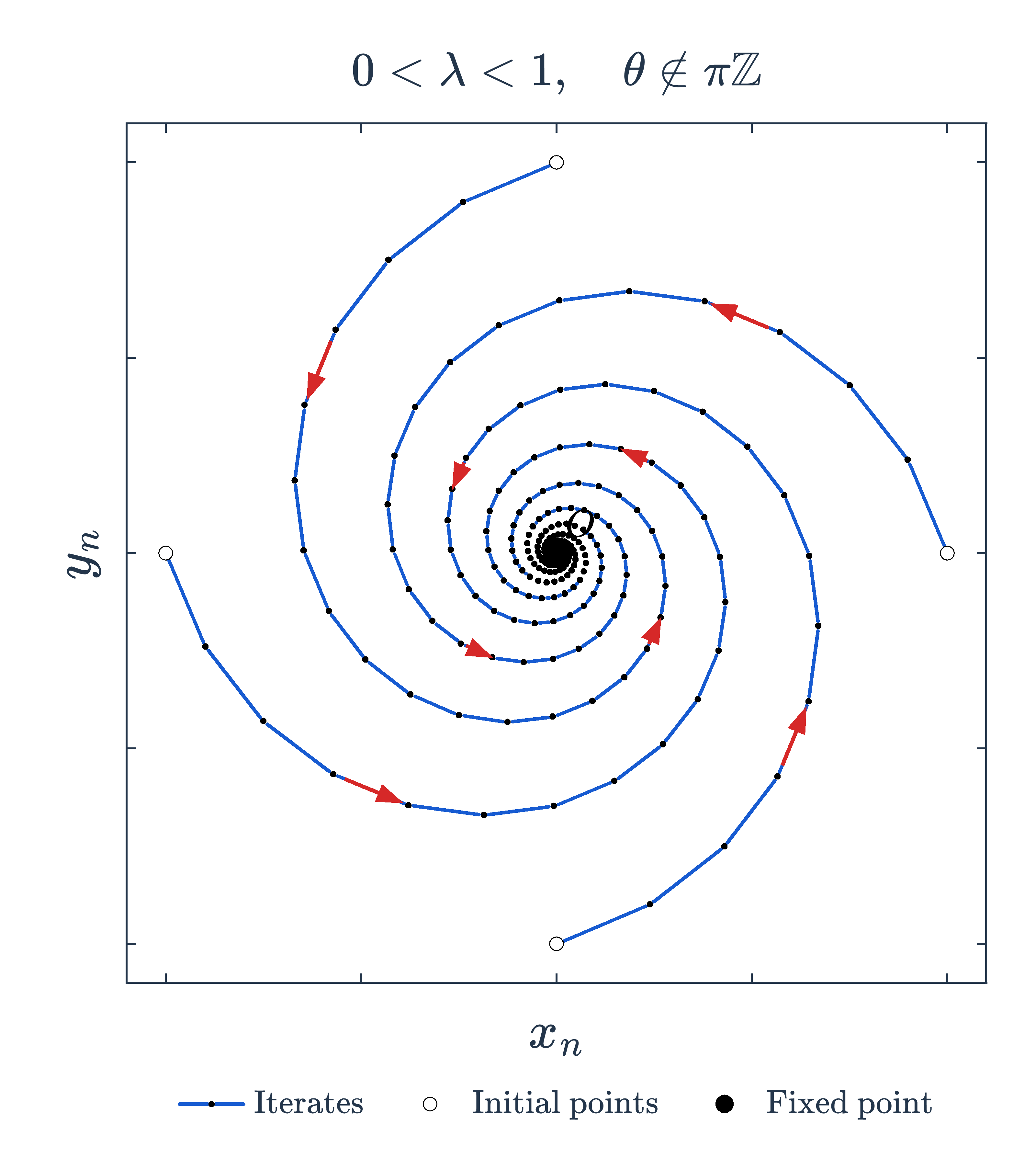}
				{Hyperbolic attracting focus}
				{case-P-focus}{fig:phase-P2}
				\hfill
				\FixedPointPanel{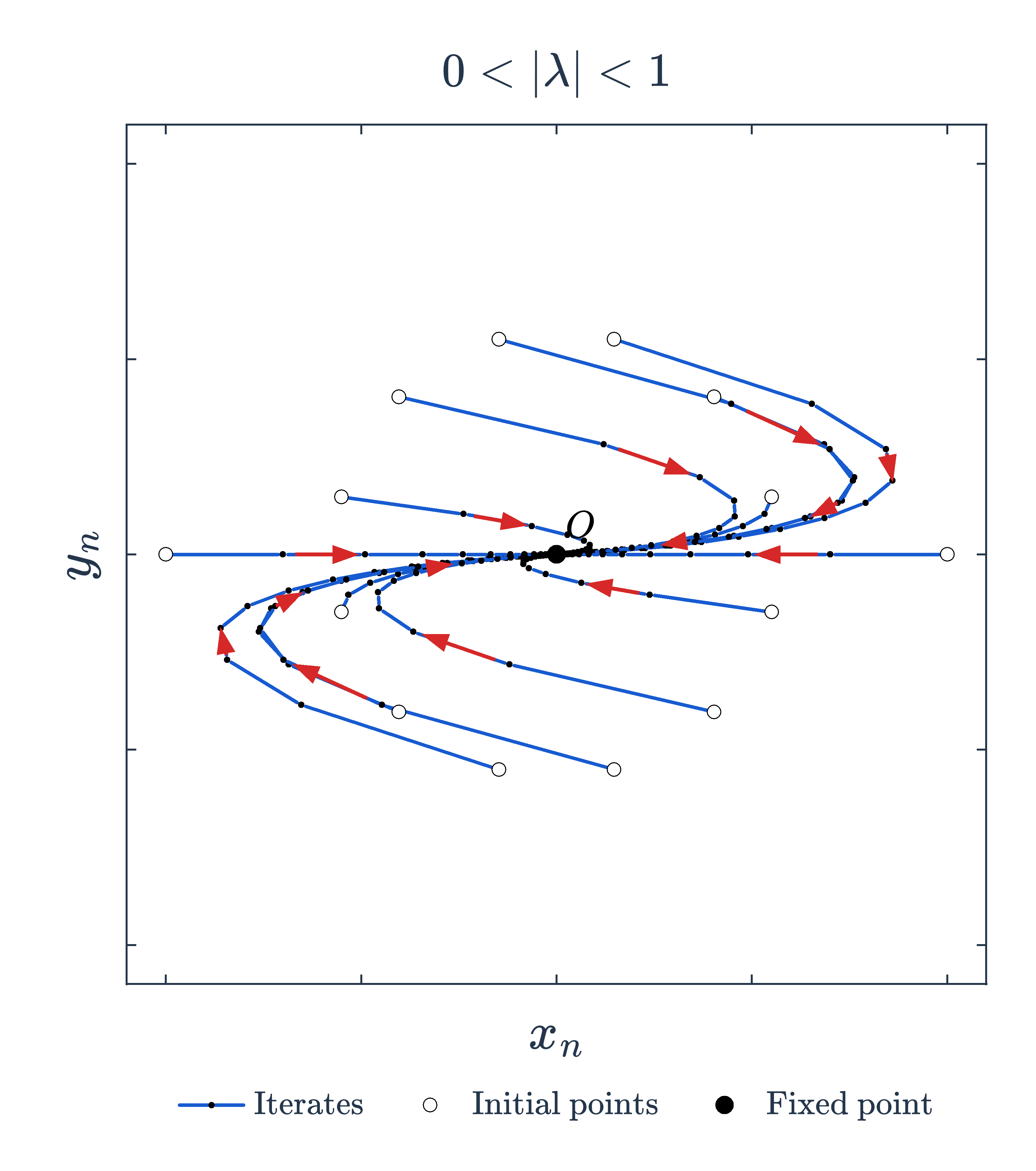}
				{Hyperbolic degenerate node}
				{case-P-degenerate-node}{fig:phase-P3}
				
				\par\vspace{4pt}
				
				\FixedPointPanel{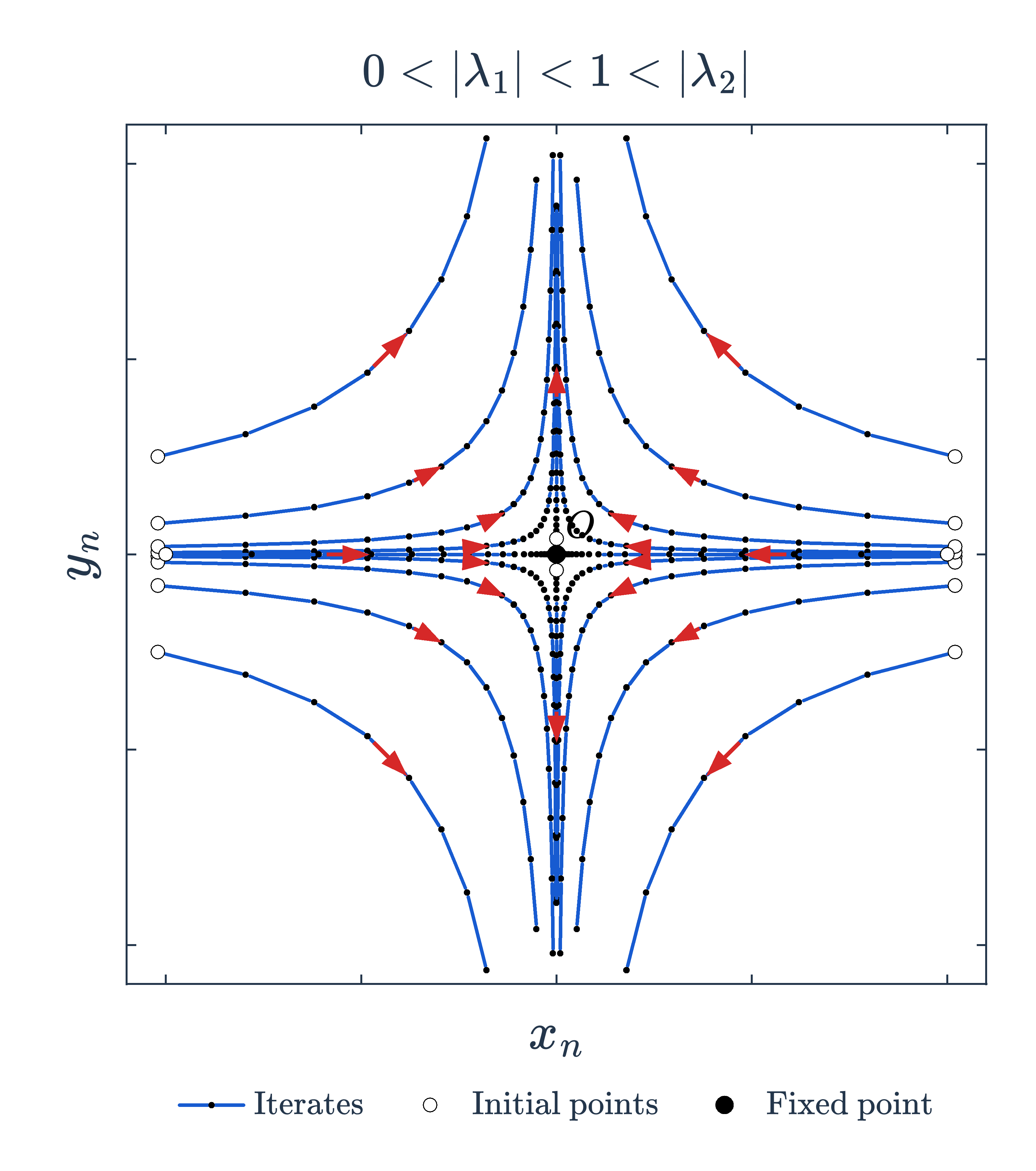}
				{Hyperbolic saddle point}
				{case-S-saddle}{fig:phase-S1}
				\hfill
				\FixedPointPanel{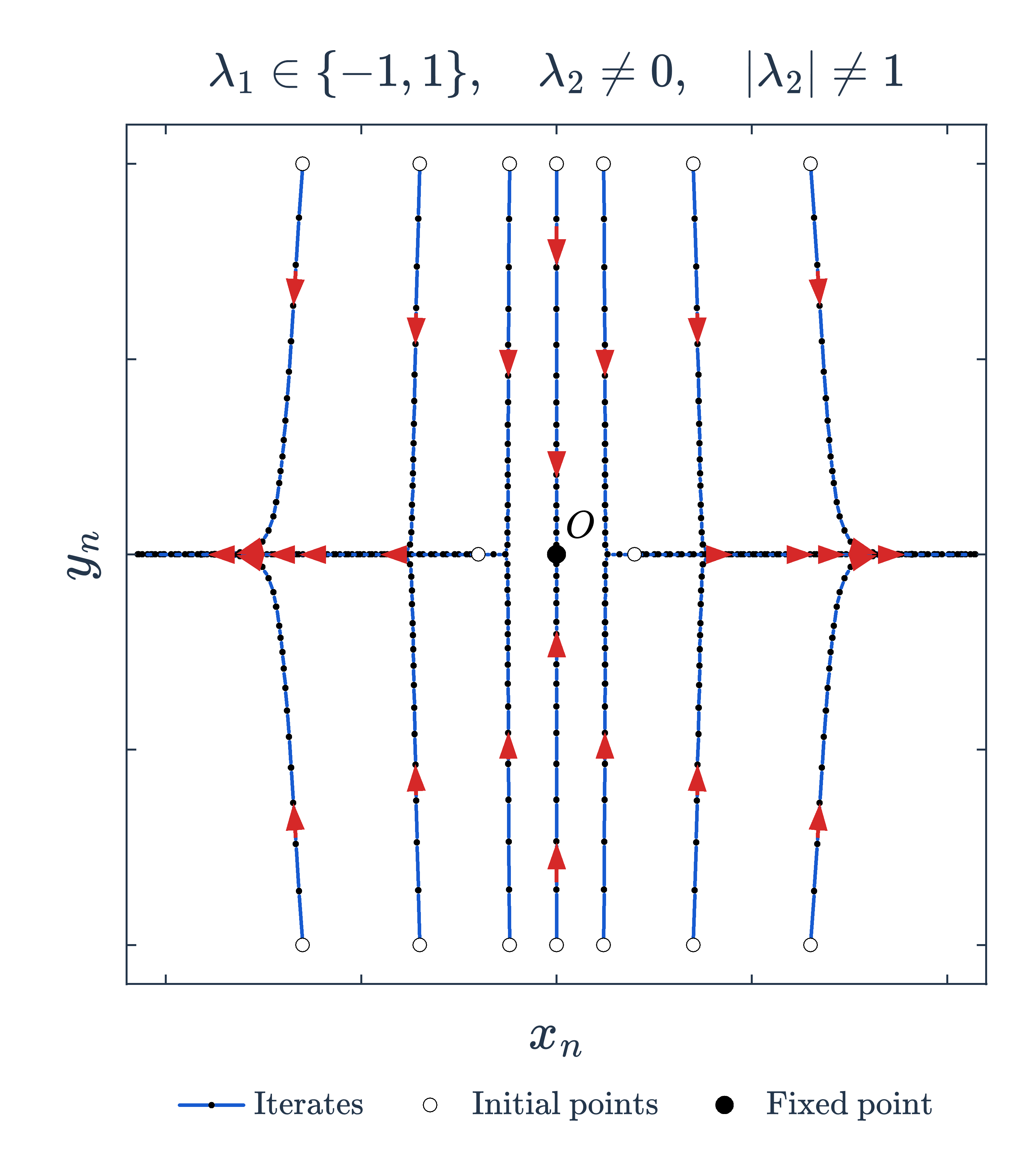}
				{Non-hyperbolic fixed point}
				{case-S-mixed}{fig:phase-S2}
				\hfill
				\FixedPointPanel{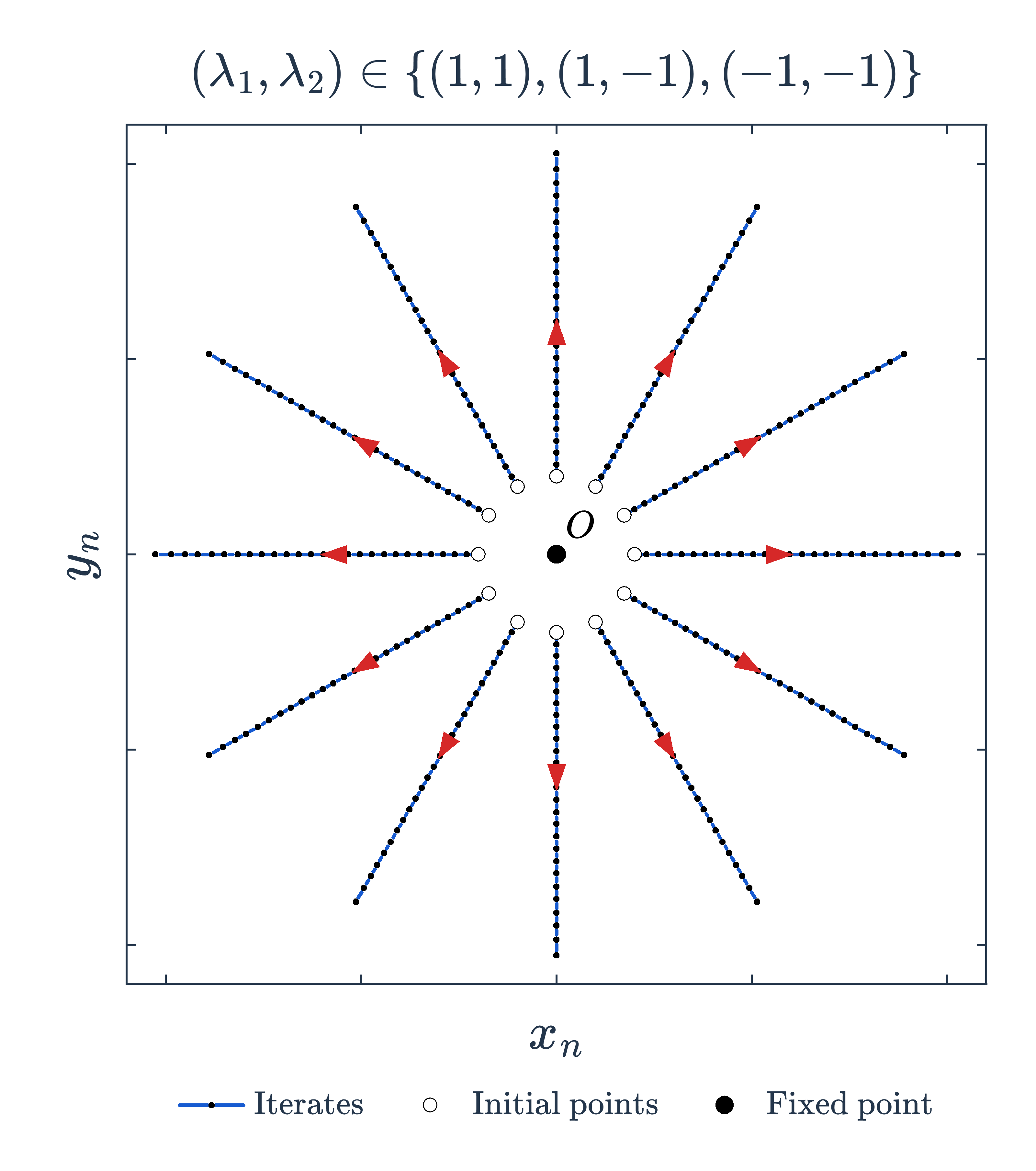}
				{Non-hyperbolic fixed point}
				{case-S-real-unit}{fig:phase-S3}
				
				\par\vspace{4pt}
				
				\FixedPointPanel{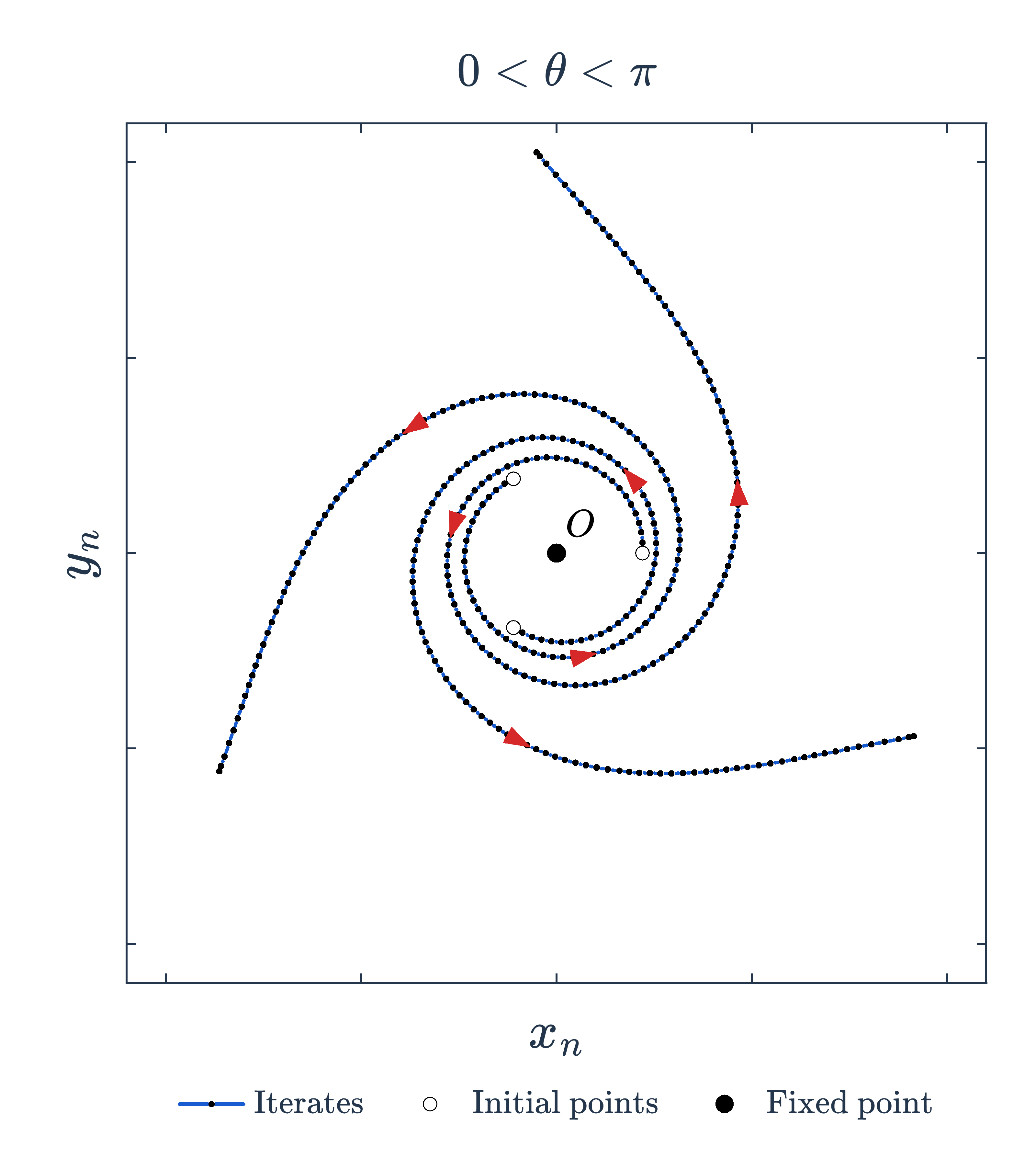}
				{Non-hyperbolic fixed point}
				{case-S-elliptic}{fig:phase-S4}
				\hspace{0.0275\linewidth}
				\FixedPointPanel{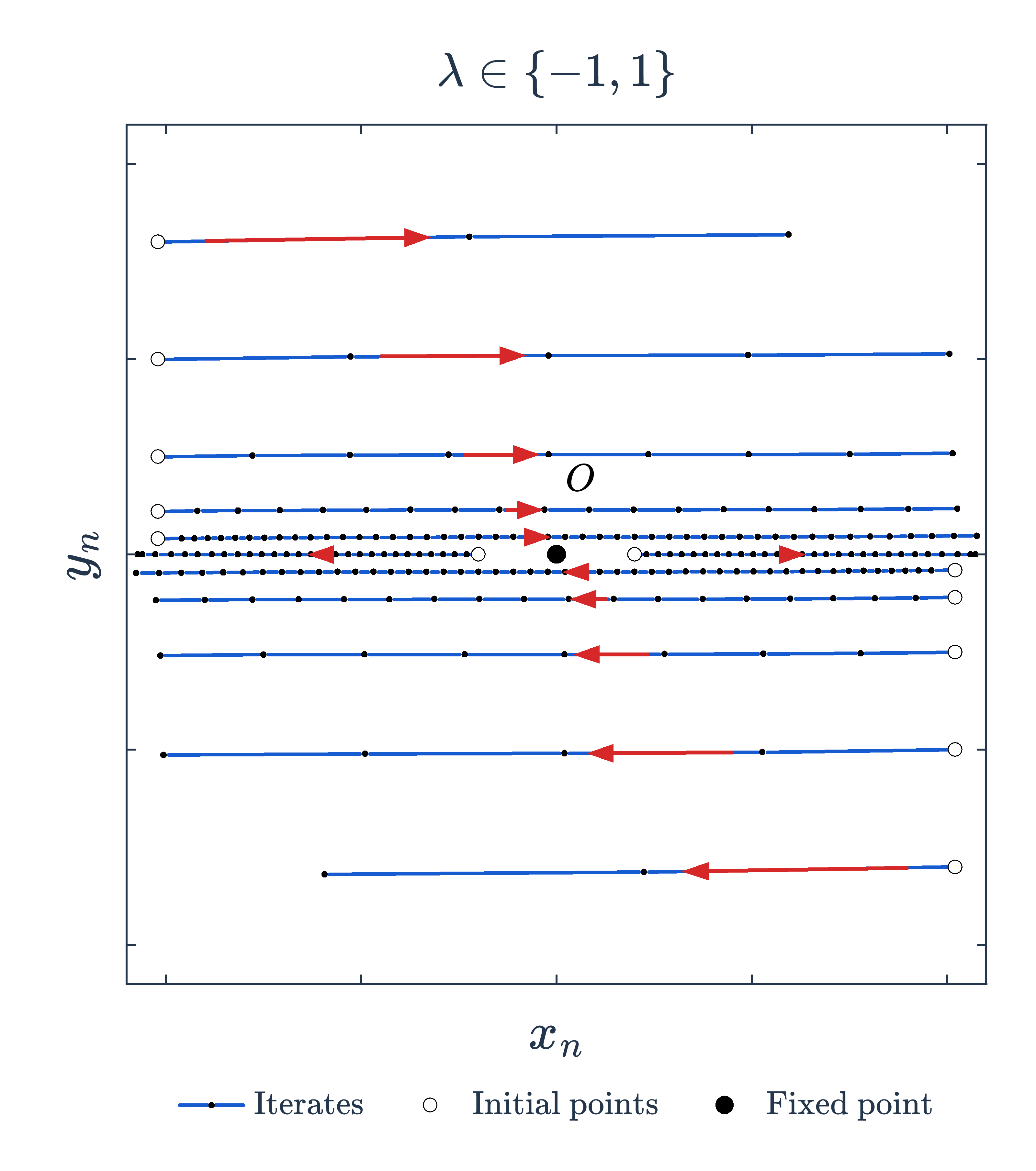}
				{Non-hyperbolic fixed point}
				{case-S-unit-jordan}{fig:phase-S5}
				
				\caption{Local phase portraits for the eight canonical cases of planar diffeomorphisms. Each panel displays discrete orbits of a specific mapping; in the non-hyperbolic setting, only a single configuration is shown, though other geometric profiles may also occur.}
				\label{fig:fixed-point-classification}
			\end{minipage}
		\end{adjustbox}
	\end{figure}

	\subsubsection{Case~\ref{case-P-attractor}: Hyperbolic Attractors with Real Diagonalizable Linear Parts}
	
	We first consider Case~\ref{case-P-attractor}. 
	The local $C^1$ linearization problem asks whether there exists a local $C^1$ diffeomorphism $\Phi:(\mathbb{R}^2,O)\to(\mathbb{R}^2,O)$ tangent to the identity such that
	\[
	\Phi\circ F=\Lambda\circ\Phi,
	\quad
	\Phi(O)=O,
	\quad
	D\Phi(O)=\operatorname{Id}.
	\]
	The normalization $D\Phi(O)=\operatorname{Id}$ entails no loss of generality. If $H\circ F=\Lambda\circ H$ and $A=DH(O)$, differentiation at $O$ gives $A\Lambda=\Lambda A$; hence $\Phi=A^{-1}H$ is a conjugacy to the same linear part and satisfies $D\Phi(O)=\operatorname{Id}$.
	The critical spectral exponent governing this problem is given by
	\[
	\alpha_0
	:=1-\frac{\log|\lambda_2|}{\log|\lambda_1|}
	\in[0,1).
	\]
	
	Hartman proved that $C^{1,1}$ contractions are locally $C^1$ linearizable \cite{Har60}. An earlier spectral-gap criterion, when specialized to the planar setting above, yields a $C^1$ linearization for $C^{1,\alpha}$ mappings under the stronger condition
	\[
	\alpha>\alpha_1,
	\quad
	\alpha_1:=\frac{\log|\lambda_1|}{\log|\lambda_2|}-1;
	\]
	see Chaperon \cite{Cha02} for details. Zhang--Zhang \cite{ZZ11} lowered this requirement to the critical exponent $\alpha_0$: every planar $C^{1,\alpha}$ contraction in Case \ref{case-P-attractor} is locally $C^1$ linearizable whenever $\alpha>\alpha_0$. They also obtained sharp H\"older estimates for the derivatives of the conjugacy and its inverse. Subsequently, Zhang--Zhang \cite{ZZ14} established that this regularity threshold is optimal: whenever $|\lambda_1|<|\lambda_2|$ (so that $\alpha_0>0$), for every $0<\alpha\leqslant\alpha_0$ there exists a planar $C^{1,\alpha}$ contraction with the prescribed eigenvalues that fails to be locally $C^1$ linearizable.

	The H\"older scale, however, does \textit{not} exhaust the regularity of $C^1$ mappings. On a compact neighborhood of the origin, the derivative of any $C^1$ mapping admits a modulus of continuity, yet it need not satisfy a H\"older condition of any positive exponent; see Section~\ref{sec:beyond-holder-comparison} for a detailed discussion. Consequently, a framework restricted to H\"older exponents cannot yield an optimal criterion for general $C^1$ systems. This motivates formulating the problem in terms of general moduli of continuity. Let $\omega:[0,\infty)\to[0,\infty)$ be a continuous, strictly increasing, concave function satisfying $\omega(0)=0$. Since our analysis is local, it suffices that $\omega$ is defined on a right-neighborhood of $0$, as it can be extended to $[0,\infty)$ preserving monotonicity and concavity. We write $F\in C_{1,\omega}$ if, on some neighborhood of the origin,
	\[
	|DF(x)-DF(y)|\leqslant L\omega(|x-y|)
	\]
	for some constant $L>0$\footnote{For our local purposes, this estimate suffices, and there is no need to introduce a full Banach norm on $C_{1,\omega}$.}. Here, the subscript in $C_{1,\omega}$ emphasizes the general modulus of continuity, distinguishing it from the standard H\"older class $C^{1,\alpha}$. This concavity convention is natural for local estimates: on a compact convex neighborhood, the intrinsic modulus $\omega_{DF}(t):=\sup_{|x-y|\leqslant t}|DF(x)-DF(y)|$ is subadditive and, unless identically zero, has an equivalent concave majorant by Stechkin's lemma (see \cite[p.~78]{Efi61} and \cite[p.~177]{DS08}). Throughout the paper, the prescribed modulus $\omega$ is understood to satisfy this convention. Because general moduli of continuity lack the scaling symmetry of the H\"older functions $t^\alpha$, smooth linearization in this setting presents substantial analytical challenges and has remained largely unexplored.
	
	Following the work of Zhang--Lu--Zhang \cite{ZLZ17} on the Hartman--Grobman theorem with differentiability at the fixed point\footnote{Here, differentiability is required only at the fixed point.} under sharp H\"older conditions in Banach spaces, Tong--Li \cite{TL24} established an analogue for general moduli of continuity. Returning to planar contractions, Tong--Xu--Li \cite{TXL24} proved local $C^1$ linearizability under the \textbf{polynomially weighted Dini condition}
	\begin{equation}\tag{$\mathrm{WD}_{\rm P}$}\label{kjh}
		\int_0^1 \frac{\omega(t)}{t^{\alpha_0+1}} dt <+ \infty.
	\end{equation}
	In the isotropic case $0<|\lambda_1|=|\lambda_2|<1$, where $\alpha_0=0$, the weighted Dini condition \eqref{kjh} reduces to the classical \textbf{Dini condition} (see \cite{FJ01a,FJ01b,Ye13,JY22,Bou24,TL25,DGT26} for applications in dynamical systems)
	\begin{equation}\tag{D}\label{dinicd}
		\int_0^1 \frac{\omega(t)}{t} dt <+ \infty.
	\end{equation}
	Moreover, the Dini condition \eqref{dinicd} is known to be optimal in this context\footnote{Although the counterexample in \cite{TXL24} is formulated for the contracting case \(0<\lambda_1=\lambda_2<1\), the same invariant-axis construction applies when \(0<|\lambda_1|=|\lambda_2|<1\) and \(\lambda_1\lambda_2<0\), after reversing the sign of the appropriate component and applying the argument to the absolute values of the iterates on the invariant axis.}. The obstruction in that counterexample originates from the divergence of an infinite product of normalized derivatives along an invariant coordinate axis, closely mirroring the one-dimensional accumulation mechanism constructed via a logarithmic modulus of continuity by Sternberg \cite{Ste57a}; see also Kuczma \cite[p.~139]{Kuc68}. When $\omega(t)=t^\alpha$, the weighted Dini condition \eqref{kjh} coincides with $\alpha>\alpha_0$, thereby recovering the optimal H\"older threshold. For general moduli of continuity, however, \eqref{kjh} provides a substantially finer criterion that goes beyond the H\"older setting.

	Nevertheless, despite the sufficiency established in \cite{TXL24}, the decisive question of whether the weighted Dini condition \eqref{kjh} is \textit{necessary} in the presence of a strict spectral gap 
	$ 0 < |\lambda_1| < |\lambda_2| < 1 $
	has remained completely open. In this anisotropic regime, the disparity between the two contracting rates generates an intrinsic cross-coupling between the distinct directions, so the invariant-axis accumulation mechanism alone does not establish the weighted threshold.
	
	Our first main result overcomes this difficulty by developing a genuinely two-dimensional counterexample, completely settling the optimality of the weighted Dini condition \eqref{kjh}.

	
	
	\begin{mainthm}[Anisotropic counterexample]\label{TH1}
		Let $0<|\lambda_1|<|\lambda_2|<1$, and let $\omega$ be a modulus of continuity for which the weighted Dini condition \eqref{kjh} fails. Then there exists a planar $C_{1,\omega}$ diffeomorphism $F$ in Case \ref{case-P-attractor} that is not locally $C^1$ linearizable.
	\end{mainthm}

	\begin{customrem}{A}\label{rem:A}
		Unlike the isotropic setting where the obstruction to linearization concentrates along an invariant coordinate axis, the counterexample in Theorem~\ref{TH1} requires a genuinely two-dimensional mechanism that couples the two distinct contracting rates; see Section~\ref{SEC3} for the detailed construction.
	\end{customrem}

	Together with \cite[Main Theorem 1 and Example 3.1]{TXL24}, Theorem~\ref{TH1} shows that the weighted Dini condition \eqref{kjh} is optimal as a uniform regularity criterion for planar contractions. This yields our second main result, which provides a complete characterization of $C^1$ linearizability in Case~\ref{case-P-attractor}\footnote{Recall that the derivative of any $C^1$ contraction admits a modulus of continuity near the origin.}.
	
	\begin{mainthm}[Hyperbolic attractors with real diagonalizable linear parts]\label{TH2}
		Let $\omega$ be a modulus of continuity. Then the following statements hold for any prescribed $0<|\lambda_1|\leqslant|\lambda_2|<1$:
		\begin{enumerate}[label=(\Roman*)]
			\item \label{TH2-I} If the weighted Dini condition \eqref{kjh} holds, then every planar $C_{1,\omega}$ diffeomorphism $F$ in Case \ref{case-P-attractor} is locally $C^1$ linearizable.
			\item If the weighted Dini condition \eqref{kjh} fails, then there exists a planar $C_{1,\omega}$ diffeomorphism $F$ in Case \ref{case-P-attractor} that is not locally $C^1$ linearizable.
		\end{enumerate}
		Equivalently, the class $C_{1,\omega}$ of planar hyperbolic diffeomorphisms in Case \ref{case-P-attractor} is locally $C^1$ linearizable if and only if the weighted Dini condition \eqref{kjh} holds.
	\end{mainthm}

Beyond the existence of a $C^1$ conjugacy guaranteed by Theorem~\ref{TH2}--\ref{TH2-I}, we determine the optimal modulus of continuity of its derivative. For the anisotropic case, the estimates in \cite[Main Theorem 2]{TXL24} admit a simpler form with two branches, distinguished by an \textit{unexpected} second weighted integral. 

\begin{mainthm}[Optimal higher regularity]\label{TH3}
Let $\omega$ be a modulus of continuity satisfying the weighted Dini condition~\eqref{kjh}.
\begin{enumerate}[label=(\Roman*)]
\item\label{TH3-1}
Assume that $0<|\lambda_1|<|\lambda_2|<1$, and put
\[
\alpha_1:=\frac{\alpha_0}{1-\alpha_0}
=\frac{\log|\lambda_1|}{\log|\lambda_2|}-1,
\quad
I_{\alpha_1}(\omega):=\int_0^1\frac{\omega(s)}{s^{1+\alpha_1}}ds.
\]
For sufficiently small $t>0$, define
\begin{equation}\label{eq:anisotropic-optimal-modulus}
	\omega_{\mathrm{A}}(t):=
	\begin{cases}
		\displaystyle t^{\alpha_1}\int_0^t\frac{\omega(s)}{s^{1+\alpha_1}}ds,
		&  I_{\alpha_1}(\omega)<+\infty,\\[3mm]
		\displaystyle\int_0^{t^{\frac{1}{\alpha_0}}}\frac{\omega(s)}{s^{1+\alpha_0}}ds
		+t^{\alpha_1}\int_{t^{\frac{1}{\alpha_0}}}^1\frac{\omega(s)}{s^{1+\alpha_1}}ds,
		&  I_{\alpha_1}(\omega)=+\infty.
	\end{cases}
\end{equation}
Then every planar $C_{1,\omega}$ diffeomorphism $F$ in Case~\ref{case-P-attractor} admits a local linearizing diffeomorphism $\Phi$ with $\Phi,\Phi^{-1}\in C_{1,\omega_{\rm A}}$.
This estimate is optimal as a uniform guarantee for each prescribed pair $(\lambda_1,\lambda_2)$: if $\eta$ is a modulus of continuity such that every
such $F$ admits a local $C_{1,\eta}$ linearizing diffeomorphism, then
\[
\omega_{\mathrm A}(t)=\mathcal O(\eta(t))
\quad\text{as }t\to0^+.
\]

\item\label{TH3-2}
Assume that $0<|\lambda_1|=|\lambda_2|<1$, and define
\begin{equation}\label{omegaA}
\omega_{\rm B}(t):=\int_0^t\frac{\omega(s)}s ds.
\end{equation}
Then every planar $C_{1,\omega}$ diffeomorphism $F$ in Case~\ref{case-P-attractor} admits a local linearizing diffeomorphism $\Phi$ with $\Phi,\Phi^{-1}\in C_{1,\omega_{\rm B}}$. This estimate is optimal for each prescribed pair $(\lambda_1,\lambda_2)$: there exists such a diffeomorphism $F$ for which every normalized local linearizing diffeomorphism $\Phi\in C_{1,\eta}$ satisfies
\[
\omega_{\rm B}(t)=\mathcal O(\eta(t))\quad\text{as }t\to0^+.
\]
\end{enumerate}
\end{mainthm}

For $\omega(t)=t^\alpha$ with $\alpha_0<\alpha\leqslant1$, formula~\eqref{eq:anisotropic-optimal-modulus} yields the exponents summarized in Table~\ref{tab:holder-exponents}. The resulting exponents coincide with the sharp H\"older estimates in \cite[Theorem~2 and Section~4]{ZZ11}; see also \cite[Remark~5]{ZZ14} and \cite[Example~3.2]{TXL24}. At the critical value $\alpha=\alpha_1\leqslant1$, the optimal modulus of continuity of the derivative is of order $t^\alpha\log(e/t)$: every exponent below $\alpha$ is admissible, whereas the endpoint strictly requires the logarithmic factor.

\begin{table}[htbp]
\centering
\begingroup
\small
\setlength{\tabcolsep}{5pt}
\setlength{\extrarowheight}{0pt}
\renewcommand{\arraystretch}{1.2}
\renewcommand{\tabularxcolumn}[1]{m{#1}}
\begin{tabularx}{\linewidth}{|
>{\centering\arraybackslash}m{0.30\linewidth}|
>{\centering\arraybackslash}m{0.30\linewidth}|
>{\centering\arraybackslash}X|}
\hline
Input H\"older exponent & Optimal derivative modulus & Guaranteed H\"older regularity \\
\hline
$\alpha_1<\alpha\leqslant1$ & $t^\alpha$ & $C^{1,\alpha}$ \\
\hline
$\alpha=\alpha_1\leqslant1$ & $t^\alpha\log(e/t)$ & $C^{1,\beta}$ for every $0<\beta<\alpha$ \\
\hline
$\alpha_0<\alpha<\alpha_1$, $\alpha\leqslant1$
& $t^{\alpha/\alpha_0-1}$ & $C^{1,\alpha/\alpha_0-1}$ \\
\hline
\end{tabularx}
\caption{Optimal output regularity for anisotropic contractions with input H\"older regularity.}
\label{tab:holder-exponents}
\endgroup
\end{table}

	\subsubsection{Case~\ref{case-P-focus}: Hyperbolic Attracting Foci}

	We next turn to Case~\ref{case-P-focus}. Since the eigenvalues have equal modulus, the classical spectral criterion of Chaperon \cite[Corollary~1.3.3]{Cha02}  yields local $C^{1,\alpha}$ linearizability for all $\alpha\in(0,1]$ (see also Newhouse \cite[Theorem~3.1 and Remark~3.2]{New17}
	for $\alpha\in(0,1)$), leaving no positive lower threshold on the H\"older exponent. However, the H\"older framework is too coarse to detect the critical regularity in terms of general moduli of continuity. Since there is no separation between the eigenvalue moduli, the critical criterion is naturally expected to be free of spectral weights, pointing directly to the classical Dini condition \eqref{dinicd}.
	
	Our fourth main result confirms that \eqref{dinicd} is indeed the definitive threshold:

	\begin{mainthm}[Hyperbolic attracting foci]\label{thm:focus}
		Let $\omega$ be a modulus of continuity. Then the following statements hold for any prescribed $0<\lambda<1$ and $\theta\notin\pi\mathbb{Z}$:
		\begin{enumerate}[label=(\Roman*)]
			\item \label{THfocus-1} If the Dini condition~\eqref{dinicd} holds, then every
			planar $C_{1,\omega}$ diffeomorphism $F$ in Case~\ref{case-P-focus} is locally $C^1$ linearizable.
			\item   \label{THfocus-2}  If the Dini condition~\eqref{dinicd} fails, then there
			exists a planar $C_{1,\omega}$ diffeomorphism $F$ in Case~\ref{case-P-focus} that is not locally
			$C^1$ linearizable.
		\end{enumerate}
		Equivalently, the class $C_{1,\omega}$ of planar hyperbolic diffeomorphisms in Case~\ref{case-P-focus} is locally $C^1$ linearizable if and only if the Dini
		condition~\eqref{dinicd} holds.
		
		Moreover, in Item \ref{THfocus-1}, the linearizing diffeomorphism can be chosen of class $C_{1,\omega_{\rm B}}$, where $\omega_{\rm B}(t)=\int_0^t\frac{\omega(s)}s ds$ is the modulus in \eqref{omegaA}. This estimate is optimal for each prescribed $\lambda$ and $\theta$: there exists such a diffeomorphism $F$ for which every normalized local linearizing diffeomorphism $\Phi\in C_{1,\eta}$ satisfies 
		\[
		\omega_{\rm B}(t)=\mathcal{O}(\eta(t))\quad\text{as } t\to0^+.
		\]
	\end{mainthm}

	This provides a complete characterization of local $C^1$ linearizability in Case~\ref{case-P-focus}.

	\subsubsection{Case~\ref{case-P-degenerate-node}: Hyperbolic Attracting Degenerate Nodes}
	In Case~\ref{case-P-degenerate-node}, the eigenvalues still have equal modulus. Since Chaperon's spectral criterion \cite{Cha02} depends solely on the spectrum and ignores the Jordan block structure, it again yields local $C^1$ linearizability under $C^{1,\alpha}$ regularity for all $0<\alpha\leqslant 1$ (see also Hartman \cite{Har60} for the $C^{1,1}$ case, and Newhouse \cite{New17}). On the H\"older scale, this degenerate node is therefore completely indistinguishable from the hyperbolic focus treated above.
	
	The modulus-of-continuity framework reveals, however, that this coincidence is an artifact of the H\"older scale: \textit{a nontrivial Jordan block requires an additional squared-logarithmic weight.}
	
	Accordingly, we introduce the following \textbf{squared-logarithmically weighted Dini condition}:
	\begin{equation}\label{eq:log2dini}\tag{$\mathrm{WD}_{\rm L}$}
		\int_0^1\frac{\omega(t)}t
		\left(\log\frac et\right)^2dt<+\infty.
	\end{equation}
	This condition is strictly stronger than the classical Dini condition, although both are satisfied by every H\"older modulus $\omega(t)=t^\alpha$ with $0<\alpha\leqslant1$. For comparison, if $\omega(t)=\left(\log(e/t)\right)^{-p}$ near $0$ with $p>0$, the classical Dini condition holds if and only if $p>1$, whereas~\eqref{eq:log2dini} holds if and only if $p>3$.
	
	In terms of condition~\eqref{eq:log2dini}, our fifth main result provides a complete characterization of local $C^1$ linearizability in Case~\ref{case-P-degenerate-node}:
	
	\begin{mainthm}[Hyperbolic attracting degenerate nodes]\label{thm:jordan}
		Let $\omega$ be a modulus of continuity. Then the following statements hold for any prescribed $0<|\lambda|<1$:
		\begin{enumerate}[label=\textup{(\Roman*)}]
			\item\label{THjordan-1}
			If the weighted Dini condition~\eqref{eq:log2dini} holds, then every planar $C_{1,\omega}$ diffeomorphism $F$ in Case~\ref{case-P-degenerate-node} is locally $C^1$ linearizable.
			
			\item\label{THjordan-2}
			If the weighted Dini condition~\eqref{eq:log2dini} fails, then there exists a planar $C_{1,\omega}$ diffeomorphism $F$ in Case~\ref{case-P-degenerate-node} that is not locally $C^1$ linearizable.
		\end{enumerate}
		Equivalently, the class $C_{1,\omega}$ of planar hyperbolic diffeomorphisms in Case~\ref{case-P-degenerate-node} is locally $C^1$ linearizable if and only if the weighted Dini condition~\eqref{eq:log2dini} holds.

		Moreover, in Item~\ref{THjordan-1}, the linearizing diffeomorphism $\Phi$ and its inverse can be chosen of class $C_{1,\omega_{\mathrm C}}$, where, for sufficiently small $t>0$,
		\begin{equation}\label{eq:jordan-regularity-modulus}
			\omega_{\mathrm C}(t):=
			\int_0^t\frac{\omega(s)}s\left(\log\frac{et}s\right)^2ds.
		\end{equation}
		This estimate is optimal for each prescribed $\lambda$: there exists such a diffeomorphism $F$ for which every normalized local linearizing diffeomorphism $\Phi\in C_{1,\eta}$ satisfies
		\[
		\omega_{\mathrm C}(t)=\mathcal O(\eta(t))
		\quad\text{as }t\to0^+.
		\]
	\end{mainthm}
	\begin{remark}
We remark that $ \omega_{\mathrm C}(t) $
is not equivalent to $\int_0^t \frac{\omega(s)}s\left(\log\frac{e}s\right)^2ds$ (and thus the inner $\log t$ factor cannot be absorbed), which can be verified simply by testing the H\"older case.
	\end{remark}

	The appearance of the squared logarithm in~\eqref{eq:log2dini} is perhaps \textit{unexpected}. It reflects the dynamical interaction between the superdiagonal entry $1$ of the unperturbed Jordan block and the lower-left entry of a general derivative perturbation (see Section~\ref{PTJORDAN} for details). Thus, the interplay between both off-diagonal positions dictates this optimal regularity threshold. In higher dimensions, one naturally expects the optimal criterion to feature a logarithmic factor whose exponent depends on the dimension (or the maximal size of the Jordan blocks). Determining this dependence presents an intriguing problem worthy of further investigation.

	\subsubsection{Case~\ref{case-S-saddle}: Hyperbolic Saddle Points}
	We now turn to Case~\ref{case-S-saddle}, corresponding to hyperbolic saddle points. Hartman \cite[Theorem~(IV)]{Har60} established local $C^1$ linearizability for planar $C^{1,1}$ saddle points. For planar $C^r$ diffeomorphisms with integer $r\geqslant2$, Stowe \cite[Theorems~2 and~3]{Sto86} investigated the regularity of linearizing transformations and normal forms. In the $C^{1,1}$ case, Zhang--Zhang--Jarczyk \cite[Corollary~2]{ZZJ14} obtained $C^{1,\beta}$ linearization with estimates arbitrarily close to the optimal exponent. In the H\"older category, Zhang--Zhang \cite[Theorems~3 and~4]{ZZ14} established $C^{1,\beta}$ linearization with sharp estimates for every $\alpha\in(0,1]$. Newhouse \cite[Theorem~1.5]{New17} also obtained $C^{1,\beta}$ linearization under $C^{1,\alpha}$ regularity, for $0<\alpha<1$ and some $\beta\in(0,\alpha)$, in real Banach spaces admitting $C^{1,\alpha}$ bump functions, subject to an $\alpha$-hyperbolicity condition. This condition holds for every planar hyperbolic saddle and every $\alpha\in(0,1)$, since its stable and unstable subspaces are both one-dimensional.
	
	As discussed earlier, H\"older regularity is insufficient to characterize general $C^1$ systems. While real diagonalizable contractions with distinct eigenvalue moduli require the spectral threshold $\alpha>\alpha_0$, planar saddle points impose no positive lower threshold on $\alpha$. The natural candidate in the modulus-of-continuity framework is therefore again the classical Dini condition~\eqref{dinicd}. The degenerate-node case nevertheless shows that the absence of a positive H\"older threshold does not, by itself, determine the exact Dini weight. Establishing the saddle criterion therefore requires both a sufficiency proof at the Dini level and the construction of a counterexample for each prescribed pair of eigenvalues whenever the condition fails.
	
	Our sixth main result provides a definitive answer to this question:
	
	\begin{mainthm}[Hyperbolic saddle points]\label{TH4}
		Let $\omega$ be a modulus of continuity. Then the following statements hold for any prescribed $0<|\lambda_1|<1<|\lambda_2|$:
		\begin{enumerate}[label=\textup{(\Roman*)}]
			\item \label{TH4-I}		
			If the Dini condition \eqref{dinicd} holds, then every planar $C_{1,\omega}$ diffeomorphism $F$ in Case \ref{case-S-saddle} is locally $C^1$ linearizable.
			
			\item \label{TH4-II}
			If the Dini condition \eqref{dinicd} fails, then there exists a planar $C_{1,\omega}$ diffeomorphism $F$ in Case \ref{case-S-saddle} that is not locally $C^1$ linearizable.
		\end{enumerate}
		Equivalently, the class $C_{1,\omega}$ of planar hyperbolic diffeomorphisms in Case \ref{case-S-saddle} is locally $C^1$ linearizable if and only if the Dini condition \eqref{dinicd} holds.

		Moreover, in Item~\ref{TH4-I}, the linearizing diffeomorphism $\Phi$ and its inverse can be chosen of class $C_{1,\omega_{\mathrm D}}$, where
\[
\delta:=
			\frac{\min\{-\log|\lambda_1|,\log|\lambda_2|\}}
			{\log|\lambda_2|-\log|\lambda_1|}\in(0,1/2]
\]
		and, for sufficiently small $t>0$,
		\begin{equation}\label{eq:saddle-regularity-modulus}
			\omega_{\mathrm D}(t):=\omega_{\rm B}(t^\delta)
			=\int_0^{t^\delta}\frac{\omega(s)}s ds.
		\end{equation}
		This estimate is optimal for each prescribed pair $(\lambda_1,\lambda_2)$: there exists such a diffeomorphism $F$ for which every normalized local linearizing diffeomorphism $\Phi\in C_{1,\eta}$ satisfies
		\[
		\omega_{\mathrm D}(t)=\mathcal O(\eta(t))
		\quad\text{as }t\to0^+.
		\]
	\end{mainthm}

	For $\omega(t)=t^\alpha$ with $0<\alpha\leqslant1$, \eqref{eq:saddle-regularity-modulus} gives $\omega_{\mathrm D}(t)=\alpha^{-1}t^{\alpha\delta}$, recovering the optimal exponent in \cite[Theorems~3 and~4]{ZZ14}. Thus, although the Dini criterion for existence is independent of the eigenvalues, the optimal modulus of the conjugacy retains a spectral dependence. For $\omega(t)=(\log(e/t))^{-p}$ with $p>1$, one has $\omega_{\mathrm D}(t)\asymp(\log(e/t))^{1-p}$, which is equivalent to the modulus $\omega_{\rm B}$ in this particular scale.

	Theorem~\ref{TH4} provides  a complete characterization  of $C^1$ linearizability in Case \ref{case-S-saddle}. Together with Theorems~\ref{TH2}, \ref{thm:focus}, and~\ref{thm:jordan}, and the corresponding results for repellers obtained by passing to the inverse mapping, this brings the uniform local $C^1$ linearizability problem for all planar hyperbolic real Jordan forms to a \textit{definitive resolution}. Each criterion is optimal for every prescribed linear part of the corresponding type.
	
	As a byproduct of the analysis leading to Theorem~\ref{TH4}, Lemma~\ref{DINILEMMA} establishes the sufficiency of the Dini condition \eqref{dinicd} for local $C^1$ linearizability in dimension one. Furthermore, the necessity of this condition follows readily from a direct adaptation of the counterexample constructed in Theorem~\ref{TH4}--\ref{TH4-II}. We also remark that the optimal regularity of the conjugacy beyond $C^1$ in this setting readily follows from the two-dimensional case by a completely analogous argument. We therefore state the following result concerning one-dimensional $C^1$ diffeomorphisms $F\colon \mathbb{R}\to\mathbb{R}$ satisfying
	\begin{equation}\label{th5tj}
		F(0)=0,\quad F'(0)=\lambda,\quad 0<|\lambda|<1.
	\end{equation}
	This result completes the characterization for general moduli of continuity beyond the classical H\"older regime \cite{Ste57a,Kuc68} (see Table \ref{tab:one-dimensional-linearization-summary}); we do not, however, regard this as a principal contribution of the present paper, in view of its substantial technical overlap with the existing literature.

	\begin{mainthm}\label{TH5}
		Let $\omega$ be a modulus of continuity. Then the following statements hold for any prescribed $0<|\lambda|<1$:
		\begin{enumerate}[label=\textup{(\Roman*)}]
			\item \label{TH5-I}		
			If the Dini condition \eqref{dinicd} holds, then every one-dimensional $C_{1,\omega}$ diffeomorphism $F$ satisfying \eqref{th5tj} is locally $C^1$ linearizable.
			
			\item \label{TH5-II}
			If the Dini condition \eqref{dinicd} fails, then there exists a one-dimensional $C_{1,\omega}$ diffeomorphism $F$ satisfying \eqref{th5tj} that is not locally $C^1$ linearizable.
		\end{enumerate}
		Equivalently, the class $C_{1,\omega}$ of one-dimensional diffeomorphisms satisfying \eqref{th5tj} is locally $C^1$ linearizable if and only if the Dini condition \eqref{dinicd} holds.
		
			Moreover, in Item \ref{TH5-I}, the linearizing diffeomorphism can be chosen of class $C_{1,\omega_{\rm B}}$. This estimate is optimal for each prescribed $\lambda$: there exists such a diffeomorphism $F$ for which every normalized local linearizing diffeomorphism $\Phi\in C_{1,\eta}$ satisfies 
		\[
		\omega_{\rm B}(t)=\mathcal{O}(\eta(t))\quad\text{as } t\to0^+.
		\]
	\end{mainthm}
	\begin{customrem}{G}\label{RE13}
		If $|\lambda|=1$ in \eqref{th5tj}, then $F$ is $C^0$ linearizable at the origin if and only if $F^2 = \mathrm{id}$ in a neighborhood of the origin, hence regularity alone does not guarantee local $ C^0 $ linearizability.
	\end{customrem}

	\subsubsection{Cases~\ref{case-S-mixed}, \ref{case-S-real-unit}, \ref{case-S-elliptic}, and~\ref{case-S-unit-jordan}: Non-hyperbolic Fixed Points}

	We finally consider Cases~\ref{case-S-mixed}, \ref{case-S-real-unit}, \ref{case-S-elliptic}, and~\ref{case-S-unit-jordan}, corresponding to the non-hyperbolic regime. The failure of linearization in the absence of hyperbolicity is a classical phenomenon. In the one-dimensional setting, Sternberg \cite[Remark~(a)]{Ste57a} observed that a mapping with derivative $1$ can be conjugate to its linear part only if it is the identity near the fixed point; for instance, the real-analytic diffeomorphism $x\mapsto x+x^3$ cannot be locally $C^0$ linearized at the origin. While the insufficiency of regularity without hyperbolicity is well understood, for completeness and to render our planar classification entirely self-contained, we record the following uniform statement and provide explicit polynomial constructions for every non-hyperbolic real Jordan form.
	
	\begin{mainthm}[Non-hyperbolic fixed points]\label{THNH}
		Let $\Lambda\in\mathrm{GL}(2,\mathbb{R})$ have at least one
		eigenvalue on the unit circle $S^1$.
		Then there exists a global real-analytic diffeomorphism
		$F\colon\mathbb{R}^2\to\mathbb{R}^2$ with polynomial components of
		degree at most three such that
		\[
		F(O)=O,\quad DF(O)=\Lambda,
		\]
		but $F$ is not locally $C^0$ (and hence not locally $C^1$) linearizable at $O$.
	\end{mainthm}
	
	Consequently, in the absence of hyperbolicity, no degree of smoothness---even real analyticity---can guarantee local linearization, even in the topological category. Together with the preceding theorems, this completes the classification of planar local $C^1$ linearizability across all real Jordan forms.

\subsection{Summary of the Optimal Results}\label{seccon}

\begingroup
\setlength{\tabcolsep}{4pt}
\setlength{\extrarowheight}{0pt}
\renewcommand{\arraystretch}{1}
\renewcommand{\tabularxcolumn}[1]{m{#1}}
\providecommand{\PlanarSummaryCell}[1]{}
\renewcommand{\PlanarSummaryCell}[1]{%
	\parbox[c]{\linewidth}{%
		\centering\vspace*{4pt}%
		#1\par\vspace*{4pt}%
	}%
}

The results of this paper provide a \textit{complete characterization} of the
regularity assumptions guaranteeing local $C^1$ linearizability for
planar diffeomorphisms with a prescribed linear part.
Table~\ref{tab:planar-linearization-summary} summarizes the optimal
criteria in the hyperbolic regime and the obstruction in the
non-hyperbolic regime, while
Table~\ref{tab:one-dimensional-linearization-summary} records the
corresponding one-dimensional results. For the real diagonalizable
contracting case, we recall that
\[
\alpha_0:=1-\frac{\log|\lambda_2|}{\log|\lambda_1|},
\quad 0<|\lambda_1|\leqslant|\lambda_2|<1.
\]
When $|\lambda_1|=|\lambda_2|$, we have $\alpha_0=0$, and the
polynomially weighted Dini condition~\eqref{kjh} reduces to the
Dini condition~\eqref{dinicd}. Only the contracting cases are listed
for attracting and repelling fixed points; the corresponding repelling
cases follow by passing to the inverse mapping.

\begin{table}[htbp]
	\centering
	\small
	\begin{tabularx}{\linewidth}{|
			>{\centering\arraybackslash}m{0.19\linewidth}|
			>{\centering\arraybackslash}m{0.33\linewidth}|
			>{\centering\arraybackslash}X|}
		\hline
		\PlanarSummaryCell{\textbf{Result}}
		& \PlanarSummaryCell{\textbf{Fixed-point type and case}}
		& \PlanarSummaryCell{\textbf{Optimal criterion or obstruction}} \\
		\hline
		\PlanarSummaryCell{Theorem~\ref{TH2}}
		& \PlanarSummaryCell{
			Case~\ref{case-P-attractor}\par
			Hyperbolic real diagonalizable attractors\par
			$0<|\lambda_1|\leqslant|\lambda_2|<1$}
		& \PlanarSummaryCell{
			Polynomially weighted Dini condition\par\smallskip
			$\displaystyle
			\int_0^1\frac{\omega(t)}{t^{1+\alpha_0}}dt<+\infty$
			\quad\eqref{kjh}} \\
		\hline
		\PlanarSummaryCell{Theorem~\ref{thm:focus}}
		& \PlanarSummaryCell{
			Case~\ref{case-P-focus}\par
			Hyperbolic attracting foci}
		& \PlanarSummaryCell{
			Dini condition\par\smallskip
			$\displaystyle
			\int_0^1\frac{\omega(t)}{t}dt<+\infty$
			\quad\eqref{dinicd}} \\
		\hline
		\PlanarSummaryCell{Theorem~\ref{thm:jordan}}
		& \PlanarSummaryCell{
			Case~\ref{case-P-degenerate-node}\par
			Hyperbolic attracting degenerate nodes}
		& \PlanarSummaryCell{
			Squared-logarithmically weighted Dini condition\par\smallskip
			$\displaystyle
			\int_0^1\frac{\omega(t)}{t}
			\left(\log\frac{e}{t}\right)^2dt<+\infty$
			\quad\eqref{eq:log2dini}} \\
		\hline
		\PlanarSummaryCell{Theorem~\ref{TH4}}
		& \PlanarSummaryCell{
			Case~\ref{case-S-saddle}\par
			Hyperbolic saddle points}
		& \PlanarSummaryCell{
			Dini condition\par\smallskip
			$\displaystyle
			\int_0^1\frac{\omega(t)}{t}dt<+\infty$
			\quad\eqref{dinicd}} \\
		\hline
		\PlanarSummaryCell{Theorem~\ref{THNH}}
		& \PlanarSummaryCell{
			Cases~\ref{case-S-mixed}, \ref{case-S-real-unit},
			\ref{case-S-elliptic}, and~\ref{case-S-unit-jordan}\par
			Non-hyperbolic fixed points}
		& \PlanarSummaryCell{
			Regularity alone does not guarantee local $C^0$ linearizability.} \\
		\hline
	\end{tabularx}
	\caption{Optimal criteria and obstructions for planar $C^1$ linearization.}
	\label{tab:planar-linearization-summary}
\end{table}

\begin{table}[htbp]
	\centering
	\small
	\begin{tabularx}{\linewidth}{|
			>{\centering\arraybackslash}m{0.19\linewidth}|
			>{\centering\arraybackslash}m{0.33\linewidth}|
			>{\centering\arraybackslash}X|}
		\hline
		\PlanarSummaryCell{\textbf{Result}}
		& \PlanarSummaryCell{\textbf{Fixed-point type}}
		& \PlanarSummaryCell{\textbf{Optimal criterion or obstruction}} \\
		\hline
		\PlanarSummaryCell{Theorem~\ref{TH5}}
		& \PlanarSummaryCell{
			Hyperbolic attracting fixed points\par
			$0<|\lambda|<1$}
		& \PlanarSummaryCell{
			Dini condition\par\smallskip
			$\displaystyle
			\int_0^1\frac{\omega(t)}{t}dt<+\infty$
			\quad\eqref{dinicd}} \\
		\hline
		\PlanarSummaryCell{Remark~\ref{RE13}}
		& \PlanarSummaryCell{
			Non-hyperbolic fixed points\par
			$|\lambda|=1$}
		& \PlanarSummaryCell{
			Regularity alone does not guarantee local $C^0$ linearizability.} \\
		\hline
	\end{tabularx}
	\caption{Optimal criterion and obstruction for one-dimensional $C^1$
		linearization.}
	\label{tab:one-dimensional-linearization-summary}
\end{table}

All the hyperbolic criteria in
Table~\ref{tab:planar-linearization-summary} are \textit{optimal} for every
prescribed choice of eigenvalues and the corresponding real Jordan
form. In each case, the displayed condition guarantees local $C^1$
linearizability for every diffeomorphism in the class $C_{1,\omega}$;
whenever it fails, a counterexample exists with that same linear part.
For every prescribed non-hyperbolic linear part,
Theorem~\ref{THNH} provides a global real-analytic diffeomorphism with
polynomial components of degree at most three that fails to admit a
local $C^0$ linearization. Together with the optimal one-dimensional
Dini criterion in Theorem~\ref{TH5} (Table~\ref{tab:one-dimensional-linearization-summary}), these results determine the exact
scope of regularity assumptions alone in guaranteeing $C^1$
linearization in dimensions one and two.

Under the corresponding integrability conditions, we also obtain \textit{optimal}
regularity estimates for the derivatives of the linearizing conjugacies in
every hyperbolic case.
For anisotropic real diagonalizable contractions, the output modulus of continuity in~\eqref{eq:anisotropic-optimal-modulus} reads
\[	\omega_{\mathrm{A}}(t):=
\begin{cases}
	\displaystyle t^{\alpha_1}\int_0^t\frac{\omega(s)}{s^{1+\alpha_1}}ds,
	&  I_{\alpha_1}(\omega)<+\infty,\\[3mm]
	\displaystyle\int_0^{t^{\frac{1}{\alpha_0}}}\frac{\omega(s)}{s^{1+\alpha_0}}ds
	+t^{\alpha_1}\int_{t^{\frac{1}{\alpha_0}}}^1\frac{\omega(s)}{s^{1+\alpha_1}}ds,
	&  I_{\alpha_1}(\omega)=+\infty,
\end{cases}\]
 with the branch
determined by $I_{\alpha_1}(\omega):=\int_0^1\frac{\omega(s)}{s^{1+\alpha_1}}ds$ and
$\alpha_1=\alpha_0/(1-\alpha_0)$.
The other output moduli of continuity, from~\eqref{omegaA},
\eqref{eq:jordan-regularity-modulus}, and~\eqref{eq:saddle-regularity-modulus}, are
\[
\omega_{\rm B}(t)=\int_0^t\frac{\omega(s)}s ds,
\quad
\omega_{\rm C}(t)=\int_0^t\frac{\omega(s)}s
\left(\log\frac{et}s\right)^2ds,
\]
and
\[
\omega_{\rm D}(t)=\omega_{\rm B}(t^\delta)
=\int_0^{t^\delta}\frac{\omega(s)}s ds,\quad \delta:=\frac{\min\{-\log|\lambda_1|,\log|\lambda_2|\}}
{\log|\lambda_2|-\log|\lambda_1|},
\quad 0<|\lambda_1|<1<|\lambda_2|.
\]
Each listed modulus of continuity is optimal for the class with the prescribed linear part; see Table~\ref{tab:linearization-sharp-regularity}.

\begin{table}[htbp]
    \centering
    \small
    \setlength{\tabcolsep}{4pt}
    \setlength{\extrarowheight}{0pt}
    \renewcommand{\arraystretch}{1.4}
    \renewcommand{\tabularxcolumn}[1]{m{#1}}
    \begin{tabularx}{\linewidth}{|
        >{\centering\arraybackslash}m{0.18\linewidth}|
        >{\centering\arraybackslash}m{0.43\linewidth}|
        >{\centering\arraybackslash}X|}
        \hline
        \textbf{Result}
        & \textbf{Fixed-point type and case}
        & \textbf{Optimal regularity of the linearizing diffeomorphism} \\
        \hline
        Theorem~\ref{TH3}--\ref{TH3-1}
        & Case~\ref{case-P-attractor}\par
          Hyperbolic real diagonalizable attractors\par
          $0<|\lambda_1|<|\lambda_2|<1$
        & $C_{1,\omega_{\rm A}}$ \\
        \hline
        Theorem~\ref{TH3}--\ref{TH3-2}
        & Case~\ref{case-P-attractor}\par
          Hyperbolic real diagonalizable attractors\par
          $0<|\lambda_1|=|\lambda_2|<1$
        & $C_{1,\omega_{\rm B}}$ \\
        \hline
        Theorem~\ref{thm:focus}
        & Case~\ref{case-P-focus}\par Hyperbolic attracting foci
        & $C_{1,\omega_{\rm B}}$ \\
        \hline
        Theorem~\ref{thm:jordan}
        & Case~\ref{case-P-degenerate-node}\par
          Hyperbolic attracting degenerate nodes
        & $C_{1,\omega_{\rm C}}$ \\
        \hline
        Theorem~\ref{TH4}
        & Case~\ref{case-S-saddle}\par Hyperbolic saddle points
        & $C_{1,\omega_{\rm D}}$ \\
        \hline
        Theorem~\ref{TH5}
        & One-dimensional hyperbolic attracting fixed points\par
          $0<|\lambda|<1$
        & $C_{1,\omega_{\rm B}}$ \\
        \hline
    \end{tabularx}
    \caption{Optimal regularity of linearizing diffeomorphisms.}
    \label{tab:linearization-sharp-regularity}
\end{table}

From an applied viewpoint, the optimality of the planar $C^1$
linearizability criteria allows one to relax regularity hypotheses in
applications with the corresponding spectral configuration, while the
accompanying counterexamples delineate the threshold beyond which
regularity alone no longer guarantees linearization. On a
methodological level, the analysis developed herein suggests an
approach to optimal conditions formulated in terms of moduli of
continuity for higher-order linearizations and for dynamical systems
in higher or even infinite dimensions. We leave these questions for future
investigation.

\endgroup

\subsection{Why General Moduli of Continuity Matter}
\label{sec:beyond-holder-comparison}

The optimal $C^1$ linearizability criteria summarized in Section~\ref{seccon} reveal fine distinctions that remain \textit{invisible} on the classical H\"older scale. Write $L(t)=\log(e/t)$ and $M(t)=\log(eL(t))$. Beyond the standard H\"older moduli $t^\alpha$, natural examples of moduli of continuity  include the log-Lipschitz modulus $tL(t)$, the log-H\"older moduli $L(t)^{-p}$, their iterated-logarithmic refinements $L(t)^{-p}M(t)^{-q}$, and the stretched-logarithmic family $e^{-L(t)^\gamma}$ for $0<\gamma<1$. These expressions are considered for sufficiently small $t>0$, where they are strictly increasing and concave across the parameter regimes specified below, and are extended to $[0,\infty)$. A smaller modulus imposes a strictly stronger regularity requirement. For every $\varepsilon,p,q>0$ and $0<\gamma<1$, it can be verified that
\[
t^\varepsilon=o\bigl(e^{-L(t)^\gamma}\bigr),\quad
e^{-L(t)^\gamma}=o\bigl(L(t)^{-p}\bigr),\quad
L(t)^{-p}=o\bigl(M(t)^{-q}\bigr)\quad\text{as }t\to0^+.
\]
Consequently, general moduli resolve an entire hierarchy of sub-H\"older regularity lying strictly between $C^1$ and $\bigcup_{\varepsilon>0}C^{1,\varepsilon}$; see Figure~\ref{fig:beyond-holder-compact}(a). In particular, a log-H\"older modulus does not imply H\"older continuity. A comprehensive comparison between these moduli of continuity and the respective Dini and weighted Dini conditions is summarized in Table~\ref{tab:beyond-holder-compact}.

\begin{table}[htbp]
	\centering
	\begingroup
	\small
	\setlength{\tabcolsep}{4pt}
	\renewcommand{\arraystretch}{1.3}
	\renewcommand{\tabularxcolumn}[1]{m{#1}}
	\begin{tabularx}{\linewidth}{|>{\centering\arraybackslash}m{0.36\linewidth}|*{3}{>{\centering\arraybackslash}X|}}
		\hline
		Modulus $\omega(t)$ near zero
		& Dini condition \par \eqref{dinicd}
		& Weighted Dini condition \par  \eqref{eq:log2dini}
		& Weighted Dini condition \par \eqref{kjh} \\
		\hline
		$t^\alpha$, $0<\alpha\leqslant1$ & Always & Always & $\alpha>\alpha_0$ \\
		\hline
		$tL(t)$ & Always & Always & Always \\
		\hline
		$t^{\alpha_0}L(t)^{-p}$, $p\in\mathbb R$ & Always & Always & $p>1$ \\
		\hline
		$L(t)^{-p}$, $p>0$ & $p>1$ & $p>3$ & Never \\
		\hline
		$L(t)^{-1}M(t)^{-q}$, $q\in\mathbb R$ & $q>1$ & Never & Never \\
		\hline
		$L(t)^{-3}M(t)^{-q}$, $q\in\mathbb R$ & Always & $q>1$ & Never \\
		\hline
		$M(t)^{-q}$, $q>0$ & Never & Never & Never \\
		\hline
		$e^{-L(t)^\gamma}$, $0<\gamma<1$ & Always & Always & Never \\
		\hline
	\end{tabularx}
	\endgroup
	\caption{Exact convergence conditions for the three linearizability criteria, with $0<\alpha_0<1$ fixed. ``Always'' and ``Never'' refer to the stated parameter ranges.}
	\label{tab:beyond-holder-compact}
\end{table}

For the input modulus $\omega=L^{-p}$ with $p>3$, the summarized notations in Section~\ref{seccon} yield
\[
\omega_{\mathrm{B}}(t)\asymp L(t)^{1-p},\quad
\omega_{\mathrm{C}}(t)\asymp L(t)^{3-p},\quad
\omega_{\mathrm{D}}(t)\asymp\omega_{\mathrm{B}}(t).
\]
The contracting Jordan configuration therefore incurs a loss of three inverse-logarithmic powers, in sharp contrast to the single-power loss in the focus or isotropic cases; see Figure~\ref{fig:beyond-holder-compact}(b). For standard H\"older moduli $\omega(t)=t^\alpha$, by contrast, both $\omega_{\mathrm{B}}$ and $\omega_{\mathrm{C}}$ remain equivalent to $t^\alpha$, entirely masking this disparity. General moduli thus distinguish both the sharp thresholds and the optimal conjugacy regularity losses that lie beyond the reach of a single H\"older exponent.

\begin{figure}[htbp]
	\centering
	\includegraphics[width=\linewidth]{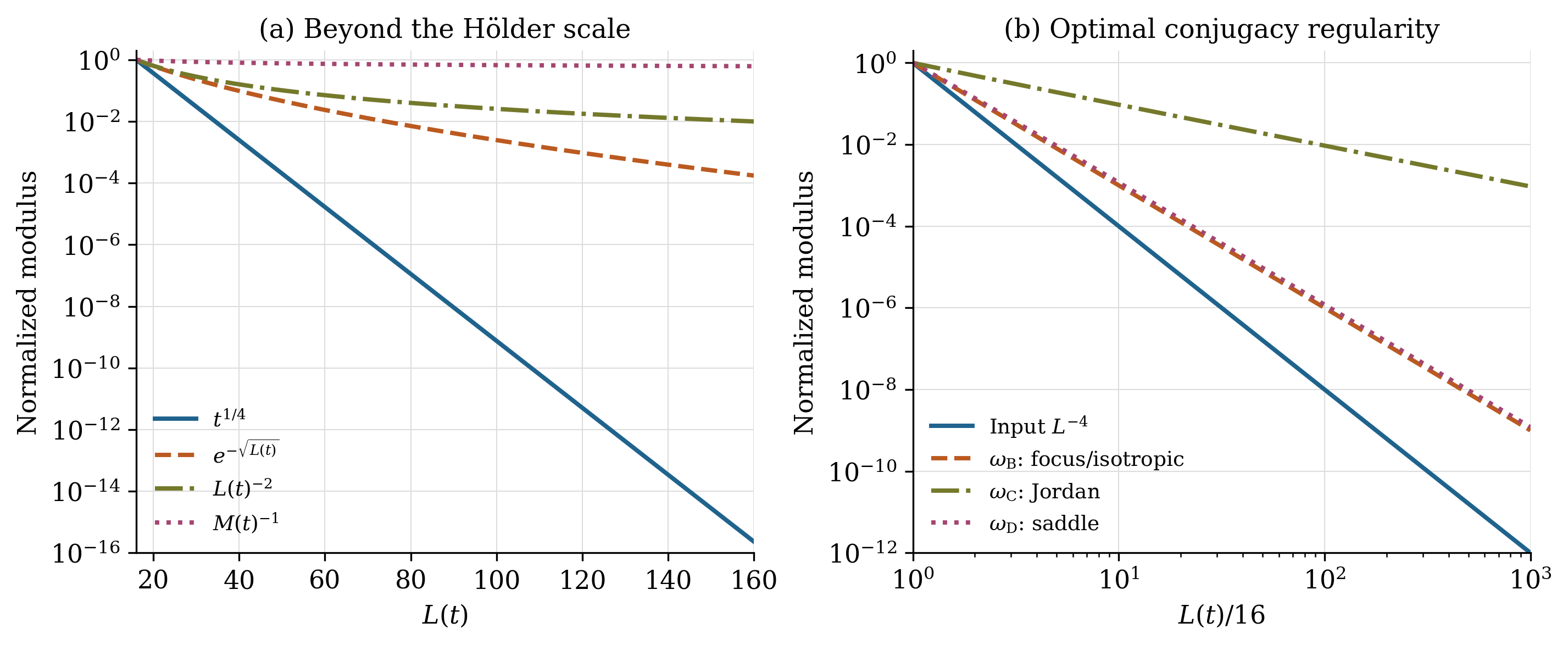}
	\caption{(a)~Representative H\"older and non-H\"older moduli. (b)~Input modulus $\omega=L^{-4}$ alongside the optimal output moduli $\omega_{\mathrm{B}}$ (focus/isotropic), $\omega_{\mathrm{C}}$ (Jordan), and $\omega_{\mathrm{D}}$ (saddle, with $\delta=1/2$). Each curve is normalized by its value at $t_*=e^{-15}$, so that $L(t_*)=16$. Moving rightward corresponds to $t\to0^+$; lower curves decay more rapidly. Panel~(a) features a logarithmic vertical scale; panel~(b) employs double logarithmic scales. Normalization allows for comparison of decay rates rather than absolute constants.}
	\label{fig:beyond-holder-compact}
\end{figure}

\section{Proofs of Main Results}
The remainder of the paper is devoted to the proofs of Theorems~\ref{TH1}, \ref{TH3}, \ref{thm:focus}, \ref{thm:jordan}, \ref{TH4}, and~\ref{THNH}. In the proofs, positive estimate constants $C_0,C_1,\ldots$ are numbered in order of first occurrence. They may depend on fixed data; independence of varying parameters is specified where needed.

	\subsection{Proof of Theorem \ref{TH1}}\label{SEC3}
	\begin{proof}

		Since the isotropic case $0<|\lambda_1| = |\lambda_2| <1$ was treated in \cite[Example 3.1]{TXL24}, throughout this section we first assume that $0 < \lambda_1 < \lambda_2 < 1$ and that the weighted Dini  condition \eqref{kjh} fails (i.e., $\int_0^1 \frac{\omega(t)}{t^{\alpha_0+1}} dt = +\infty$). We focus first on this case of same-signed eigenvalues because it contains the essential mechanism; the case of opposite signs will then require only minor modifications.
		
		Choose a symmetric bump function $\chi \in C_c^\infty(\mathbb{R})$ satisfying $0 \leqslant \chi(t) \leqslant 1$ for all $t \in \mathbb{R}$, with
		\[
		\chi(t) = \begin{cases}
			1, & |t| \leqslant 1, \\
			0, & |t| \geqslant 2,
		\end{cases}
		\]
		and define the auxiliary smooth function $K(t) := t \chi(t) \in C_c^\infty(\mathbb{R})$. Figure~\ref{fig:chi-bump} illustrates the profile of $\chi$.

		\par\medskip
		\noindent
		\begin{minipage}[t]{0.48\textwidth}
			\vspace{0pt}
			\centering
			\includegraphics[width=\linewidth]{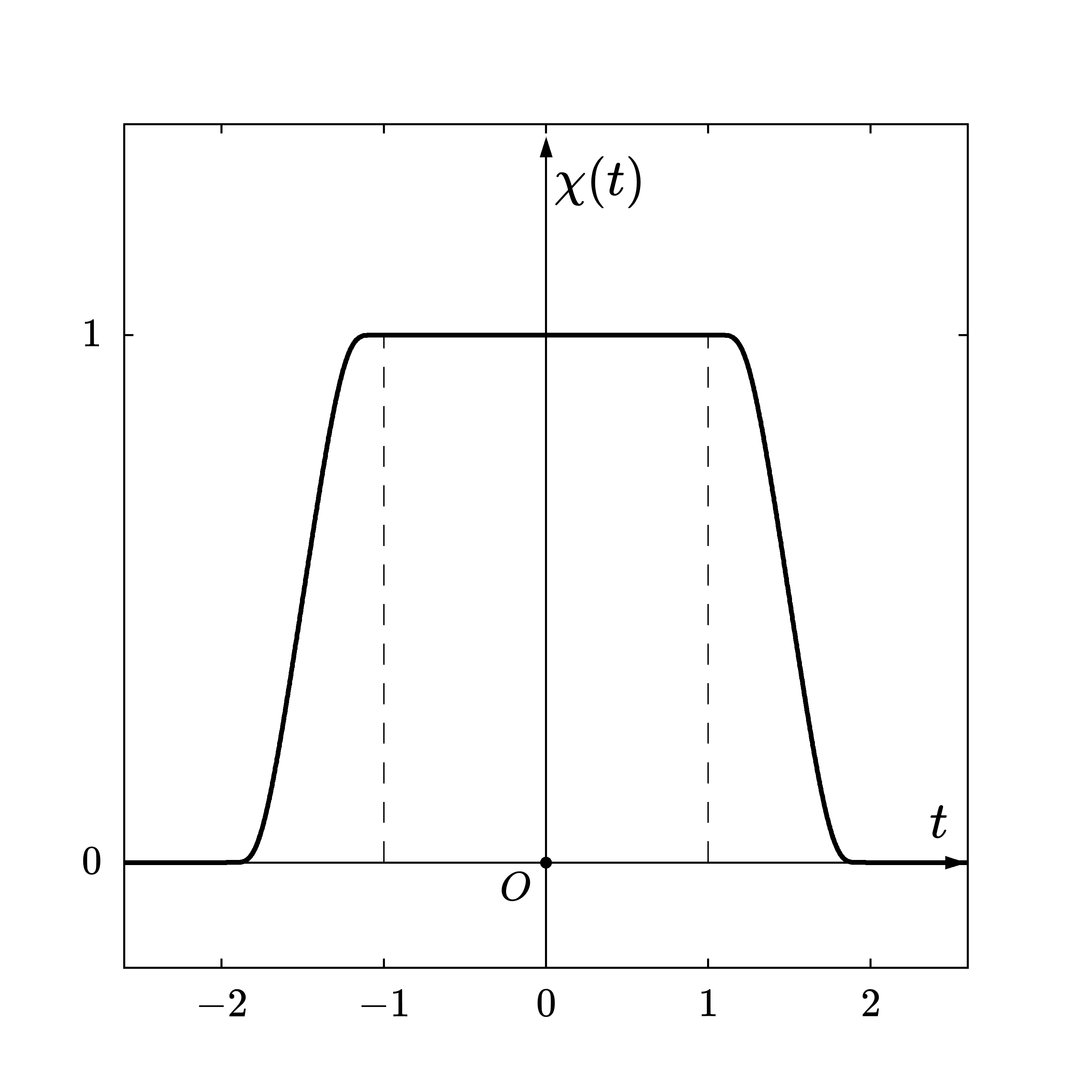}
			\captionsetup{
				font=small,
				justification=centering,
				singlelinecheck=false,
				hypcap=false
			}
			\captionof{figure}{A symmetric bump function $\chi$.}
			\label{fig:chi-bump}
		\end{minipage}\hfill%
		\begin{minipage}[t]{0.48\textwidth}
			\vspace{0pt}
			\centering
			\includegraphics[width=\linewidth]{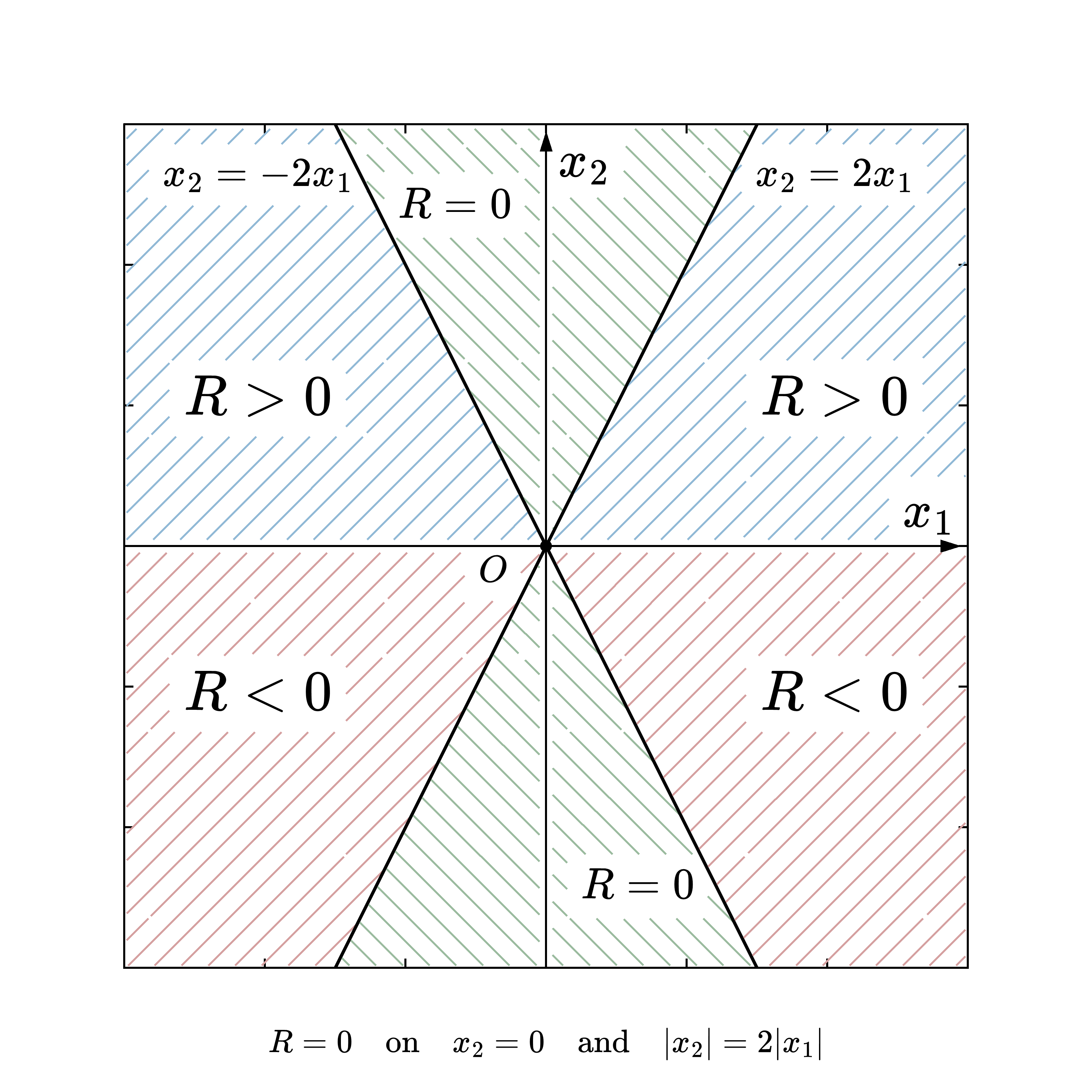}
			\captionsetup{
				font=small,
				justification=centering,
				singlelinecheck=false,
				hypcap=false
			}
			\captionof{figure}{Sign regions of the perturbation $R$.}
			\label{fig:R-sign-regions}
		\end{minipage}
		\par\medskip

		For the given concave modulus of continuity $\omega$, we introduce its primitive $\Omega$ and the normalized averaged modulus of continuity $\widehat{\omega}$ by
		\[
		\Omega(t) := \int_0^t \omega(s) ds, \quad \widehat{\omega}(t) := \frac{\Omega(t)}{t} = \frac{1}{t}\int_0^t \omega(s)ds \quad (t > 0), \quad \widehat{\omega}(0) := 0.
		\]
		It should be pointed out that such a normalized average $\widehat{\omega}$ heuristically resembles $\omega$ in terms of order of magnitude, yet enhances its regularity, which plays a pivotal role in our subsequent construction of concrete counterexamples.
		
		\begin{lemma}\label{prop21}
			The normalized averaged modulus of continuity $\widehat{\omega}$ possesses the following properties:
			\begin{enumerate}[label=(\arabic*), ref=(\arabic*)]

				\item \label{p1} For all $t > 0$, $\frac{1}{2}\omega(t) \leqslant \widehat{\omega}(t) \leqslant \omega(t)$;
				\item \label{p3} $\widehat{\omega} \in C^1(0, +\infty)$ with $0 \leqslant \widehat{\omega}'(t) = \frac{\omega(t) - \widehat{\omega}(t)}{t} \leqslant \frac{\omega(t)}{t}$ for all $t > 0$, and for any $t_1, t_2 > 0$,
				\begin{equation*}
					|\widehat{\omega}(t_1) - \widehat{\omega}(t_2)| \leqslant \omega(|t_1 - t_2|).
				\end{equation*}
			\end{enumerate}
		\end{lemma}
		\begin{proof}
			By the monotonicity and the concavity of $\omega$ on $[0, t]$ and the fact that $\omega(0)=0$, we have $\frac{s}{t}\omega(t) \leqslant \omega(s) \leqslant \omega(t)$. Integrating over $s \in [0, t]$ yields $\frac{1}{2}t\omega(t) \leqslant \Omega(t) \leqslant t\omega(t)$, which immediately gives Item \ref{p1} upon dividing by $t$.
			
			For Item \ref{p3}, the first property is evident. For any $t_1, t_2 > 0$, using the change of variables $s = ut$, we write $\widehat{\omega}(t) = \int_0^1 \omega(ut)du$. Applying the subadditivity $|\omega(a) - \omega(b)| \leqslant \omega(|a-b|)$ for all $a,b\geqslant 0$\footnote{The subadditivity of $\omega$ is a direct consequence of its concavity and the condition $\omega(0) = 0$.} together with Item \ref{p1} gives
			\begin{equation*}
				|\widehat{\omega}(t_1) - \widehat{\omega}(t_2)| \leqslant \int_0^1 |\omega(ut_1) - \omega(ut_2)|du \leqslant \int_0^1 \omega(u|t_1 - t_2|)du = \widehat{\omega}(|t_1 - t_2|) \leqslant \omega(|t_1 - t_2|),
			\end{equation*}
			which completes the proof.
		\end{proof}
		
		We now define the desired counterexample $F: \mathbb{R}^2 \to \mathbb{R}^2$ by
		\begin{equation}\label{eq:3.2}
			F(x_1, x_2) := \begin{pmatrix} \lambda_1 x_1 + R(x_1, x_2) \\ \lambda_2 x_2 \end{pmatrix},
		\end{equation}
		where the nonlinear perturbation $R: \mathbb{R}^2 \to \mathbb{R}$ is given by
		\begin{equation}\label{eq:3.3}
			R(x_1, x_2) := \begin{cases}
				\Omega(|x_1|)\operatorname{sgn}(x_1) K\left(\dfrac{x_2}{x_1}\right), & x_1 \neq 0, \\
				0, & x_1 = 0.
			\end{cases}
		\end{equation}
		Figure~\ref{fig:R-sign-regions} illustrates the sign regions of the perturbation $R$.
		Here, the bump function is introduced primarily to ensure the regularity of $ F $ (see Lemma \ref{pro:regularity} below); it plays no essential role, however, in establishing the failure of $ C^1 $ linearization.
		
		\begin{lemma}\label{pro:regularity}
			The mapping $F$ defined by  \eqref{eq:3.2}--\eqref{eq:3.3} satisfies $F(O) = O$, $DF(O) = \Lambda$, and belongs to the regularity class $C_{1,\omega}(U)$ on any bounded neighborhood $U$ of $ O $. Moreover, after restricting $F$ to a sufficiently small neighborhood of $O$, it is a local $C^1$ diffeomorphism and a contraction.
		\end{lemma}
		
		\begin{proof}
			We divide $\mathbb{R}^2$ into the active conical sector $\mathcal C_{\text{in}}:= \{(x_1, x_2) \in \mathbb{R}^2 : |x_2| \leqslant 2|x_1|\}$ and the outer vanishing region $\mathcal C_{\text{out}}:= \{(x_1, x_2) \in \mathbb{R}^2 : |x_2| > 2|x_1|\}$, separated by the boundary interface $\Sigma := \{(x_1, x_2) \in \mathbb{R}^2 : |x_2| = 2|x_1|\}$. 
			
			We first verify the differentiability of $R$ on $\mathbb{R}^2 \setminus \{O\}$. Since $\operatorname{supp}(K) \subset [-2, 2]$, $R \equiv 0$ identically on the open set $\mathcal{C}_{\text{out}}$, whence $DR(x) = (0, 0)$ for all $x \in \mathcal{C}_{\text{out}}$. On the open set $\mathbb{R}^2 \setminus \{x_1 = 0\}$, set $\theta = x_2/x_1$ and $G(t) := t K'(t) \in C_c^\infty(\mathbb{R})$. Since $ R $ is odd, i.e.,  
			$R(-x)=-R(x)$, its partial derivatives are even. Direct differentiation gives, for all $x_1 \neq 0$,
			\begin{equation}\label{eq:4.1}
				\partial_1 R(x) = \omega(|x_1|) K(\theta) - \widehat{\omega}(|x_1|) G(\theta), \quad \partial_2 R(x) = \widehat{\omega}(|x_1|) K'(\theta).
			\end{equation}
			Since $K, K'$, and $G$ vanish identically for $|\theta| \geqslant 2$, \eqref{eq:4.1} continuously evaluates to $(0, 0)$ on $\Sigma \setminus \{O\}$. Therefore, $R$ is everywhere differentiable on $\mathbb{R}^2 \setminus \{O\}$ with $DR$ continuous on $\mathbb{R}^2 \setminus \{O\}$.
			
			Next, we prove that \(R\) is Fr\'echet differentiable at \(O\) with
			\(DR(O)=0\). If \(x\in\mathcal C_{\mathrm{out}}\), then \(R(x)=0\).
			If \(x\in\mathcal C_{\mathrm{in}}\), then
			\[
			|R(x)|
			\leqslant \|K\|_\infty \Omega(|x_1|)
			\leqslant \|K\|_\infty |x_1|\omega(|x_1|)
			\leqslant \|K\|_\infty  |x |\omega( |x |).
			\]
			Here, $\|\cdot\|_\infty$ denotes the standard supremum norm with respect to the variable.
			Consequently, the same estimate holds for every \(x\in\mathbb R^2\). Since
			\(R(O)=0\) and \(\omega(0)=0\), we have
			\[
			0\leqslant
			\frac{|R(x)-R(O)|}{ |x |}
			\leqslant
			\|K\|_\infty\omega( |x |)
			\to 0
			\quad\text{as }x\to O.
			\]
			Thus \(R\) is Fr\'echet differentiable at \(O\), with \(DR(O)=0\).
			It follows from \eqref{eq:3.2} that $ F(O)=O $ and $ DF(O)=\Lambda $.
			
			Moreover, if \(x_1\neq0\), then \eqref{eq:4.1} and
			Lemma \ref{prop21} give
			\[
			|DR(x) |
			\leqslant
			\bigl(
			\|K\|_\infty+\|G\|_\infty+\|K'\|_\infty
			\bigr)\omega(|x_1|)
			\leqslant {C_{0}}\omega( |x |),
			\]
			where $ {C_{0}}:=\|K\|_\infty+\|G\|_\infty+\|K'\|_\infty>0$.
			If \(x_1=0\) and \(x\neq O\), then \(R\) vanishes in a neighborhood of
			\(x\), and hence \(DR(x)=0\). Therefore,
			\begin{equation}\label{fs}
				|DR(x)-DR(O) |
				= |DR(x) |
				\leqslant {C_{0}}\omega( |x |)
				\to 0
				\quad\text{as }x\to O.
			\end{equation}
			Thus \(DR\) is continuous at \(O\). Together with the continuity of \(DR\)
			on \(\mathbb R^2\setminus\{O\}\) established above, this proves that
			\(R\in C^1(\mathbb R^2)\), and hence \(F\in C^1(U)\).

			It remains to show that $DR$ has modulus $\omega$.
			
			By symmetry, we first consider points $x, y \in \mathcal{C}_{\text{in}}^+ := \{x_1 > 0, |x_2| \leqslant 2x_1\}$. Let $h :=  |x - y |$, $\theta_x := x_2/x_1$, and $\theta_y := y_2/y_1 \in [-2, 2]$. 
			
			If $h \geqslant \frac{1}{2}x_1$, then $x_1 \leqslant 2h$ and $y_1 \leqslant x_1 + h \leqslant 3h$. Recalling \eqref{fs}, the monotonicity and subadditivity of $\omega$ yield
			\begin{equation*}
				|DR(x) - DR(y) | \leqslant  |DR(x) | +  |DR(y) | \leqslant {C_{0}} (\omega(2h) + \omega(3h)) \leqslant 5{C_{0}} \omega(h).
			\end{equation*}
			
			We now assume $h < \frac{1}{2}x_1$. Observe that
			\begin{equation}\label{zzsj}
				|\theta_x - \theta_y| = \left|\frac{x_2 - y_2}{x_1} + \frac{y_2(y_1 - x_1)}{x_1 y_1}\right| \leqslant \frac{|x_2 - y_2|}{x_1} + \frac{|y_2|}{y_1}\frac{|x_1 - y_1|}{x_1} \leqslant \frac{h}{x_1} + 2\frac{h}{x_1} = \frac{3h}{x_1},
			\end{equation}
			where we used $|y_2|/y_1 \leqslant 2$ since $y \in \mathcal{C}_{\text{in}}^+$. We now estimate the partial derivatives separately:
			\begin{enumerate}[label=(\roman*),leftmargin=1.5em]
				\item For $\partial_{2} R$, decomposing the difference gives
				\begin{equation*}
					|\partial_2 R(x) - \partial_2 R(y)| \leqslant |\widehat{\omega}(x_1) - \widehat{\omega}(y_1)| \|K'\|_\infty + \widehat{\omega}(x_1) |K'(\theta_x) - K'(\theta_y)|.
				\end{equation*}
				By Lemma \ref{prop21}, $|\widehat{\omega}(x_1) - \widehat{\omega}(y_1)| \leqslant \omega(|x_1 - y_1|) \leqslant \omega(h)$. For the second term, applying the mean value theorem to $K'$ on the interval between $\theta_x$ and $\theta_y$ together with \eqref{zzsj}, there exists some intermediate value $\xi$ such that
				\begin{equation*}
					|K'(\theta_x) - K'(\theta_y)| = |K''(\xi)| |\theta_x - \theta_y| \leqslant \|K''\|_\infty \frac{3h}{x_1}.
				\end{equation*}
				Since $\omega$ is concave with $\omega(0) = 0$, the function $t \mapsto \frac{\omega(t)}{t}$ is non-increasing on $(0, +\infty)$. Recalling $\widehat{\omega}(x_1) \leqslant \omega(x_1)$ from Lemma \ref{prop21} and noting $h < x_1$, we have
				\begin{equation}\label{fddome}
					\frac{\widehat{\omega}(x_1)}{x_1} \leqslant \frac{\omega(x_1)}{x_1} \leqslant \frac{\omega(h)}{h} \implies \widehat{\omega}(x_1)\frac{3h}{x_1} \leqslant 3\omega(h).
				\end{equation}
				Combining these estimates leads to the following bound with ${C_{1}} := \|K'\|_\infty + 3\|K''\|_\infty$:
				\begin{equation*}
					|\partial_2 R(x) - \partial_2 R(y)| \leqslant {C_{1}}\omega(h).
				\end{equation*}
				
				\item For $\partial_{1} R$, we decompose the difference into four terms:
				\begin{align*}
					|\partial_1 R(x) - \partial_1 R(y)| &\leqslant |\omega(x_1) - \omega(y_1)| \|K\|_\infty + \omega(x_1) |K(\theta_x) - K(\theta_y)| \\
					&\quad + |\widehat{\omega}(x_1) - \widehat{\omega}(y_1)| \|G\|_\infty + \widehat{\omega}(x_1) |G(\theta_x) - G(\theta_y)|.
				\end{align*}
				Applying the subadditivity $|\omega(x_1) - \omega(y_1)| \leqslant \omega(h)$, Lemma \ref{prop21}, the mean value theorem to $K$ and $G$, and the ratio bound \eqref{fddome}, we obtain similarly that
				\begin{equation*}
					|\partial_1 R(x) - \partial_1 R(y)| \leqslant {C_{2}}\omega(h),
				\end{equation*}
				where ${C_{2}} := \|K\|_\infty + 3\|K'\|_\infty + \|G\|_\infty + 3\|G'\|_\infty$.
			\end{enumerate}
			Therefore, setting ${C_{3}} := \max\{5{C_{0}}, {C_{1}}+ {C_{2}}\} > 0$, we arrive at
			\begin{equation*}
				|DR(x) - DR(y) | \leqslant {C_{3}} \omega( |x - y |), \quad \forall x, y \in \mathcal{C}_{\text{in}}^+,
			\end{equation*}
			and an identical estimate holds on $\mathcal{C}_{\text{in}}^-$ by symmetry.
			
			Finally, for arbitrary $x, y \in U$:
			\begin{itemize}[leftmargin=1.5em]
				\item If $x, y \in \mathcal{C}_{\text{out}}$, then $DR(x) - DR(y) = (0,0)$.
				\item If $x \in \mathcal{C}_{\text{in}}^+$ and $y \in \mathcal{C}_{\text{out}}$, let $z$ be the first point at which the line segment from $x$ to $y$ meets $\Sigma$. Then $z$ belongs to the closure of $\mathcal{C}_{\text{in}}^+$, $DR(z)=DR(y)=(0,0)$, and $ |x-z |\leqslant |x-y |$. Hence
				\begin{equation*}
					|DR(x) - DR(y) | =  |DR(x) - DR(z) | \leqslant {C_{3}} \omega( |x - z |) \leqslant {C_{3}} \omega( |x - y |).
				\end{equation*}
				\item If $x \in \mathcal{C}_{\text{in}}^+$ and $y \in \mathcal{C}_{\text{in}}^-$, the segment $[x, y]$ crosses the $x_2$-axis at some $z=(0,z_2)^\top$, where $DR(z)=(0,0)$. Applying the preceding mixed-region estimate to $(x,z)$ and $(y,z)$ when $z\neq O$, and using \eqref{fs} directly when $z=O$, we obtain
				\begin{equation*}
					|DR(x) - DR(y) | \leqslant  |DR(x) - DR(z) | +  |DR(z) - DR(y) | \leqslant 2{C_{3}} \omega( |x - y |).
				\end{equation*}
			\end{itemize}
			Setting ${C_{4}} := 2{C_{3}} > 0$, we obtain
			\[ |DF(x) - DF(y) | =  |DR(x) - DR(y) | \leqslant {C_{4}} \omega( |x - y |), \quad \forall x, y \in U,\]
			which proves the asserted $C_{1,\omega}$ regularity. 
			
			Finally, since $DR(x)\to 0$ as $x\to O$, we may choose a sufficiently small closed ball $U_0$ centered at $O$ such that
			\[
			\sup_{x\in U_0} |DR(x) |<\min\bigl\{1-\lambda_2, |\Lambda^{-1} |^{-1}\bigr\}.
			\]
			Then $\sup_{x\in U_0} |DF(x) |<1$, while
			\[
			DF(x)=\Lambda+
			\begin{pmatrix}
				\partial_1 R(x)&\partial_2 R(x)\\
				0&0
			\end{pmatrix}
			\]
			is invertible on $U_0$ by the Neumann-series criterion. Since $F(O)=O$, a closed ball $U_0$ may be chosen so that $F(U_0)\subset U_0$. The inverse function theorem now shows that $F$ is a local $C^1$ diffeomorphism, and the derivative bound shows that it is a contraction on $U_0$. This completes the proof of Lemma \ref{pro:regularity}.
		\end{proof}

		Next, we introduce a key mapping $\Psi$ and establish its $C^1$ regularity.
		
		\begin{lemma}\label{jde14yl}
			Suppose that the mapping $F$ defined by  \eqref{eq:3.2}--\eqref{eq:3.3} admits a local $C^1$ linearizing diffeomorphism. Then, on a sufficiently small forward invariant neighborhood of $O$, the limit
			\begin{equation}\label{psidy}
				\Psi(x):=\lim_{n\to\infty}\Lambda^{-n}F^n(x)
			\end{equation}
			exists locally uniformly and defines a $C^1$ mapping.
		\end{lemma}
		
		\begin{proof}
			Let $V$ be a sufficiently small forward invariant neighborhood of $O$ on which there is a $C^1$ diffeomorphism $\Phi$ satisfying
			\begin{equation}\label{geds}
				\Phi(F(x))=\Lambda\Phi(x),\quad \forall x\in V,
			\end{equation}
			with $\Phi(O)=O$ and $D\Phi(O)=\operatorname{Id}$.
			
			Since $\pi_2F(x)=\lambda_2x_2$, the derivative of every iterate has the triangular form
			\[
			DF^n(x)=
			\begin{pmatrix}
				a_n(x)&b_n(x)\\
				0&\lambda_2^n
			\end{pmatrix},
			\quad
			a_n(x):=\partial_1(\pi_1F^n)(x).
			\]
			Writing $D\Phi=(\varphi_{ij})_{1\leqslant i,j\leqslant2}$ and differentiating the equation obtained by iterating \eqref{geds} $n$ times ($ n \in \mathbb{N} $), we obtain
			\[
			D\Phi(F^n(x))DF^n(x)=\Lambda^nD\Phi(x).
			\]
			The first-row, first-column entry of this identity reads
			\[
			\varphi_{11}(F^n(x))a_n(x)=\lambda_1^n\varphi_{11}(x).
			\]
			Note that $F^n(x)$ converges to $O$ locally uniformly, while $\varphi_{11}(O)=1$. Consequently,
			\begin{equation}\label{dyfljx}
				\lim_{n\to\infty}	\lambda_1^{-n}a_n(x)=\lim_{n\to\infty}\frac{\varphi_{11}(x)}{\varphi_{11}(F^n(x))}
				=\varphi_{11}(x)
			\end{equation}
			locally uniformly on $V$.
			
			By \eqref{eq:3.3}, $F_1(0,x_2)=0$, and hence $\pi_1F^n(0,x_2)=0$ for every $n\geqslant1$. Therefore, utilizing \eqref{dyfljx} gives
			\[	\lim_{n\to\infty}\lambda_1^{-n}\pi_1F^n(x_1,x_2)
			=\lim_{n\to\infty}\int_0^{x_1}\lambda_1^{-n}a_n(t,x_2)dt
			=\int_0^{x_1}\varphi_{11}(t,x_2)dt
			=\Phi_1(x_1,x_2)-\Phi_1(0,x_2).\]
			Moreover, $ 	\lambda_2^{-n}\pi_2F^n(x_1,x_2)=x_2 $ for all $n$. Thus the limit in \eqref{psidy} exists locally uniformly and is explicitly given by
			\[
			\Psi(x_1,x_2)=\bigl(\Phi_1(x_1,x_2)-\Phi_1(0,x_2),x_2\bigr)^\top,
			\]
			which proves that $\Psi\in C^1(V)$. 
		\end{proof}

		We are now in a position to prove that such a contraction $F$ is not $C^1$ linearizable.

		\begin{lemma}\label{tdlm}
			The mapping $F$ defined by  \eqref{eq:3.2}--\eqref{eq:3.3} does not admit a local $C^1$ linearization near $O$.
		\end{lemma}
		
		\begin{proof}
			Suppose, to the contrary, that $F$ admits a local $C^1$ linearization. By Lemma \ref{jde14yl}, after shrinking $V$ to a forward invariant neighborhood if necessary, the mapping $\Psi$ defined by \eqref{psidy} exists and belongs to $C^1$.
			
			Choose \(\xi>0\) so small that
			\((\xi,y)\in V\) for every \(0\leqslant y\leqslant\xi\).
			Write
			\[
			F^k(\xi,y):=(x_k(y),y_k)=(x_k,\lambda_2^ky),\quad k\in \mathbb{N},
			\]
			where we denote $x_k=x_k(y)$ for brevity. Because \(R\geqslant0\) in the first quadrant, induction gives $x_k\geqslant\lambda_1^k\xi >0$.
			For $0<y<\xi$, define the permissible non-escape time by
			\begin{equation}\label{cutoffN}
				N(y):=\max\bigl\{N\in\mathbb{N}: \lambda_1^{-N}\lambda_2^Ny\leqslant\xi\bigr\}.
			\end{equation}
			Then $N(y)\to+\infty$ as $y\to0^+$. For every $0\leqslant k\leqslant N(y)$, we have
			\[
			0<\frac{y_k}{x_k}
			\leqslant\frac{\lambda_2^ky}{\lambda_1^k\xi}\leqslant1.
			\]
			Since $K(t)=t$ for $0\leqslant t\leqslant1$, we therefore have
			\begin{equation}\label{Rxjgj}
				R(x_k,y_k)
				=\Omega(x_k)K\left(\frac{y_k}{x_k}\right)=\widehat{\omega}(x_k)y_k\geqslant\widehat{\omega}(\lambda_1^k\xi)\lambda_2^ky,
				\quad 0\leqslant k\leqslant N(y),
			\end{equation}
			where we have used the monotonicity of $\widehat{\omega}$.
			
			Note that the first component of the iterates satisfies
			\[
			x_{k+1}=\lambda_1x_k+R(x_k,y_k).
			\]
			Consequently, for every $n \in \mathbb{N}^+$,
			\begin{equation}\label{dd1flgj}
				\lambda_1^{-n}x_n
				=\xi+\frac{1}{\lambda_1}\sum_{k=0}^{n-1}\lambda_1^{-k}R(x_k,y_k).
			\end{equation}
			Since the summands in \eqref{dd1flgj} are nonnegative, for $n>N(y)$, utilizing \eqref{Rxjgj} and $\lambda_2 =\lambda_1^{1-\alpha_0}$ yields
			\[
			\frac{\lambda_1^{-n}x_n-\xi}{y}
			\geqslant\frac{1}{\lambda_1}
			\sum_{k=0}^{N(y)}
			\left(\frac{\lambda_2}{\lambda_1}\right)^k
			\widehat{\omega}(\lambda_1^k\xi)
			=\frac{1}{\lambda_1}
			\sum_{k=0}^{N(y)}
			\frac{\widehat{\omega}(\lambda_1^k\xi)}{\lambda_1^{k\alpha_0}}.
			\]
			Passing to the limit as $n\to\infty$ and recalling that
			\[
			\Psi_1(\xi,0)=\lim_{n\to\infty}\lambda_1^{-n}\pi_1F^n(\xi,0)=\xi,
			\]
			we obtain
			\begin{equation}\label{xsdsgj1}
				\frac{\Psi_1(\xi,y)-\Psi_1(\xi,0)}{y}
				\geqslant\frac{1}{\lambda_1}
				\sum_{k=0}^{N(y)}
				\frac{\widehat{\omega}(\lambda_1^k\xi)}{\lambda_1^{k\alpha_0}}.
			\end{equation}
			
			In what follows, we assert that
			\[
			\lim_{y\to0^+}
			\frac{\Psi_1(\xi,y)-\Psi_1(\xi,0)}{y}=+\infty,
			\]
			whenever $\int_0^1 \frac{\omega(t)}{t^{\alpha_0+1}} dt = +\infty$. For every $k \in \mathbb{N}$, the monotonicity of $\omega$ gives
			\begin{equation}\label{wj111}
				\int_{\lambda_1^{k+1}\xi}^{\lambda_1^k\xi}
				\frac{\omega(t)}{t^{\alpha_0+1}}dt
				\leqslant\omega(\lambda_1^k\xi)
				\int_{\lambda_1^{k+1}\xi}^{\lambda_1^k\xi}
				t^{-\alpha_0-1}dt
				={C_{5}}\frac{\omega(\lambda_1^k\xi)}{\lambda_1^{k\alpha_0}},\quad {C_{5}}:=\frac{\lambda_1^{-\alpha_0}-1}{\alpha_0\xi^{\alpha_0}}>0.
			\end{equation}
			Therefore, summing \eqref{wj111} over $k \in \mathbb{N}$ and utilizing Lemma \ref{prop21} yields
			\begin{equation}\label{wj222}
				\sum_{k=0}^{\infty}
				\frac{\widehat{\omega}(\lambda_1^k\xi)}{\lambda_1^{k\alpha_0}} 
				\geqslant \frac{1}{2}\sum_{k=0}^{\infty}
				\frac{\omega(\lambda_1^k\xi)}{\lambda_1^{k\alpha_0}} 
				\geqslant \frac{1}{2{C_{5}}}\sum_{k=0}^\infty \int_{\lambda_1^{k+1}\xi}^{\lambda_1^k\xi} \frac{\omega(t)}{t^{\alpha_0+1}}dt 
				= \frac{1}{2{C_{5}}} \int_0^\xi \frac{\omega(t)}{t^{\alpha_0+1}}dt 
				= +\infty.
			\end{equation}
			Combining \eqref{xsdsgj1}, \eqref{wj222}, and the fact that $N(y)\to+\infty$ as $y\to0^+$, we establish the assertion.
			This contradicts the differentiability of the $C^1$ mapping $\Psi$ at $(\xi,0)$. Therefore, $F$ cannot be $C^1$ linearized near $O$.
		\end{proof}

		Lemma \ref{pro:regularity} and Lemma \ref{tdlm} complete the proof of Theorem \ref{TH1} for $ 0<\lambda_1<\lambda_2<1 $.

		It remains to explain the modifications required for the other possible signs of the eigenvalues. If both eigenvalues are negative, the preceding construction applies after replacing $\lambda_i$ by $|\lambda_i|$ in the estimates and multiplying $R$ by $-1$. Suppose now that $0<|\lambda_1|<|\lambda_2|<1$ and $\lambda_1\lambda_2<0$. The crux is to introduce an oscillatory factor to cancel out the alternating signs arising from the eigenvalues. Since the underlying mechanism is identical to the one above, we only provide a sketch of the proof for brevity. We replace \eqref{eq:3.3} by
		\[
		R(x_1,x_2)
		:=
		\begin{cases}
			\operatorname{sgn}(\lambda_1)
			\Omega(|x_1|)
			\operatorname{sgn}(x_1)
			\cos\left(
			\displaystyle
			\frac{\pi\log|x_1|}
			{|\log|\lambda_1||}
			\right)
			K\left(\dfrac{x_2}{x_1}\right),
			& x_1\neq0,\\[2ex]
			0,
			& x_1=0.
		\end{cases}
		\]
		The proof that $F\in C_{1,\omega}$ proceeds analogously, and the proof of Lemma~\ref{pro:regularity} applies with only minor notational changes. Choose $\delta>0$ sufficiently small such that
		\[
		\cos\left(
		\frac{\pi\log t}{|\log|\lambda_1||}
		\right)
		\geqslant
		\frac{1}{2},
		\quad \forall 1\leqslant t\leqslant1+\delta.
		\]
		Fix a sufficiently large $m\in\mathbb N$, set $\xi=|\lambda_1|^{2m}$, and write, as before, $F^k(\xi,y)=(x_k,\lambda_2^ky)$. We retain the permissible non-escape time $N(y)$ defined in \eqref{cutoffN}, with $\lambda_1$ and $\lambda_2$ replaced by $|\lambda_1|$ and $|\lambda_2|$, respectively. The same summation estimate as above, together with \(\operatorname{supp}\chi\subset[-2,2]\), gives by a standard bootstrap argument
		\[
		1
		\leqslant
		\frac{x_k}{\lambda_1^k\xi}
		\leqslant
		1+\delta,
		\quad\forall
		k\in\mathbb N,
		\]
		whenever \(y>0\) is sufficiently small.
		Indeed, as long as these inequalities hold, the cutoff factor can remain nonzero for only a uniformly bounded number of iterates beyond \(N(y)\), while the sum of all the corresponding normalized increments tends to zero as \(y\to0^+\). Since \(\xi=|\lambda_1|^{2m}\) and $\lambda_1\lambda_2<0$, we have
		\[
		\left(
		\frac{\operatorname{sgn}(\lambda_2)}
		{\operatorname{sgn}(\lambda_1)}
		\right)^k
		\cos\left(
		\frac{\pi\log|x_k|}
		{|\log|\lambda_1||}
		\right)
		=
		\cos\left(
		\frac{\pi}{|\log|\lambda_1||}
		\log\frac{x_k}{\lambda_1^k\xi}
		\right)
		\geqslant
		\frac12.
		\]
		Thus every nonzero normalized increment is nonnegative. Moreover,
		\[\left|\frac{\lambda_2^ky}{x_k}\right| = \frac{|\lambda_1|^{-k\alpha_0}y} {\xi\dfrac{x_k}{\lambda_1^k\xi}} \to+\infty \quad\text{as }k\to+\infty,\]
		so the cutoff factor eventually vanishes. For \(0\leqslant k\leqslant N(y)\), however, the cutoff factor equals one. Repeating the computation leading to \eqref{xsdsgj1}, we therefore obtain
		\[\frac{\Psi_1(\xi,y)-\Psi_1(\xi,0)}{y} \geqslant \frac{1}{2|\lambda_1|} \sum_{k=0}^{N(y)} \frac{\widehat{\omega}(|\lambda_1|^k\xi)} {|\lambda_1|^{k\alpha_0}}. \]
		When condition \eqref{kjh} fails, the series on the right-hand side diverges. Since \(N(y)\to+\infty\) as \(y\to0^+\), it follows that
		\[\frac{\Psi_1(\xi,y)-\Psi_1(\xi,0)}{y} \to+\infty \quad\text{as }y\to0^+,\]
		contradicting the differentiability of \(\Psi\) at \((\xi,0)\). This completes the proof for eigenvalues of opposite signs. We remark that the construction presented above is by no means unique. Indeed, a simpler variant can be obtained along the lines of \eqref{eq:jordan-counterexample} and~\eqref{eq:perturbation} in Section~\ref{PTJORDAN} by merely inserting appropriate powers of $-1$ into \eqref{eq:perturbation}, suitably adapted to the present setting.

		The proof of Theorem \ref{TH1} is now complete.

	\end{proof}

\subsection{Proof of Theorem \ref{TH3}}
We first consider Item~\ref{TH3-1}, where $0<|\lambda_1|<|\lambda_2|<1$. If $I_{\alpha_1}(\omega)<+\infty$, we estimate the products defining $\Lambda^{-n}DF^n$ directly. Otherwise, we first linearize the second coordinate and straighten an invariant curve. The off-diagonal derivative of the transformed mapping vanishes on the vertical axis, allowing both contraction rates to enter the estimate. Optimality follows from triangular mappings with the prescribed linear part.
	
	Recall that
	\[
	\alpha_1=\frac{\alpha_0}{1-\alpha_0}>\alpha_0>0,
	\quad
	|\lambda_1|=|\lambda_2|^{1+\alpha_1},
	\quad
	\left|\frac{\lambda_2}{\lambda_1}\right|
	=|\lambda_2|^{-\alpha_1}=|\lambda_1|^{-\alpha_0}.
	\]

	\begin{lemma}\label{Lemma25}
		For a concave modulus of continuity $\nu$ satisfying~\eqref{kjh}, define
\begin{alignat}{2}
	\label{eq:anis-operators}
	B_\nu(t)&=\int_0^{t^{1/\alpha_0}}\frac{\nu(s)}{s^{1+\alpha_0}}ds,
	&\quad H_\nu(t)&=t^{\alpha_1}\int_{t^{1/\alpha_0}}^t\frac{\nu(s)}{s^{1+\alpha_1}}ds,\\
	\label{eq:anis-operators-rest}
	K_\nu(t)&=B_\nu(t)+H_\nu(t),
	&\quad R_\nu(t)&=t^{\alpha_1}\int_t^1\frac{\nu(s)}{s^{1+\alpha_1}}ds.
\end{alignat}
		Both branches in~\eqref{eq:anisotropic-optimal-modulus} define concave moduli of continuity near zero. For $w=\omega_{\rm B}$, defined in~\eqref{omegaA}, there is a constant $C_{6}\geqslant1$ such that
		\begin{equation}\label{lxmgj1}
			\omega(t)\leqslant w(t)\leqslant K_\omega(t),\quad
			K_w(t)\leqslant C_{6}K_\omega(t),\quad
			R_w(t)\leqslant C_{6}\bigl(w(t)+R_\omega(t)\bigr).
		\end{equation}
		For every fixed $c\in(0,1)$ and $\gamma>0$, comparison on geometric intervals gives
		\begin{equation}\label{eq:anis-geometric-integrals}
			\sum_{n=0}^{\infty}c^{-n\gamma}\nu(c^nt)
			\asymp t^\gamma\int_0^t\frac{\nu(s)}{s^{1+\gamma}}ds
		\end{equation}
		whenever either side is finite, and
		\begin{equation}\label{lxmgj2}
			\sum_{n=0}^{\infty}\left|\frac{\lambda_2}{\lambda_1}\right|^n\min\{\nu(|\lambda_2|^nt),\nu(|\lambda_1|^n)\}
			\asymp K_\nu(t).
		\end{equation}
	\end{lemma}
	\begin{proof}
		Set $d=t^{1/\alpha_0}$. Concavity gives $\omega\leqslant w$, and comparison of the kernels on $(0,d)$ and $(d,t)$ gives $w\leqslant K_\omega$. By Tonelli's theorem,
		\[
		\int_0^1\frac{w(s)}{s^{1+\alpha_0}}ds
		=\frac1{\alpha_0}\int_0^1\frac{\omega(s)}s(s^{-\alpha_0}-1)ds<+\infty.
		\]
		Integration by parts, using $\alpha_1/\alpha_0=1+\alpha_1$, now yields
		\[
		\begin{aligned}
		K_w(t)&=\frac{H_\omega(t)}{\alpha_1}+\frac{B_\omega(t)}{\alpha_0}
		-\frac{w(d)}t-\frac{w(t)}{\alpha_1},
		\\
		R_w(t)&=\frac{w(t)-t^{\alpha_1}w(1)+R_\omega(t)}{\alpha_1}.
		\end{aligned}
		\]
		The boundary terms at zero vanish by~\eqref{kjh}, proving~\eqref{lxmgj1}.
		
		When $I_{\alpha_1}(\omega)<+\infty$, the formula $\omega_{\rm A}(t)=\int_0^1\omega(tu)u^{-1-\alpha_1}du$ gives concavity. For the other branch, put $V=B_\omega+H_\omega+R_\omega$. Then
		\[
		\begin{aligned}
		V'(t)&=\alpha_1 t^{\alpha_1-1}\int_d^1\frac{\omega(s)}{s^{1+\alpha_1}}ds>0,\\
		V''(t)&=\alpha_1(\alpha_1-1)t^{\alpha_1-2}\int_d^1\frac{\omega(s)}{s^{1+\alpha_1}}ds
		-\frac{\alpha_1}{\alpha_0}t^{-3}\omega(d).
		\end{aligned}
		\]
		For $\alpha_1\leqslant1$ this is nonpositive. For $\alpha_1>1$, concavity implies
		$\omega(s)\leqslant(s/d)\omega(d)$ for $s\geqslant d$, and hence
		$\int_d^1\omega(s)s^{-1-\alpha_1}ds\leqslant\omega(d)d^{-\alpha_1}/(\alpha_1-1)$; thus $V''(t)<0$ as well. Finally, $V(t)\to0$ follows by dominated convergence from
		\[
		t^{\alpha_1}\int_d^1\frac{\omega(s)}{s^{1+\alpha_1}}ds
		=\int_d^1\left(\frac ds\right)^{\alpha_1-\alpha_0}
		\frac{\omega(s)}{s^{1+\alpha_0}}ds.
		\]
		To prove~\eqref{eq:anis-geometric-integrals}, bound $\nu$ on each interval $[c^{n+1}t,c^nt]$ by its endpoint values. For~\eqref{lxmgj2}, split the sum at
		$N=\lfloor\log t/\log(|\lambda_1|/|\lambda_2|)\rfloor$. For $n\leqslant N$ the minimum is $\nu(|\lambda_2|^nt)$, and for $n>N$ it is $\nu(|\lambda_1|^n)$. Comparison on geometric intervals gives $H_\nu$ and $B_\nu$, respectively, with fixed changes of scale absorbed by concavity.
	\end{proof}
	
	We now prove the upper bound. Write $F(x)=\Lambda x+f(x)$. On a sufficiently small convex invariant neighborhood $U$, fix $C_{7}\geqslant1$ such that
	\[
	|Df(x)-Df(y)|\leqslant C_{7}\omega(|x-y|),\quad
	|F^n(x)|\leqslant q^n|x|,\quad |\lambda_2|<q<1.
	\]
	The ordinary Dini condition follows from~\eqref{kjh}. Therefore
	\[
	\prod_{j=0}^{\infty}\left(1+\frac{C_{7}}{|\lambda_2|}
	\omega(q^j\operatorname{diam}U)\right)<+\infty.
	\]
	The chain rule, also applied to every orbit segment, gives a constant $C_{8}\geqslant1$ satisfying
	\begin{equation}\label{eq:anis-exact-rates}
		|DF^n(x)|\leqslant C_{8}|\lambda_2|^n,\quad
		|F^n(x)-F^n(y)|\leqslant C_{8}|\lambda_2|^n|x-y|.
	\end{equation}
	Suppose first that $I_{\alpha_1}(\omega)<+\infty$, and put
	\[
	P_n(x)=\Lambda^{-n}DF^n(x),\quad
	E_n(x)=\Lambda^{-(n+1)}Df(F^n(x))\Lambda^n.
	\]
	Then $P_0=I$ and $P_{n+1}=(I+E_n)P_n$. For a constant $C_{9}\geqslant1$, concavity and~\eqref{eq:anis-exact-rates} give
\begin{equation}\label{eq:anis-strong-e}
	\sup_U|E_n|\leqslant C_{9}\left|\frac{\lambda_2}{\lambda_1}\right|^n\omega(|\lambda_2|^n), \quad |E_n(x)-E_n(y)|\leqslant C_{9}\left|\frac{\lambda_2}{\lambda_1}\right|^n\omega(|\lambda_2|^n|x-y|).
\end{equation}
	The first bound is summable. With
	$C_{10}=\exp(\sum_{n\geqslant0}\sup_U|E_n|)$,
	subtracting the two products gives
	\begin{equation}\label{eq:anis-product-difference}
		|P_n(x)-P_n(y)|\leqslant C_{10}^{2}
		\sum_{j=0}^{n-1}|E_j(x)-E_j(y)|.
	\end{equation}
	Uniform convergence of the products, followed by integration on segments from $O$, gives $\Psi=\lim_{n\to\infty}\Lambda^{-n}F^n$ in $C^1$. Since $D\Psi(O)=I$ and $\Psi(F(x))=\Lambda\Psi(x)$, this is a local linearization.
	Equations~\eqref{eq:anis-strong-e}, \eqref{eq:anis-product-difference} and~\eqref{eq:anis-geometric-integrals} yield, for some $C_{11}>0$,
	\[
	|D\Psi(x)-D\Psi(y)|\leqslant C_{11}\omega_{\rm A}(|x-y|).
	\]
	This proves the first branch of~\eqref{eq:anisotropic-optimal-modulus}.
	
	Now suppose that $I_{\alpha_1}(\omega)=+\infty$. Note that \eqref{eq:anis-exact-rates} still holds. Subtracting the products defining $DF^n$ gives, for some $C_{12}\geqslant1$,
	\begin{equation}\label{eq:anis-iterate-difference}
		|DF^n(x)-DF^n(y)|\leqslant C_{12}|\lambda_2|^n
		\sum_{j=0}^{n-1}\omega(|\lambda_2|^j|x-y|)
		\leqslant C_{12}^{2}|\lambda_2|^nw(|x-y|).
	\end{equation}
	The derivative increments of $\psi_{2,n}=\lambda_2^{-n}\pi_2F^n$ are bounded by a summable multiple of $\omega(|\lambda_2|^n)$. By~\eqref{eq:anis-iterate-difference}, their differences at $x,y$ are bounded by
	\[
	C_{13}\bigl(\omega(|\lambda_2|^n|x-y|)+\omega(|\lambda_2|^n)w(|x-y|)\bigr),
	\]
	for some $C_{13}\geqslant1$. Thus $\psi_2=\lim_{n\to\infty}\psi_{2,n}\in C_{1,w}$. Set
	\[
	T(x):=\begin{pmatrix}x_1\\\psi_2(x)\end{pmatrix},
	\quad H:=T\circ F\circ T^{-1}.
	\]
	Then $DT(O)=I$ and
	\[
	H(x,y)=\begin{pmatrix}\lambda_1x+h(x,y)\\\lambda_2y\end{pmatrix},
	\quad h\in C_{1,w},\quad Dh(O)=0.
	\]
	We use throughout that a local $C^1$ diffeomorphism and its inverse have the same concave derivative modulus, by the inverse derivative formula and local Lipschitz continuity.
	
	A smooth cutoff extends $h$, unchanged near $O$, so that $|Dh|_\infty$ is arbitrarily small and $h(x,y)=0$ for $|y|\geqslant\rho$. For some $C_{14}\geqslant1$,
	\begin{equation}\label{eq:anis-cutoff}
		|Dh(z)-Dh(z')|\leqslant C_{14}w(|z-z'|),\quad
		|Dh(z)|\leqslant C_{14}w(|z|).
	\end{equation}
	Choose $\ell>0$ small and require
	\[
	(|\lambda_1|+|Dh|_\infty)\ell+|Dh|_\infty\leqslant|\lambda_2|\ell.
	\]
	Set $g(y)=0$ for $|y|\geqslant\rho$, and extend successively toward zero by
	\begin{equation}\label{eq:anis-graph-recursion}
		g(\lambda_2 y)=\lambda_1g(y)+h(g(y),y).
	\end{equation}
	The cutoff ensures agreement of the values and first derivatives at the annular boundaries. Induction gives $|g(y)|\leqslant\ell|y|$ and $|g'|\leqslant\ell$ away from zero.
	
	Let $w_*(t)=w(\min\{t,\rho\})$ and define the concave modulus
	\begin{equation}\label{eq:anis-graph-modulus}
		\nu(t)=\sum_{j=0}^{\infty}|\lambda_2|^{\alpha_1 j}w_*(|\lambda_2|^{-j}t).
	\end{equation}
	Differentiating~\eqref{eq:anis-graph-recursion} gives
	\[
	g'(y)=\frac{\lambda_1+\partial_1h(g(y/\lambda_2),y/\lambda_2)}{\lambda_2}g'(y/\lambda_2)
	+\frac{\partial_2h(g(y/\lambda_2),y/\lambda_2)}{\lambda_2}.
	\]
	Since $w$ is Dini, products of the coefficients of $g'$ over $n$ backward iterates are bounded by $C_{15}|\lambda_2|^{\alpha_1 n}$, where $C_{15}\geqslant1$. Subtracting this recurrence at two points and using~\eqref{eq:anis-cutoff} yield
	\[
|g'(y)-g'(z)|\leqslant C_{16}\nu(|y-z|),\quad
		|g'(y)|\leqslant C_{16}\nu(|y|)
\]
	for some $C_{16}>0$. To justify the value at zero, first compare the recurrence for $g'(y)$, $y\neq0$, with the zero sequence at the fixed point. The terminal term vanishes by boundedness of $g'$, and the resulting bound $|g'(y)|\leqslant C_{16}\nu(|y|)$ shows that $g$ extends as a $C^1$ function with $g'(0)=0$. Subtracting the recurrence along two arbitrary backward orbits gives the first estimate. Thus $g\in C_{1,\nu}$. Comparing~\eqref{eq:anis-graph-modulus} with integrals gives
	\begin{equation}\label{eq:anis-graph-integral}
		\nu(t)\leqslant C_{17}\bigl(t^{\alpha_1}+w(t)+R_w(t)\bigr),
	\end{equation}
	for some $C_{17}\geqslant1$. Also $\nu$ satisfies~\eqref{kjh}, since integration of~\eqref{eq:anis-graph-modulus} gives a convergent geometric sum with ratio $|\lambda_2|^{\alpha_1-\alpha_0}<1$.
	
	To straighten this graph, set
	\[
	S(x,y):=\begin{pmatrix}x-g(y)\\y\end{pmatrix},
	\quad G:=S\circ H\circ S^{-1}.
	\]
	By~\eqref{eq:anis-graph-recursion},
	\[
G(x,y)=\begin{pmatrix}
\lambda_1x+h(x+g(y),y)-h(g(y),y)\\
\lambda_2y
\end{pmatrix}.
\]
	Write
	\[
	DG=\begin{pmatrix}\lambda_1+A&B\\0&\lambda_2\end{pmatrix},\quad
	A(x,y)=\partial_1h(x+g(y),y),
	\]
	and
	
\begin{equation}\label{eq:anis-flat-cancellation}
	B(x,y)=[\partial_1h(x+g(y),y)-\partial_1h(g(y),y)]g'(y)+\partial_2h(x+g(y),y)-\partial_2h(g(y),y).
\end{equation}
	Consequently, for a constant $C_{18}\geqslant1$,
	\[
	|A(x,y)|\leqslant C_{18}w(|(x,y)|),\quad
	|B(x,y)|\leqslant C_{18}w(|x|).
	\]
	As in~\eqref{eq:anis-exact-rates}, these bounds imply
	\begin{equation}\label{eq:anis-flat-rates}
		\begin{aligned}
		|DG^n(z)|&\leqslant C_{19}|\lambda_2|^n,\quad
		|\pi_1G^n(z)|\leqslant C_{19}|\lambda_1|^n,\\
		|G^n(z)-G^n(z')|&\leqslant C_{19}|\lambda_2|^n|z-z'|.
		\end{aligned}
	\end{equation}
	on a small invariant neighborhood, for some $C_{19}\geqslant1$. For the estimate on $\pi_1G^n$, use
	\[
	G_1(x,y)=x\left(\lambda_1+\int_0^1\partial_1h(g(y)+ux,y)du\right)
	\]
	and the Dini condition on $w$.
	
	For $t=|z-z'|$, subtraction in~\eqref{eq:anis-flat-cancellation} yields a constant $C_{20}\geqslant1$ such that
\begin{equation}\label{eq:anis-flat-difference}
	\begin{aligned}
		|A(G^nz)-A(G^nz')|&\leqslant C_{20}w(|\lambda_2|^nt),\\
		|B(G^nz)-B(G^nz')|&\leqslant C_{20}
		\min\{w(|\lambda_2|^nt),w(|\lambda_1|^n)\}\\
		&\quad+C_{20}w(|\lambda_1|^n)\nu(|\lambda_2|^nt).
	\end{aligned}
\end{equation}
	The second estimate uses the cancellation in~\eqref{eq:anis-flat-cancellation}: the factor multiplying the difference of $g'$ is bounded by a multiple of $w(|\lambda_1|^n)$, while the remaining terms are bounded both by $w(|\lambda_2|^nt)$ and by $w(|\lambda_1|^n)$.
	
	For $G$, the matrices $E_n$ in the product argument satisfy
	\[
	\sup|E_n|\leqslant C_{21}\left(w(|\lambda_2|^n)+\left|\frac{\lambda_2}{\lambda_1}\right|^nw(|\lambda_1|^n)\right),
	\]
	for some $C_{21}\geqslant1$. This bound is summable by~\eqref{kjh} for $w$. Repeating the ordered-product argument in~\eqref{eq:anis-product-difference}, and using~\eqref{eq:anis-flat-difference} and~\eqref{lxmgj2}, give
	\begin{equation}\label{eq:anis-flat-output}
		|D\Psi(z)-D\Psi(z')|\leqslant C_{22}\bigl(K_w(t)+\nu(t)\bigr)
	\end{equation}
	for the canonical linearization $\Psi$ of $G$ and some $C_{22}>0$. Indeed, concavity gives $|\lambda_2/\lambda_1|^nw(|\lambda_1|^n)\geqslant w(|\lambda_2|^nt)$ when $|\lambda_1|^n<|\lambda_2|^nt$ and $t\leqslant1$, so $\sum_{n=0}^{\infty}w(|\lambda_2|^nt)$ is bounded by a multiple of $K_w(t)$. For the remaining term, monotonicity of $\nu$ gives
\[
\sum_{n=0}^{\infty}\left|\frac{\lambda_2}{\lambda_1}\right|^n w(|\lambda_1|^n)\nu(|\lambda_2|^nt) \leqslant \nu(t)\sum_{n=0}^{\infty} \left|\frac{\lambda_2}{\lambda_1}\right|^nw(|\lambda_1|^n).
\]
	The last series is finite by~\eqref{kjh} for $w$ and is independent of $t$.
	
	The original conjugacy is $\Phi=\Psi\circ S\circ T$. Composition and~\eqref{lxmgj1}, \eqref{eq:anis-graph-integral} give, for some $C_{23}>0$,
	\[
	|D\Phi(x)-D\Phi(y)|\leqslant C_{23}
	\bigl(K_\omega(t)+R_\omega(t)+t^{\alpha_1}\bigr),\quad t=|x-y|.
	\]
	Since $I_{\alpha_1}(\omega)=+\infty$, we have $t^{\alpha_1}=o(R_\omega(t))$. By~\eqref{eq:anis-operators} and \eqref{eq:anis-operators-rest}, this proves the second branch of~\eqref{eq:anisotropic-optimal-modulus}. The same modulus holds for $\Phi^{-1}$.
	
	For optimality, we construct examples that give the necessary lower bounds. To include opposite signs of the eigenvalues, put $\sigma=\operatorname{sgn}(\lambda_1/\lambda_2)$ and define, for $0<c<1$ and $u\neq0$,
	\begin{equation}\label{eq:anis-phase}
		\vartheta_c(u)=
		\begin{cases}1,&\sigma=1,\\
			\cos\bigl(\pi\log|u|/\log c\bigr),&\sigma=-1.
		\end{cases}
	\end{equation}
	Then $\vartheta_c(\pm cu)=\sigma\vartheta_c(u)$, and, for a constant $C_{24}\geqslant1$,
	$|\vartheta_c^{(j)}(u)|\leqslant C_{24}|u|^{-j}$ for $j=0,1,2$ and $c\in\{|\lambda_1|,|\lambda_2|\}$.
	For the averaged modulus $\widehat\omega$ from Lemma~\ref{prop21}, the product $\vartheta_c(u)\widehat\omega(|u|)$ extends by zero and has $\omega$-modulus. This follows, as in Lemma~\ref{pro:regularity}, by using its size when the two points are far apart relative to their distance from zero, and its derivative bound by a multiple of $\omega(|u|)/|u|$ otherwise.
	
	\begin{lemma}\label{Lemma26}
		There is a triangular $C_{1,\omega}$ mapping with linear part $\Lambda$ such that every normalized $C_{1,\eta}$ linearization satisfies $\omega_{\rm A}=\mathcal O(\eta)$ if $I_{\alpha_1}(\omega)<+\infty$, and $R_\omega=\mathcal O(\eta)$ otherwise.
	\end{lemma}
	\begin{proof}
		Fix $\varepsilon>0$ and put
		\[
		P(y)=\int_0^y\vartheta_{|\lambda_2|}(s)\widehat\omega(|s|)ds,\quad
		F_R(x,y):=\begin{pmatrix}\lambda_1x+\varepsilon P(y)\\\lambda_2y\end{pmatrix}.
		\]
		Then $F_R\in C_{1,\omega}$ and $DF_R(O)=\Lambda$, so it is a local contraction and diffeomorphism after restriction.
		If $\Phi\in C_{1,\eta}$ is a normalized linearization, the inverse image of the slow coordinate axis is an invariant $C_{1,\eta}$ graph $x=g(y)$, with $g(0)=g'(0)=0$. With $v=g'$,
		\begin{equation}\label{xsdsgj2}
			\lambda_2v(\lambda_2y)=\lambda_1v(y)+\varepsilon\vartheta_{|\lambda_2|}(y)\widehat\omega(|y|).
		\end{equation}
		Choose $u=|\lambda_2|^{2m}>0$ in its domain. Then $\vartheta_{|\lambda_2|}(\lambda_2^ju)=\sigma^j$, and iteration of~\eqref{xsdsgj2} gives
		\begin{equation}\label{zygj1-1}
			\left(\frac{\lambda_1}{\lambda_2}\right)^{-n}v(\lambda_2^nu)
			=v(u)+\frac{\varepsilon}{\lambda_1}
			\sum_{j=0}^{n-1}|\lambda_2|^{-\alpha_1 j}\widehat\omega(|\lambda_2|^ju).
		\end{equation}
		There is $C_{25}>0$ with $|v(y)|\leqslant C_{25}\eta(|y|)$.
		If $I_{\alpha_1}(\omega)<+\infty$, the right-hand side has a finite limit. If it is nonzero, then $t^{\alpha_1}=\mathcal O(\eta(t))$, and $\omega_{\rm A}(t)=o(t^{\alpha_1})$. If it is zero, then
		\[
		|v(\lambda_2^nu)|=\frac{\varepsilon}{|\lambda_1|}\sum_{k=0}^{\infty}|\lambda_2|^{-\alpha_1 k}\widehat\omega(|\lambda_2|^{n+k}u)
		\asymp\omega_{\rm A}(|\lambda_2|^nu).
		\]
		In either case~\eqref{eq:anis-geometric-integrals} and interpolation between consecutive geometric scales imply $\omega_{\rm A}=\mathcal O(\eta)$.
		If $I_{\alpha_1}(\omega)=+\infty$, the sum in~\eqref{zygj1-1} diverges and eventually dominates $v(u)$. Geometric comparison gives, for some $C_{26}>0$,
		\[
		R_\omega(|\lambda_2|^nu)\leqslant C_{26}|v(\lambda_2^nu)|
		\leqslant C_{26}C_{25}\eta(|\lambda_2|^nu).
		\]
		Here the contribution from $[u,1]$ is of order $(|\lambda_2|^nu)^{\alpha_1}$ and is absorbed for large $n$. To pass from geometric scales to all small $t$, use
		\[
		R_\omega(|\lambda_2|^{n+1}u)\geqslant |\lambda_2|^{\alpha_1}R_\omega(t),
		\quad |\lambda_2|^{n+1}u<t\leqslant |\lambda_2|^nu,
		\]
		and monotonicity of $\eta$.
	\end{proof}
	
	The remaining two lower bounds use the following identity. If
	$F_2(x,y)=\lambda_2y$ and $F_1(0,y)=0$, then every normalized $C^1$ linearization satisfies
	\begin{equation}\label{eq:TH3_Psi_Phi}
		\Psi_1(x,y):=\lim_{n\to\infty}\lambda_1^{-n}\pi_1F^n(x,y)
		=\Phi_1(x,y)-\Phi_1(0,y).
	\end{equation}
	Indeed, the first-column derivative of the conjugacy equation gives
	\[
	\lambda_1^{-n}\partial_1\pi_1F^n(x,y)
	=\frac{\partial_1\Phi_1(x,y)}{\partial_1\Phi_1(F^n(x,y))}.
	\]
	Integration from $0$ to $x$ proves~\eqref{eq:TH3_Psi_Phi}, locally uniformly. In particular $\Psi_1\in C_{1,\eta}$ whenever $\Phi\in C_{1,\eta}$.
	
	To obtain the lower bound $H_\omega=\mathcal O(\eta)$, take the even cutoff $\chi$ from Section~\ref{SEC3} and define, for $x\neq0$,
	\[
	F_H(x,y):=\begin{pmatrix}
	\lambda_1\bigl[x+\varepsilon\operatorname{sgn}(x)\Omega(|y|)\chi(\frac{y}{x})\bigr]\\
	\lambda_2y
	\end{pmatrix},
	\]
	and extend it continuously to $x=0$. The estimates in Lemma~\ref{pro:regularity} give $F_H\in C_{1,\omega}$: on the transition cone, the terms involving $\Omega(|y|)/|x|$ have size bounded by a multiple of $\omega(|x|)$ and derivative bounded by a multiple of $\omega(|x|)/|x|$; the remaining term $\operatorname{sgn}(y)\omega(|y|)$ has $\omega$-modulus by concavity. Both axes are linear and $DF_H(x,0)=\Lambda$.
	The argument for~\eqref{eq:anis-flat-rates}, \eqref{eq:anis-flat-difference}, and~\eqref{eq:anis-flat-output}, with $g=0$ and $w=\omega$, gives convergence of the normalized derivatives and $\partial_2\Psi_1(\xi,0)=0$.
	Fix a small $0<\xi<1$ and start at $(\xi,\xi t)$. The first component keeps the sign of $\lambda_1^n\xi$, and its absolute value is at least $|\lambda_1|^n\xi$. For
	$N(t)=\lfloor\log t/\log(|\lambda_1|/|\lambda_2|)\rfloor$, the cutoff is one for $0\leqslant n\leqslant N(t)$; all normalized increments are nonnegative. Hence
\begin{equation}\label{eq:anis-head-lower}
	\Psi_1(\xi,\xi t)-\xi \geqslant \varepsilon\sum_{n=0}^{N(t)}|\lambda_1|^{-n}\Omega(|\lambda_2|^n\xi t) \geqslant \frac{\varepsilon\xi^2t}{2}\sum_{n=0}^{N(t)}\left|\frac{\lambda_2}{\lambda_1}\right|^n\omega(|\lambda_2|^nt).
\end{equation}
	The last inequality uses Lemma~\ref{prop21} and $\omega(\xi s)\geqslant\xi\omega(s)$.
	For any normalized $C_{1,\eta}$ linearization,~\eqref{eq:TH3_Psi_Phi} and $\partial_2\Psi_1(\xi,0)=0$ bound the left-hand side of~\eqref{eq:anis-head-lower} by $C_{27}\xi t\eta(\xi t)$ for some $C_{27}>0$. Geometric comparison, as in~\eqref{lxmgj2}, yields
	\begin{equation}\label{eq:anis-head-necessary}
		H_\omega(t)\leqslant C_{28}\eta(t)
	\end{equation}
	with a constant $C_{28}>0$.
	
	\begin{lemma}\label{Lemma27}
		There is a $C_{1,\omega}$ triangular mapping with linear part $\Lambda$ for which every normalized $C_{1,\eta}$ linearization satisfies $B_\omega=\mathcal O(\eta)$.
	\end{lemma}
	\begin{proof}
		For $x\neq0$ put
		\[
		F_B(x,y):=\begin{pmatrix}
		\lambda_1\bigl[x+\varepsilon y\vartheta_{|\lambda_1|}(x)\widehat\omega(|x|)\chi(\frac{y}{x})\bigr]\\
		\lambda_2y
		\end{pmatrix},
		\]
		and extend it continuously to $x=0$. This perturbation is $\varepsilon\lambda_1\vartheta_{|\lambda_1|}(x)$ times~\eqref{eq:3.3}; the derivative bounds in~\eqref{eq:anis-phase} and Lemma~\ref{pro:regularity} give $F_B\in C_{1,\omega}$.
		Choose $\xi=|\lambda_1|^{2m}>0$ sufficiently small. By~\eqref{eq:anis-geometric-integrals},
		\begin{equation}\label{gjgj1}
			L_\xi:=\sum_{n=0}^{\infty}\left|\frac{\lambda_2}{\lambda_1}\right|^n\omega(|\lambda_1|^n\xi)<+\infty.
		\end{equation}
		Write $x_n=\pi_1F_B^n(\xi,y)$ and $z_n=x_n/(\lambda_1^n\xi)$; the second coordinate of the orbit is $\lambda_2^ny$. Put
		\[
		m_n=\begin{cases}1,&\sigma=1,\\
			\cos(\pi\log z_n/\log |\lambda_1|),&\sigma=-1.
		\end{cases}
		\]
		While $z_n$ is near one, the normalized recurrence is
		\[
		z_{n+1}=z_n+\frac{\varepsilon y}{\xi}\left|\frac{\lambda_2}{\lambda_1}\right|^n m_n
		\widehat\omega(|\lambda_1|^n\xi z_n)\chi(\lambda_2^ny/x_n).
		\]
		Concavity gives $\widehat\omega(|\lambda_1|^n\xi z_n)\leqslant z_n\widehat\omega(|\lambda_1|^n\xi)$ for $z_n\geqslant1$. Choose $y>0$ so small that
		$\exp(\varepsilon yL_\xi/\xi)$ lies in an interval around one on which $m_n\geqslant1/2$. Induction and the product bound then give a constant $C_{29}>0$ such that
		\begin{equation}\label{gjgj2}
			1\leqslant z_n\leqslant\exp(\varepsilon yL_\xi/\xi)
			\leqslant1+C_{29}y\leqslant2,\quad n\geqslant0.
		\end{equation}
		Using $\widehat\omega'(s)\leqslant\omega(s)/s$ and~\eqref{gjgj2}, one obtains, for some $C_{30}>0$,
		\begin{equation}\label{gjgj3}
			\sum_{n=0}^{\infty}\left|\frac{\lambda_2}{\lambda_1}\right|^n
			\left|m_n\widehat\omega(|\lambda_1|^n\xi z_n)-\widehat\omega(|\lambda_1|^n\xi)\right|
			\leqslant C_{30}y.
		\end{equation}
		Here the cosine difference is bounded by a multiple of $|z_n-1|$, and~\eqref{gjgj1} controls its weighted sum.
		By~\eqref{eq:TH3_Psi_Phi}, the quotient
		\begin{equation}\label{gjgj4}
			Q_\xi(y):=\frac{\Psi_1(\xi,y)-\xi}{y}
			=\varepsilon\sum_{n=0}^{\infty}\left|\frac{\lambda_2}{\lambda_1}\right|^n m_n\widehat\omega(|\lambda_1|^n\xi z_n)
			\chi(\lambda_2^ny/x_n)
		\end{equation}
		is well defined. Dominated convergence,~\eqref{gjgj1}, and~\eqref{gjgj3} give
		\[
\partial_2\Psi_1(\xi,0)=\varepsilon\sum_{n=0}^{\infty}\left|\frac{\lambda_2}{\lambda_1}\right|^n\widehat\omega(|\lambda_1|^n\xi).
\]
		Let $M(y)=\min\{n\in \mathbb{N}:|\lambda_2/\lambda_1|^ny\geqslant4\xi\}$. By~\eqref{gjgj2}, the cutoff in~\eqref{gjgj4} vanishes for $n\geqslant M(y)$. Consequently,
		\begin{equation}\label{gjgj6}
			\partial_2\Psi_1(\xi,0)-Q_\xi(y)
			\geqslant\varepsilon\sum_{n=M(y)}^{\infty}\left|\frac{\lambda_2}{\lambda_1}\right|^n\widehat\omega(|\lambda_1|^n\xi)
			-\varepsilon C_{30}y.
		\end{equation}
		Geometric comparison and Lemma~\ref{prop21} give a constant $C_{31}\geqslant1$ such that
		\[
\sum_{n=M(y)}^{\infty}\left|\frac{\lambda_2}{\lambda_1}\right|^n\widehat\omega(|\lambda_1|^n\xi)
			\geqslant C_{31}^{-1}\int_0^{|\lambda_1|^{M(y)}\xi}
			\frac{\omega(s)}{s^{1+\alpha_0}}ds.
\]
		The minimality of $M(y)$ implies $|\lambda_1|^{M(y)}\xi\asymp y^{1/\alpha_0}$, with fixed factors depending on $\xi,|\lambda_1|,|\lambda_2|$. Concavity gives, for every fixed $d>0$,
		\[
		\int_0^{du}\frac{\omega(s)}{s^{1+\alpha_0}}ds
		\geqslant\min\{1,d^{1-\alpha_0}\}
		\int_0^u\frac{\omega(s)}{s^{1+\alpha_0}}ds.
		\]
		Hence there is $C_{32}>0$ with
		\begin{equation}\label{gjgj8}
			\sum_{n=M(y)}^{\infty}\left|\frac{\lambda_2}{\lambda_1}\right|^n\widehat\omega(|\lambda_1|^n\xi)
			\geqslant C_{32}^{-1}B_\omega(y).
		\end{equation}
		Combining~\eqref{gjgj6} and~\eqref{gjgj8} yields
		\begin{equation}\label{gjgj9}
			B_\omega(y)\leqslant\frac{C_{32}}{\varepsilon}
			|\partial_2\Psi_1(\xi,0)-Q_\xi(y)|
			+C_{32}C_{30}y.
		\end{equation}
		For any normalized $C_{1,\eta}$ linearization,~\eqref{eq:TH3_Psi_Phi} gives a constant $C_{33}>0$ such that
		\begin{equation}\label{gjgj10}
			|\partial_2\Psi_1(\xi,0)-Q_\xi(y)|\leqslant C_{33}\eta(y).
		\end{equation}
		Since a positive concave modulus satisfies $y=\mathcal O(\eta(y))$,~\eqref{gjgj9} and \eqref{gjgj10} prove the assertion for all signs of the eigenvalues.
	\end{proof}
	
	If every mapping in the prescribed class admits a $C_{1,\eta}$ linearization, Lemma~\ref{Lemma26} gives $\omega_{\rm A}=\mathcal O(\eta)$ when $I_{\alpha_1}(\omega)<+\infty$. When $I_{\alpha_1}(\omega)=+\infty$, applying Lemmas~\ref{Lemma26} and~\ref{Lemma27}, together with~\eqref{eq:anis-head-necessary}, to $F_R,F_H,F_B$ gives
	\[
	\omega_{\rm A}(t)=B_\omega(t)+H_\omega(t)+R_\omega(t)=\mathcal O(\eta(t)).
	\]
	Since $F_R,F_H,F_B$ have the prescribed linear part, these bounds prove uniform sharpness. Together with the upper estimate, this proves Item~\ref{TH3-1}.
	
	It remains to prove Item~\ref{TH3-2}. Let $|\lambda_1|=|\lambda_2|=\lambda\in(0,1)$. Concavity gives
	\begin{equation}\label{zygj2-2}
		\omega_{\rm B}(t)=\int_0^t\frac{\omega(s)}s ds
		\geqslant\int_0^t\frac{\omega(t)}t ds=\omega(t).
	\end{equation}
	Since $|\Lambda^{\pm n}|=\lambda^{\pm n}$, the Dini condition suffices for the estimates~\eqref{eq:anis-exact-rates} and~\eqref{eq:anis-product-difference}. For some $C_{34}\geqslant1$, they give
	\begin{equation}\label{eq:isotropic-product-upper}
		|D\Psi(x)-D\Psi(y)|\leqslant C_{34}
		\sum_{n=0}^{\infty}\omega(\lambda^n|x-y|)
		\leqslant C_{34}^{2}\omega_{\rm B}(|x-y|).
	\end{equation}
	Thus $\Psi=\lim_{n \to \infty}\Lambda^{-n}F^n$ and its inverse belong to $C_{1,\omega_{\rm B}}$. Optimality follows from the next lemma.
	
	\begin{lemma}\label{Lemma28}
		For every prescribed isotropic diagonal linear part $\Lambda$, there exists a $C_{1,\omega}$ contraction for which every normalized $C_{1,\eta}$ linearization satisfies $\omega_{\rm B}=\mathcal O(\eta)$.
	\end{lemma}
	\begin{proof}
		Fix $\varepsilon>0$ and set
		\[
		f_0(s)=s+\varepsilon\int_0^s\omega(|u|)du,\quad
		F(x,y):=\begin{pmatrix}\lambda_1f_0(x)\\\lambda_2y\end{pmatrix}.
		\]
		This mapping is $C_{1,\omega}$, has derivative $\Lambda$ at $O$, and is a local contraction and diffeomorphism after restriction. For the orbit $(x_n,0)=F^n(t,0)$, $t>0$, one has $|x_n|\geqslant\lambda^nt$. The canonical conjugacy from~\eqref{eq:isotropic-product-upper} therefore satisfies
		\begin{equation}\label{zygj2-3}
			\partial_1\Psi_1(t,0)
			=\prod_{n=0}^{\infty}(1+\varepsilon\omega(|x_n|))
			\geqslant1+\varepsilon\sum_{n=0}^{\infty}\omega(\lambda^nt).
		\end{equation}
		The product converges by the Dini condition. Geometric comparison gives, for some $C_{35}>0$,
		\begin{equation}\label{zygj2-4}
			\omega_{\rm B}(t)\leqslant C_{35}
			\sum_{n=0}^{\infty}\omega(\lambda^nt).
		\end{equation}
		Every normalized $C^1$ linearization coincides with $\Psi$: from $\Phi(z)-z=o(|z|)$ and $|F^n(x)|\leqslant C_{36}\lambda^n|x|$ for a constant $C_{36}\geqslant1$,
		\[
		\Phi(x)-\Lambda^{-n}F^n(x)
		=\Lambda^{-n}(\Phi(F^n(x))-F^n(x))\to 0.
		\]
		If $\Phi\in C_{1,\eta}$,~\eqref{zygj2-3} and \eqref{zygj2-4} now imply, for some $C_{37}>0$,
		\[
\omega_{\rm B}(t)\leqslant\frac{C_{35}}{\varepsilon}
			|\partial_1\Phi_1(t,0)-1|\leqslant C_{37}\eta(t).
\]
		This proves the lemma and completes the proof of Theorem~\ref{TH3}.
	\end{proof}

	\subsection{Proof of Theorem~\ref{thm:focus}}\label{PTD}
	
	\begin{proof}
		In this subsection, $|\cdot|$ denotes the Euclidean norm and the induced operator norm. In view of Case~\ref{case-P-focus},
		\begin{equation}\label{eq:focus-powers}
			|\Lambda^n|=\lambda^n,
			\quad
			|\Lambda^{-n}|=\lambda^{-n},
			\quad
			n\in\mathbb{N}.
		\end{equation}
		
		We first prove Item~\ref{THfocus-1}, where the Dini condition~\eqref{dinicd} holds (i.e., $ 	\int_0^1 \frac{\omega(t)}{t} dt <+ \infty $). Write $F(x)=\Lambda x+f(x)$, so that $f(O)=O$ and $Df(O)=O$. It follows from the definition of $C_{1,\omega}$ that, for some constant ${C_{38}}>0$,
		\begin{equation}\label{eq:focus-flat-estimates}
			|Df(x)|\leqslant {C_{38}}\omega(|x|),
			\quad
			|f(x)|=\left|\int_0^1 Df(tx)x dt\right|\leqslant {C_{38}}|x|\omega(|x|).
		\end{equation}
		Fix $q\in(\lambda,1)$, and choose $0<r\leqslant1$ so small that the closed Euclidean ball $U=\{x\in\mathbb{R}^2:x_1^2+x_2^2\leqslant r^2\}$ is contained in the domain of $F$ and $\lambda+{C_{38}}\omega(r)\leqslant q$. By~\eqref{eq:focus-flat-estimates}, $F(U)\subset U$ and
		\begin{equation}\label{eq:focus-orbits}
			|F(x)|\leqslant q|x|,
			\quad
			|F^n(x)|\leqslant q^n|x|\leqslant q^nr,
			\quad
			\forall x\in U.
		\end{equation}
		
		By the monotonicity of $\omega$ and the Dini condition~\eqref{dinicd},
		\begin{equation}\label{eq:focus-series}
			\sum_{j=0}^{\infty}\omega(q^jr)\leqslant\omega(r)+\frac{1}{|\log q|}\sum_{j=1}^{\infty}\int_{q^jr}^{q^{j-1}r}\frac{\omega(t)}{t}dt=\omega(r)+\frac{1}{|\log q|}\int_0^r\frac{\omega(t)}{t}dt<+\infty.
		\end{equation}
		The chain rule together with~\eqref{eq:focus-flat-estimates} and~\eqref{eq:focus-orbits} therefore yields, for all $n\in\mathbb{N}$,
		\begin{align}
			|DF^n(x)|&\leqslant\prod_{j=0}^{n-1}\bigl(\lambda+|Df(F^j(x))|\bigr)\leqslant\lambda^n\prod_{j=0}^{n-1}\left(1+\frac{{C_{38}}}{\lambda}\omega(q^jr)\right)\notag \\
			\label{eq:focus-derivatives}	&\leqslant\lambda^n\exp\left(\frac{{C_{38}}}{\lambda}\sum_{j=0}^{\infty}\omega(q^jr)\right)=:{C_{39}}\lambda^n,
			\quad
			\forall x\in U,
		\end{align}
		where $0<{C_{39}}<+\infty$ by~\eqref{eq:focus-series}.
		
		Set $\Psi_n=\Lambda^{-n}F^n$. Direct computation gives
		\[
		D\Psi_{n+1}(x)-D\Psi_n(x)=\Lambda^{-(n+1)}Df(F^n(x))DF^n(x).
		\]
		It follows from~\eqref{eq:focus-powers}, \eqref{eq:focus-orbits}, \eqref{eq:focus-flat-estimates}, and~\eqref{eq:focus-derivatives} that
		\[
		|D\Psi_{n+1}(x)-D\Psi_n(x)|\leqslant\frac{{C_{38}}{C_{39}}}{\lambda}\omega(q^nr),
		\quad
		\forall x\in U.
		\]
		Since the bound is independent of $x$ and summable by~\eqref{eq:focus-series}, $\{D\Psi_n\}$ converges uniformly on $U$ to a continuous matrix-valued function $H$. Moreover, since $\Psi_n(O)=O$ and $U$ is convex, the segment integral yields
		\[
		|\Psi_{n+1}(x)-\Psi_n(x)|=\left|\int_0^1\bigl(D\Psi_{n+1}(tx)-D\Psi_n(tx)\bigr)x dt\right|\leqslant\frac{{C_{38}}{C_{39}}}{\lambda}r\omega(q^nr),
		\quad
		\forall x\in U.
		\]
		Thus $\{\Psi_n\}$ converges uniformly on $U$ to a continuous mapping $\Psi$. For any $x$ in the interior of $U$ and all sufficiently small $h$, passing to the limit in the segment integral gives
		\[
		\Psi(x+h)-\Psi(x)=\int_0^1 H(x+th)h dt.
		\]
		Consequently, $\Psi$ is of class $C^1$ and $D\Psi=H$. Moreover,
		\[
		\Psi(O)=O,
		\quad
		D\Psi(O)=\lim_{n\to\infty}\Lambda^{-n}DF^n(O)=\mathrm{Id}.
		\]
		The inverse function theorem shows that $\Psi$ is a local $C^1$ diffeomorphism. Finally,
		\[
		\Psi_n(F(x))=\Lambda\Psi_{n+1}(x),
		\]
		and passage to the limit gives $\Psi( F(x))=\Lambda\Psi(x)$. This proves the $C^1$ assertion in Item~\ref{THfocus-1}.

		We now estimate the modulus of $D\Psi$. In \eqref{eq:focus-flat-estimates}, choose ${C_{38}}$ at the outset so that it also bounds the full $\omega$-modulus of $Df$. By convexity of $U$ and \eqref{eq:focus-derivatives},
		\[
		|F^n(x)-F^n(y)|\leqslant C_{39}\lambda^n|x-y|.
		\]
		Set $A_n(x)=D\Psi_n(x)$ and
		\[
		B_n(x)=\Lambda^{-(n+1)}Df(F^n(x))\Lambda^n.
		\]
		Then $A_0=\operatorname{Id}$, $A_{n+1}=(\operatorname{Id}+B_n)A_n$, and
\[
\sup_{x\in U}\vert{}B_n(x)\vert{}\leqslant\frac{C_{38}}\lambda\omega(q^nr), \quad \vert{}B_n(x)-B_n(y)\vert{}\leqslant\frac{C_{38}}\lambda \omega(C_{39}\lambda^n\vert{}x-y\vert{}).
\]
		The first bound is summable, so every finite consecutive product of the factors $\operatorname{Id}+B_j(x)$ has norm at most
		\[
		C_{40}:=\exp\left(\frac{C_{38}}\lambda\sum_{j=0}^{\infty}\omega(q^jr)\right)<+\infty.
		\]
		Telescoping the ordered products gives
		\[
		A_n(x)-A_n(y)
		=\sum_{j=0}^{n-1}
		(\operatorname{Id}+B_{n-1}(x))\cdots(\operatorname{Id}+B_{j+1}(x))
		(B_j(x)-B_j(y))A_j(y),
		\]
		where an empty product is understood as $\operatorname{Id}$. Since $|A_j(y)|\leqslant C_{40}$, concavity yields
\[
\vert{}A_n(x)-A_n(y)\vert{}\leqslant C_{40}^2\sum_{j=0}^{n-1}\vert{}B_j(x)-B_j(y)\vert{}\leqslant C_{41}\sum_{j=0}^{n-1}\omega(\lambda^j\vert{}x-y\vert{}),
\]
		where $C_{41}:=C_{40}^2C_{38}\max\{1,C_{39}\}/\lambda$ is independent of $n,x,y$.
		For sufficiently small $t>0$, monotonicity gives
		\[
		\sum_{j=0}^{\infty}\omega(\lambda^jt)
		\leqslant\omega(t)+\frac1{|\log\lambda|}\int_0^t\frac{\omega(s)}s ds
		\leqslant\left(1+\frac1{|\log\lambda|}\right)\omega_{\rm B}(t),
		\]
		where the last step uses \eqref{zygj2-2}. 
		Passing to the uniform limit gives
\[
\vert{}D\Psi(x)-D\Psi(y)\vert{}\leqslant C_{42}\omega_{\rm B}(\vert{}x-y\vert{}), \quad C_{42}:=C_{41}\left(1+\frac1{\vert{}\log\lambda\vert{}}\right).
\]
		Thus $\Psi\in C_{1,\omega_{\rm B}}$. We also record the inverse argument for later use: if a local $C^1$ diffeomorphism $H$ has derivative modulus $\nu$, with $\nu$ concave, then
		\[
		DH^{-1}=(DH\circ H^{-1})^{-1}
		\]
		and $P^{-1}-Q^{-1}=P^{-1}(Q-P)Q^{-1}$ give the same modulus for $H^{-1}$ on a smaller neighborhood, by boundedness of the inverse derivatives and local Lipschitz continuity.

		We next prove Item~\ref{THfocus-2}, provided the Dini condition~\eqref{dinicd} fails (i.e., $ 	\int_0^1 \frac{\omega(t)}{t} dt =+ \infty $). Suppose that the Dini condition~\eqref{dinicd} fails. Recall the normalized averaged modulus $\widehat{\omega}$ introduced in Lemma~\ref{prop21}. Define
\begin{equation}\label{eq:focus-counterexample}
			F(x):=\bigl(1+\widehat{\omega}(|x|)\bigr)\Lambda\begin{pmatrix}x_1\\x_2\end{pmatrix}.
		\end{equation}
		The same argument as in the proof of Lemma~\ref{pro:regularity} shows that $F\in C_{1,\omega}$, $F(O)=O$, and $DF(O)=\Lambda$. After restricting its domain if necessary, the inverse function theorem shows that $F$ is a local $C^1$ diffeomorphism.
		
		Choose $q\in(\lambda,1)$ and $0<r\leqslant1$ sufficiently small that
		\[
		\lambda\bigl(1+\omega(r)\bigr)\leqslant q,
		\quad
		\omega(r)\leqslant1.
		\]
		For $0<|x|<r$, set $r_n=|F^n(x)|$. Since the rotation matrix in $\Lambda$ is an isometry,~\eqref{eq:focus-counterexample} gives
		\begin{equation}\label{eq:radial-orbits}
			r_{n+1}=\lambda r_n\bigl(1+\widehat{\omega}(r_n)\bigr),
			\quad
			\lambda^nr_0\leqslant r_n\leqslant q^nr_0.
		\end{equation}
		
		Since the Dini condition~\eqref{dinicd} fails, we have $\int_0^{r_0}\frac{\omega(t)}{t}dt=+\infty$. By the monotonicity of $\omega$,
		\[
		\sum_{j=0}^{\infty}\omega(\lambda^jr_0)\geqslant\frac{1}{|\log\lambda|}\sum_{j=0}^{\infty}\int_{\lambda^{j+1}r_0}^{\lambda^jr_0}\frac{\omega(t)}{t}dt=\frac{1}{|\log\lambda|}\int_0^{r_0}\frac{\omega(t)}{t}dt=+\infty.
		\]
		By~\eqref{eq:radial-orbits} and Lemma~\ref{prop21}, $0\leqslant\widehat{\omega}(r_j)\leqslant\omega(r_j)\leqslant\omega(r)\leqslant1$. Since $0\leqslant u\leqslant1$ implies $\log(1+u)\geqslant u/2$, we obtain
		\begin{equation}\label{eq:radial-divergence}
			\log\frac{r_n}{\lambda^nr_0}=\sum_{j=0}^{n-1}\log\bigl(1+\widehat{\omega}(r_j)\bigr)\geqslant\frac{1}{2}\sum_{j=0}^{n-1}\widehat{\omega}(r_j)\geqslant\frac{1}{4}\sum_{j=0}^{n-1}\omega(\lambda^jr_0)\to+\infty\quad\text{as }n\to\infty.
		\end{equation}
		
		Suppose, to the contrary, that $F$ admits a local $C^1$ linearizing diffeomorphism $\Phi$ with $\Phi(O)=O$ and $D\Phi(O)=\mathrm{Id}$, such that $\Phi(F(x))=\Lambda\Phi(x)$. Then we have
		\[
		|\Phi(x)|\geqslant\frac{1}{2}|x|
		\]
		for all sufficiently small $x$. By~\eqref{eq:radial-orbits}, we may choose $x\neq O$ so small that its entire forward orbit is contained in the domain of the conjugacy and in the neighborhood on which the preceding estimate holds. Then
		\[
		\frac{1}{2}r_n\leqslant|\Phi(F^n(x))|=|\Lambda^n\Phi(x)|=\lambda^n|\Phi(x)|.
		\]
		Consequently,
		\[
		\frac{r_n}{\lambda^nr_0}\leqslant\frac{2|\Phi(x)|}{|x|}<+\infty,
		\]
		contradicting~\eqref{eq:radial-divergence}. This proves Item~\ref{THfocus-2}.

		It remains to prove the optimality of the $\omega_{\rm B}$-estimate in Item~\ref{THfocus-1}. We consider the radial mapping defined in \eqref{eq:focus-counterexample}, now assuming that $\omega$ satisfies the Dini condition \eqref{dinicd}. Its normalized iterates and their limit satisfy
		\[
		\Lambda^{-n}F^n(x)=\left(\prod_{j=0}^{n-1}(1+\widehat\omega(r_j))\right)x,
		\quad
		\Psi(x)=p(|x|)x,
		\quad
		p(r)=\prod_{j=0}^{\infty}(1+\widehat\omega(r_j)),
		\]
		where $r_0=r$ and the recurrence is \eqref{eq:radial-orbits}. The product is finite by $r_j\leqslant q^jr$ and the Dini condition \eqref{dinicd}. Since $r_j\geqslant\lambda^jr$, nonnegativity of its factors and Lemma~\ref{prop21} give
		\[
		p(r)-1\geqslant\sum_{j=0}^{\infty}\widehat\omega(r_j)
		\geqslant\frac12\sum_{j=0}^{\infty}\omega(\lambda^jr)
		\geqslant\frac{\omega_{\rm B}(r)}{2|\log\lambda|}.
		\]
		For a unit vector $v$ perpendicular to $x\neq O$, the radial formula gives $D\Psi(x)v=p(|x|)v$. Hence
		\[
		|D\Psi(x)-\operatorname{Id}|\geqslant p(|x|)-1
		\geqslant\frac{\omega_{\rm B}(|x|)}{2|\log\lambda|}.
		\]
		For any normalized $C^1$ linearization $\Phi$, the conjugacy equation gives
		\[
		\Phi(x)-\Psi_n(x)
		=\Lambda^{-n}\bigl(\Phi(F^n(x))-F^n(x)\bigr)\to 0 \quad \text{as } n \to \infty,
		\]
		since $\Phi(z)-z=o(|z|)$ and~\eqref{eq:focus-derivatives} implies $|F^n(x)|\leqslant C_{39}\lambda^n|x|$. Hence $\Phi=\Psi$, and the preceding lower bound gives $\omega_{\rm B}(t)=\mathcal O(\eta(t))$ whenever $\Phi\in C_{1,\eta}$.

		The proof of Theorem~\ref{thm:focus} is now complete.
		
	\end{proof}

	\subsection{Proof of Theorem~\ref{thm:jordan}}\label{PTJORDAN}
	
	\begin{proof}

		In Case~\ref{case-P-degenerate-node}, the linear part is represented by
		\begin{equation}\label{eq:jordan-form}
			\Lambda=\begin{pmatrix}\lambda&1\\0&\lambda\end{pmatrix},
			\quad 0<|\lambda|<1.
		\end{equation}
		We first prove Item~\ref{THjordan-1} under the weighted Dini condition~\eqref{eq:log2dini} (i.e., $\int_0^1\frac{\omega(t)}t\left(\log\frac et\right)^2dt<+\infty$). Fix $q\in(|\lambda|,1)$ and $0<\kappa<q-|\lambda|$, and equip $\mathbb{R}^2$ with the equivalent norm
		\[
		|x|_*=\max\{|x_1|,\kappa^{-1}|x_2|\},
		\quad x=(x_1,x_2)^\top.
		\]
		We denote by $|\cdot|_*$ the corresponding operator norm on $\mathbb{R}^{2\times2}$. Since
		\[
		\Lambda^{\pm n}=\lambda^{\pm n}\begin{pmatrix}1&\pm n/\lambda\\0&1\end{pmatrix},
		\quad n\in\mathbb{N},
		\]
		we have $|\Lambda|_*=|\lambda|+\kappa<q$, as well as
		\begin{equation}\label{eq:jordan-powers}
			|\Lambda^{\pm n}|_*=|\lambda|^{\pm n}\left(1+\frac{n\kappa}{|\lambda|}\right)
			\leqslant {C_{43}}(n+1)|\lambda|^{\pm n},
			\quad n\in\mathbb{N},
		\end{equation}
		where the constant ${C_{43}}:=\max\{1,\kappa/|\lambda|\}>0$.
		
		Write $F(x)=\Lambda x+f(x)$ so that $f(O)=O$ and $Df(O)=O$. From the norm comparison $|x|\leqslant|x|_*\leqslant\kappa^{-1}|x|$ and $|A|_*\leqslant\kappa^{-1}|A|$ on $\mathbb{R}^2$ and $\mathbb{R}^{2\times2}$, the $C_{1,\omega}$ regularity gives a constant ${C_{44}}>0$ satisfying
		\begin{equation}\label{eq:jordan-flat-estimates}
			|Df(x)|_*\leqslant {C_{44}}\omega(|x|_*),
			\quad |f(x)|_*=\left|\int_0^1Df(tx)x dt\right|_*\leqslant {C_{44}}|x|_*\omega(|x|_*).
		\end{equation}
		Choose $0<r\leqslant1$ small enough so that the closed ball $U=\{x\in\mathbb{R}^2:|x|_*\leqslant r\}$ is contained in a neighborhood on which~\eqref{eq:jordan-flat-estimates} holds and satisfies $|\lambda|+\kappa+{C_{44}}\omega(r)\leqslant q$. Then $F(U)\subset U$, and
		\begin{equation}\label{eq:jordan-orbits}
			|F(x)|_*\leqslant q|x|_*,
			\quad |F^n(x)|_*\leqslant q^n|x|_*\leqslant q^nr,
			\quad \forall x\in U,\ n\in\mathbb{N}.
		\end{equation}
		
		For $j\in\mathbb{N}^+$ and $t\in[q^jr,q^{j-1}r]$,
		\[
		\log\frac et\geqslant1+(j-1)|\log q|
		\geqslant j\min\{1,|\log q|\}
		\geqslant\frac{j+1}2\min\{1,|\log q|\}.
		\]
		By monotonicity of $\omega$,
		\[
		(j+1)^2\omega(q^jr)\leqslant
		\frac4{|\log q|\min\{1,|\log q|\}^2}
		\int_{q^jr}^{q^{j-1}r}\frac{\omega(t)}t\left(\log\frac et\right)^2dt.
		\]
		Summing over $j\in\mathbb{N}^+$ and using~\eqref{eq:log2dini}, we obtain
		\begin{align}
			\sum_{j=0}^{\infty}(j+1)\omega(q^jr)
			&\leqslant\sum_{j=0}^{\infty}(j+1)^2\omega(q^jr) \notag \\
			&\leqslant\omega(r)+\frac{4}{|\log q|\min\{1,|\log q|\}^2}
			\int_0^r\frac{\omega(t)}{t}\left(\log\frac{e}{t}\right)^2dt<+\infty. \label{eq:jordan-series}
		\end{align}
		
		Observe that
		\[
		DF^{j+1}(x)-\Lambda DF^j(x)=Df(F^j(x))DF^j(x).
		\]
		Multiplying on the left by $\Lambda^{n-1-j}$ and summing over $0\leqslant j\leqslant n-1$ yields
		\begin{equation}\label{eq:variation}
			DF^n(x)=\Lambda^n+\sum_{j=0}^{n-1}\Lambda^{n-1-j}Df(F^j(x))DF^j(x),
			\quad n\in\mathbb{N}^+.
		\end{equation}
		Set $a_n(x)=|DF^n(x)|_*/((n+1)|\lambda|^n)$ and $C_{45} =C_{43}C_{44}/\vert{}\lambda\vert{}>0$. Estimates~\eqref{eq:jordan-powers}, \eqref{eq:jordan-flat-estimates}, \eqref{eq:jordan-orbits}, and~\eqref{eq:variation} lead to
		\[
		\begin{aligned}
			a_n(x)&\leqslant C_{43}+C_{45}\sum_{j=0}^{n-1}\frac{n-j}{n+1}(j+1)\omega(q^jr)a_j(x)\\
			&\leqslant C_{43}+C_{45}\sum_{j=0}^{n-1}(j+1)\omega(q^jr)a_j(x),
			\quad \forall x\in U.
		\end{aligned}
		\]
		The discrete Gronwall inequality therefore gives
		\[
		a_n(x)\leqslant C_{43}\prod_{j=0}^{n-1}
		\bigl(1+C_{45}(j+1)\omega(q^jr)\bigr)
		\leqslant C_{43}\exp\left(C_{45}\sum_{j=0}^{\infty}(j+1)\omega(q^jr)\right)
		=:C_{46}<+\infty,
		\]
		where $C_{46}>0$ is independent of $x$ and $n$, and finite by~\eqref{eq:jordan-series}. Consequently,
		\begin{equation}\label{eq:jordan-derivatives}
			|DF^n(x)|_*\leqslant {C_{46}}(n+1)|\lambda|^n,
			\quad n\in\mathbb{N},
			\quad \forall x\in U.
		\end{equation}
		
		Now consider $\Psi_n=\Lambda^{-n}F^n$. We have
		\[
		D\Psi_{n+1}(x)-D\Psi_n(x)=\Lambda^{-(n+1)}Df(F^n(x))DF^n(x).
		\]
		Combining~\eqref{eq:jordan-powers}, \eqref{eq:jordan-flat-estimates}, \eqref{eq:jordan-orbits}, and~\eqref{eq:jordan-derivatives}, we obtain
		\[
		\begin{aligned}
			|D\Psi_{n+1}(x)-D\Psi_n(x)|_*
			&\leqslant {C_{45}}{C_{46}}(n+1)(n+2)\omega(q^nr)\\
			&\leqslant2{C_{45}}{C_{46}}(n+1)^2\omega(q^nr),
			\quad \forall x\in U.
		\end{aligned}
		\]
		Since this bound is independent of $x$ and summable by~\eqref{eq:jordan-series}, the sequence $\{D\Psi_n\}$ converges uniformly on $U$ to a continuous matrix-valued mapping $H$. Since $\Psi_n(O)=O$ and $U$ is convex, the mean value estimate along segments yields
		\[
		\begin{aligned}
			|\Psi_{n+1}(x)-\Psi_n(x)|_*
			&=\left|\int_0^1\bigl(D\Psi_{n+1}(tx)-D\Psi_n(tx)\bigr)x dt\right|_*\\
			&\leqslant2{C_{45}}{C_{46}}r(n+1)^2\omega(q^nr),
			\quad \forall x\in U.
		\end{aligned}
		\]
		Consequently, $\{\Psi_n\}$ also converges uniformly on $U$ to a continuous mapping $\Psi$. 
		The remainder of the $C^1$ argument proceeds in complete analogy with that of Section~\ref{PTD}, proving the $C^1$ assertion in Item~\ref{THjordan-1}.

			The regularity argument follows Section~\ref{PTD}. Under~\eqref{eq:log2dini},
			\[
			\omega_{\mathrm C}(t)=\int_0^1\frac{\omega(tu)}u
			\left(\log\frac eu\right)^2du
			\]
			defines a concave modulus of continuity. Comparison on geometric intervals and concavity give, for each fixed $\sigma\in(0,1)$,
			\begin{equation}\label{eq:jordan-modulus-series}
				\omega_{\mathrm C}(t)\asymp
				\sum_{j=0}^{\infty}(j+1)^2\omega(\sigma^jt).
			\end{equation}
			Define $A_n,B_n$ as in Section~\ref{PTD}. The $q$-Lipschitz property of $F$ on $U$, the modulus of $Df$, and~\eqref{eq:jordan-powers} give
\[
\sup_U|B_n|_*\leqslant{C_{47}}(n+1)^2\omega(q^nr),\quad |B_n(x)-B_n(y)|_*\leqslant{C_{47}}(n+1)^2\omega(q^n|x-y|_*).
\]
			The first bound is summable by~\eqref{eq:jordan-series}. The same ordered-product argument, passage to the uniform limit, and~\eqref{eq:jordan-modulus-series} with $\sigma=q$ yield
			\[
			|D\Psi(x)-D\Psi(y)|_*
			\leqslant{C_{48}}\sum_{j=0}^{\infty}(j+1)^2\omega(q^j|x-y|_*)
			\leqslant{C_{49}}\omega_{\mathrm C}(|x-y|_*),
			\]
			with constants independent of $n,x,y$. Equivalence of norms and the inverse argument in Section~\ref{PTD} prove that $\Psi,\Psi^{-1}\in C_{1,\omega_{\mathrm C}}$.

			We next prove Item~\ref{THjordan-2}, assuming that~\eqref{eq:log2dini} fails (i.e., $\int_0^1\frac{\omega(t)}t\left(\log\frac et\right)^2dt=+\infty$).
		  Fix $0<\xi\leqslant1$, and set
		\begin{equation}\label{eq:centers}
			z_n=(z_{n,1}, z_{n,2})^\top=\Lambda^n(0,\xi)^\top=(n\lambda^{n-1}\xi,\lambda^n\xi)^\top,
			\quad r_n=\frac14|\lambda|^n\xi,
			\quad n\geqslant1.
		\end{equation}
		Since
		\[
		|z_n|=n|\lambda|^{n-1}\xi,
		\quad\frac{|z_{n+1}|}{|z_n|}=|\lambda|\left(1+\frac1n\right),
		\]
		we may choose $q\in(|\lambda|,1)$ and $N\geqslant1$ such that
		\[
		|\lambda|\left(1+\frac1n\right)\leqslant q,
		\quad\frac{(1-q)n}{|\lambda|}>\frac12,
		\quad n\geqslant N.
		\]
		For $m>n\geqslant N$, iteration yields $|z_m|\leqslant q^{m-n}|z_n|\leqslant q|z_n|$, and hence
		\begin{equation}\label{eq:disjoint}
			|z_n-z_m|\geqslant|z_n|-|z_m|\geqslant(1-q)n|\lambda|^{n-1}\xi
			>\frac12|\lambda|^n\xi>r_n+r_m.
		\end{equation}
		Thus, the closed balls $\overline{B}(z_n,r_n)$, $n\geqslant N$, are pairwise disjoint. Moreover, $r_n/|z_n|=|\lambda|/(4n)<1$, so none of these balls contains the origin $O$.
		
		Let $\chi\in C_c^\infty(\mathbb{R})$ be the symmetric cutoff function used in Section~\ref{SEC3}, satisfying $0\leqslant\chi\leqslant1$, $\chi(t)=1$ for $|t|\leqslant1$, and $\chi(t)=0$ for $|t|\geqslant2$. Recall the normalized averaged modulus $\widehat{\omega}$ introduced in Lemma~\ref{prop21}. Define
		\begin{equation}\label{eq:jordan-counterexample}
			F(x_1,x_2)=\begin{pmatrix}\lambda x_1+x_2\\\lambda x_2+R(x_1,x_2)\end{pmatrix},
		\end{equation}
		where
		\begin{equation}\label{eq:perturbation}
			R(x_1,x_2)=\lambda^2\sum_{n=N}^{\infty}\widehat{\omega}(r_n)(x_1-z_{n,1})
			\chi\left(\frac{2(x_1-z_{n,1})}{r_n}\right)
			\chi\left(\frac{2(x_2-z_{n,2})}{r_n}\right).
		\end{equation}
		The $n$-th summand is supported in $\overline{B}(z_n,r_n)$. In view of~\eqref{eq:disjoint}, at most one summand is nonzero at any given point, and the series is locally finite on $\mathbb{R}^2\setminus\{O\}$ because $|z_n|+r_n\to0$. An argument analogous to the proof of Lemma~\ref{pro:regularity} shows that $F\in C_{1,\omega}$, $F(O)=O$, and $DF(O)=\Lambda$. After restricting its domain to a suitable neighborhood of $O$, the inverse function theorem ensures that $F$ is a local $C^1$ diffeomorphism.
		
		At $z_n$, the factor $x_1-z_{n,1}$ vanishes and both cutoffs are identically one in a neighborhood of $z_n$. In conjunction with~\eqref{eq:centers}, \eqref{eq:jordan-counterexample}, and~\eqref{eq:perturbation}, this yields
		\begin{equation}\label{eq:orbit-data}
			F(z_n)=z_{n+1},
			\quad DF(z_n)=\begin{pmatrix}\lambda&1\\\lambda^2\widehat{\omega}(r_n)&\lambda\end{pmatrix},
			\quad n\geqslant N.
		\end{equation}
		The core mechanism underlying this counterexample is to exploit the logarithmic factor naturally generated by the Jordan block in the linear regime and inject the perturbation into its unoccupied entries, thereby establishing an effective coupling that ultimately gives rise to a squared-logarithmic factor.
		
		Suppose, to the contrary, that $F$ admits a local $C^1$ linearizing diffeomorphism $\Phi$ satisfying $\Phi(O)=O$, $D\Phi(O)=\mathrm{Id}$, and $\Phi(F(x))=\Lambda\Phi(x)$. Set
		\[
		P=\operatorname{diag}(1,\lambda),
		\quad e_n=\widehat{\omega}(r_n),
		\quad n\geqslant N,
		\]
		and, for all sufficiently large $n$, let $H_n=P^{-1}D\Phi(z_n)P$. Then $H_n\to\mathrm{Id}$ and $e_n\geqslant0$. By~\eqref{eq:jordan-form} and~\eqref{eq:orbit-data},
		\[
		P^{-1}\Lambda P=\lambda\begin{pmatrix}1&1\\0&1\end{pmatrix},
		\quad P^{-1}DF(z_n)P=\lambda\begin{pmatrix}1&1\\e_n&1\end{pmatrix}.
		\]
		Differentiating the conjugacy equation at $z_n$ gives $D\Phi(z_{n+1})DF(z_n)=\Lambda D\Phi(z_n)$ for all sufficiently large $n$. Conjugating by $P$ and cancelling $\lambda\neq0$, we obtain
		\begin{equation}\label{eq:matrix}
			H_{n+1}\begin{pmatrix}1&1\\e_n&1\end{pmatrix}
			=\begin{pmatrix}1&1\\0&1\end{pmatrix}H_n.
		\end{equation}
		Write $H_n=\begin{pmatrix}a_n&b_n\\c_n&d_n\end{pmatrix}$. Comparing entries in~\eqref{eq:matrix} yields
		\begin{equation}\label{eq:entries}
			\begin{aligned}
				c_n-c_{n+1} &= e_nd_{n+1},      &\quad d_n-d_{n+1} &= c_{n+1},\\
				a_n-a_{n+1} &= e_nb_{n+1}-c_n,  &\quad b_n-b_{n+1} &= a_{n+1}-d_n.
			\end{aligned}
		\end{equation}
		Choose $M\geqslant N$ sufficiently large so that these identities hold and, recalling that $H_n\to\mathrm{Id}$,
		\[
		d_{n+1}\geqslant\frac12,
		\quad |b_{n+1}|\leqslant\frac14,
		\quad n\geqslant M.
		\]
		For $m\geqslant n\geqslant M$, iterating the first identity in~\eqref{eq:entries} gives
		\[
		c_n-c_{m+1}=\sum_{k=n}^{m}e_kd_{k+1}.
		\]
		Taking $m\to\infty$ and using $c_{m+1}\to0$, we obtain
		\begin{equation}\label{eq:first-tail}
			c_n=\sum_{k=n}^{\infty}e_kd_{k+1}\geqslant\frac12\sum_{k=n}^{\infty}e_k\geqslant0,
			\quad 0\leqslant e_n\leqslant2c_n,
			\quad n\geqslant M.
		\end{equation}
		Set $\tau_n=d_n-a_n$, so that $\tau_n\to0$. By~\eqref{eq:entries} and~\eqref{eq:first-tail},
		\[
		\tau_n-\tau_{n+1}=c_n+c_{n+1}-e_nb_{n+1}
		\geqslant\frac12c_n+c_{n+1}\geqslant\frac12c_n,
		\quad n\geqslant M.
		\]
		Summing from $n$ to $m$, letting $m\to\infty$, and rearranging the nonnegative sums, we find
		\begin{equation}\label{eq:second-tail}
			\tau_n\geqslant\frac12\sum_{j=n}^{\infty}c_j
			\geqslant\frac14\sum_{j=n}^{\infty}\sum_{k=j}^{\infty}e_k
			=\frac14\sum_{k=n}^{\infty}(k-n+1)e_k\geqslant0,
			\quad n\geqslant M.
		\end{equation}
		The second and fourth identities in~\eqref{eq:entries} also give
		\[
		b_n-b_{n+1}=-\tau_{n+1}-c_{n+1}\leqslant-\tau_{n+1},
		\quad n\geqslant M.
		\]
		Since $b_n\to0$, summation yields $\sum_{n=M}^{\infty}\tau_{n+1}\leqslant-b_M<+\infty$. Together with~\eqref{eq:second-tail} and an interchange of the summation order, this leads to
		\[
		\begin{aligned}
			\frac18\sum_{k=M+1}^{\infty}(k-M)(k-M+1)e_k
			&= \frac14\sum_{k=M+1}^{\infty}e_k\sum_{n=M}^{k-1}(k-n) \\
			&= \frac14\sum_{n=M}^{\infty}\sum_{k=n+1}^{\infty}(k-n)e_k
			\leqslant\sum_{n=M}^{\infty}\tau_{n+1}\leqslant-b_M<+\infty.
		\end{aligned}
		\]
		This implies the following key estimate:
		\begin{equation}\label{eq:jordan-necessary-series}
			\sum_{n=N}^{\infty}(n+1)^2e_n
			<+\infty.
		\end{equation}
		
		On the other hand,~\eqref{eq:centers} gives
		\[
		\log\frac{e}{r_{n+1}}=\log\frac{4e}{\xi}+(n+1)\log\frac1{|\lambda|}
		\leqslant(n+1)\left(\log\frac{4e}{\xi}+\log\frac1{|\lambda|}\right).
		\]
		By the monotonicity of $\omega$,
		\[
		\int_{r_{n+1}}^{r_n}\frac{\omega(t)}t\left(\log\frac et\right)^2dt
		\leqslant\log\frac1{|\lambda|}\left(\log\frac{e}{r_{n+1}}\right)^2\omega(r_n)
		\leqslant {C_{50}}(n+1)^2\omega(r_n),
		\]
		where ${C_{50}}:=\log(1/|\lambda|)\bigl(\log(4e/\xi)+\log(1/|\lambda|)\bigr)^2>0$. Applying Lemma~\ref{prop21} and summing the interval estimates, we obtain
		\[
		\begin{aligned}
			\sum_{n=N}^{\infty}(n+1)^2e_n
			&\geqslant\frac12\sum_{n=N}^{\infty}(n+1)^2\omega(r_n)
			\geqslant\frac1{2{C_{50}}}\sum_{n=N}^{\infty}\int_{r_{n+1}}^{r_n}\frac{\omega(t)}t\left(\log\frac et\right)^2dt\\
			&=\frac1{2{C_{50}}}\int_0^{r_N}\frac{\omega(t)}t\left(\log\frac et\right)^2dt=+\infty,
		\end{aligned}
		\]
		contradicting~\eqref{eq:jordan-necessary-series}. This proves Item~\ref{THjordan-2}.

			For optimality under~\eqref{eq:log2dini}, retain the centers~\eqref{eq:centers} and perturbation~\eqref{eq:perturbation}, but replace the radii by
\[
r_n=c|z_n|,\quad 0<c<\frac{1-q}{1+q}.
\]
		This choice is made to ensure that $r_n$ and $|z_n|$ remain on the same scale, thereby allowing a uniform scale in the subsequent estimates for the modulus of continuity.
			For $m>n\geqslant N$, one has
			\[
			|z_n-z_m|\geqslant(1-q)|z_n|>r_n+r_m.
			\]
			Thus the supports are disjoint and avoid $O$. The derivative of each bump is bounded by ${C_{51}}\omega(r_n)$ and has Lipschitz constant at most ${C_{51}}\omega(r_n)/r_n$, uniformly in $n$. Concavity and the segment argument in Lemma~\ref{pro:regularity} give $F\in C_{1,\omega}$, while~\eqref{eq:orbit-data} remains valid.
			
			For any normalized linearization $\Phi\in C_{1,\eta}$, use the same $e_n,H_n$ and matrix entries as above. Equations~\eqref{eq:entries}, \eqref{eq:first-tail}, and~\eqref{eq:second-tail} apply unchanged, and summing $b_n-b_{n+1}=-\tau_{n+1}-c_{n+1}$ gives
			\begin{equation}\label{eq:jordan-sharp-tail}
				-b_n\geqslant\sum_{i=n}^{\infty}\tau_{i+1}
				\geqslant\frac18\sum_{j=1}^{\infty}j(j+1)e_{n+j}.
			\end{equation}
			Since $|z_{n+j}|\geqslant|\lambda|^j|z_n|$, Lemma~\ref{prop21} and concavity imply $e_{n+j}\geqslant(c/2)\omega(|\lambda|^j|z_n|)$. Hence~\eqref{eq:jordan-modulus-series} and~\eqref{eq:jordan-sharp-tail} yield
			\[
			\omega_{\mathrm C}(|z_n|)\leqslant{C_{52}}|b_n|
			\leqslant{C_{53}}\eta(|z_n|).
			\]
			The constants are independent of $n$. As $|\lambda|\leqslant|z_{n+1}|/|z_n|\leqslant q$, concavity of $\omega_{\mathrm C}$ and monotonicity of $\eta$ extend this bound to every sufficiently small $t>0$. Thus $\omega_{\mathrm C}(t)=\mathcal O(\eta(t))$, for either sign of $\lambda$, completing the proof of Theorem~\ref{thm:jordan}.
	\end{proof}

\subsection{Proof of Theorem~\ref{TH4}}

\begin{proof}
	We first prove Item~\ref{TH4-I} under the Dini condition~\eqref{dinicd}. We straighten the invariant foliations and then apply the following one-dimensional lemma on the invariant axes.
	
	\begin{lemma}\label{DINILEMMA}
		Let $g$ be a local $C^1$ diffeomorphism near $0$ satisfying $g(0)=0$ and $g'(0)=\lambda$ with $0<|\lambda|<1$. If $g'$ admits a Dini modulus of continuity $\nu$, then there exists a local $C^1$ diffeomorphism $h$ such that $h(0)=0$, $h'(0)=1$, and $h(g(x))=\lambda h(x)$. Moreover, $h,h^{-1}\in C_{1,\nu_1}$, where $\nu_1(t):=\int_0^t\frac{\nu(s)}{s} ds$.
	\end{lemma}
	\begin{proof}
		Choose $q\in(|\lambda|,1)$ and $I=[-r,r]$ so that $g(I)\subset I$, $\sup_{x\in I}|g'(x)|\leqslant q$, and $1/2\leqslant g'/\lambda\leqslant3/2$. Set $h_n=\lambda^{-n}g^n$. Then
		\[
		h_n'(x)
		=\prod_{k=0}^{n-1}\frac{g'(g^k(x))}{\lambda},
		\quad
		\sup_{x \in I}\left|\frac{g'(g^k(x))}{\lambda}-1\right|
		\leqslant{C_{54}}\nu(q^kr).
		\]
		The Dini condition \eqref{dinicd} on $ \nu $ makes the latter bounds summable. Thus the logarithms of these positive products converge uniformly, and $h_n(0)=0$ gives $h_n\to h$ in $C^1(I)$ with $h'>0$ and $h'(0)=1$. Passing to the limit in $h_n (g(x))=\lambda h_{n+1}(x)$ gives the conjugacy equation.
		The argument in Section~\ref{PTD} together with concavity gives
		\[
		|h'(x)-h'(y)|\leqslant{C_{55}}\sum_{k=0}^{\infty}\nu(q^k|x-y|)
		\leqslant{C_{56}}\nu_1(|x-y|).
		\]
		The inverse assertion follows as in that section.
	\end{proof}

		Throughout the two-dimensional argument,
		\[
		\Lambda=\operatorname{diag}(\lambda_1,\lambda_2),
		\quad 0<|\lambda_1|<1<|\lambda_2|,
		\quad
		\delta=\frac{\min\{\log(1/|\lambda_1|),\log|\lambda_2|\}}
		{\log(|\lambda_2|/|\lambda_1|)}.
		\]
	We denote by $|\cdot|$ the Euclidean norm and its induced matrix operator norm, with respect to which all arclengths and leaf Jacobians are measured. By equivalence of norms and concavity of $\omega$, this geometric choice leaves the class $C_{1,\omega}$ invariant.
	
	We write the original local diffeomorphism as $F(x)=\Lambda x+f(x)$, where $f(O)=O$ and $Df(O)=O$. The $C_{1,\omega}$ regularity allows us, after shrinking the domain of $F$ if necessary, to fix a constant ${C_{57}}>0$ such that
	\[
	|Df(x)-Df(y)|\leqslant {C_{57}}\omega(|x-y|), \quad |f(x)|\leqslant {C_{57}}|x|\omega(|x|).
	\]
	We first extend the perturbation by a cutoff. Fix a smooth cutoff function $\widecheck\chi\in C_c^\infty(\mathbb R^2)$ identically equal to $1$ near $O$. For sufficiently small $r>0$, we define $f_r(x)=\widecheck\chi(x/r)f(x)$ on the original domain and extend it by zero elsewhere. By Leibniz's rule together with the concavity and subadditivity of $\omega$, there exists a constant ${C_{58}}>0$, independent of $r$, such that
	\begin{equation}\label{eq:saddle-cutoff-modulus}
		|Df_r(x)-Df_r(y)| \leqslant {C_{58}}\omega(|x-y|), \quad \forall x,y\in\mathbb R^2.
	\end{equation}
	Furthermore, there exists ${C_{59}}>0$ independent of $r$ satisfying
	\[
	|Df_r|_\infty \leqslant {C_{59}} \sup_{|x|\leqslant {C_{59}}r}|Df(x)| \to 0 \quad \text{as }r\to 0.
	\]
Choose $r>0$ to satisfy all smallness conditions below. Define
\[
G(x):=\begin{pmatrix}
\lambda_1x_1+(f_r(x))_1\\
\lambda_2x_2+(f_r(x))_2
\end{pmatrix}.
\]
If $|\Lambda^{-1}||Df_r|_\infty<1$, then $G$ is a global $C^1$ diffeomorphism of $\mathbb R^2$ belonging to $C_{1,\omega}$: for each $y\in\mathbb R^2$, the mapping $x\mapsto\Lambda^{-1}(y-f_r(x))$ is a contraction, and $DG(x)$ is everywhere invertible. The mapping $G$ coincides with $F$ near $O$.

	Decreasing $r$ further if necessary, the global Hadamard graph transform provides invariant stable and unstable foliations $\mathcal W^s$ and $\mathcal W^u$. The leaves are $C^1$ global graphs over the stable and unstable coordinate axes, respectively, with slopes uniformly bounded in absolute value by a fixed constant $\kappa<1$. The associated tangent line bundles $E^s$ and $E^u$ are continuous. Moreover, there exist constants ${C_{60}}\geqslant1$ and $\rho\in(0,1)$ such that
	\begin{align}
		d_s(G^np,G^nq) &\leqslant {C_{60}}\rho^n d_s(p,q), \quad q\in\mathcal W^s(p), \notag\\
		d_u(G^{-n}p,G^{-n}q) &\leqslant {C_{60}}\rho^n d_u(p,q), \quad q\in\mathcal W^u(p),\quad n\in \mathbb{N},
		\label{eq:saddle-leaf-contraction}
	\end{align}
	where $d_s$ and $d_u$ denote the intrinsic arclength distances on the corresponding leaves.
	
	To control the regularity of these foliations, we fix $r_0>0$ and truncate the modulus by setting $\omega_*(t)=\omega(\min\{t,r_0\})$. Modifying the constant in~\eqref{eq:saddle-cutoff-modulus} yields a constant ${C_{61}}>0$ such that
	\begin{equation}\label{eq:saddle-DG-modulus}
		|DG(p)-DG(q)| \leqslant {C_{61}}\omega_*(|p-q|), \quad \forall p,q\in\mathbb R^2.
	\end{equation}
	Choose
	\[
	L>\max\{|\lambda_2|,|\lambda_1|^{-1}\},
	\quad \frac{|\lambda_1|}{|\lambda_2|}<\vartheta<1,
	\quad \vartheta L<1.
	\]
	Such a choice is possible because
	\[
	\frac{|\lambda_1|}{|\lambda_2|}
	\max\{|\lambda_2|,|\lambda_1|^{-1}\}
	=\max\{|\lambda_1|,|\lambda_2|^{-1}\}<1.
	\]
	For sufficiently small $r$, both $G$ and $G^{-1}$ have Lipschitz constant less than $L$. We represent the tangent fields in slope coordinates as $E^u(p)=\operatorname{span}\{(u(p),1)^\top\}$ and $E^s(p)=\operatorname{span}\{(1,s(p))^\top\}$. For any matrix $M=(m_{ij})$ sufficiently close to $\Lambda$, the action of $M$ on the unstable and backward stable slopes is given by
	\[
	T_M^u(z)=\frac{m_{11}z+m_{12}}{m_{21}z+m_{22}}, \quad T_M^s(z)=\frac{m_{11}z-m_{21}}{m_{22}-m_{12}z}.
	\]
	At the unperturbed matrix $M=\Lambda$, both rational mappings have derivative equal to $\lambda_1/\lambda_2$. After decreasing $r$, both mappings have Lipschitz constant at most $\vartheta$ on the relevant slope intervals, and depend Lipschitz continuously on $M$. Setting $\Gamma(t)=\sup_{|p-q|\leqslant t}\max\{|u(p)-u(q)|,|s(p)-s(q)|\}$, this projective contraction implies the functional inequality
	\begin{equation}\label{eq:saddle-Gamma-recursion}
		\Gamma(t) \leqslant\vartheta\Gamma(Lt) +{C_{62}}\omega_*(Lt)
	\end{equation}
	for some constant ${C_{62}}>0$. Iterating~\eqref{eq:saddle-Gamma-recursion} and using the boundedness of $\Gamma$, we obtain with ${C_{63}}:={C_{62}}/\vartheta>0$ that
	\begin{equation}\label{eq:saddle-tangent-modulus}
		\Gamma(t) \leqslant {C_{63}}\sum_{j=0}^{\infty} \vartheta^j\omega_*(L^jt) =:\widecheck\omega(t).
	\end{equation}
	The same continuity estimate holds for the continuously oriented unit tangent vector fields along the foliations, 	upon choosing ${C_{62}}$ sufficiently large.

		Since $\vartheta L<1$, concavity gives $\omega_*(L^jt)\leqslant L^j\omega_*(t)$; hence
		\begin{equation}\label{eq:saddle-tangent-original-modulus}
			\widecheck\omega(t)\leqslant{C_{64}}\omega_*(t),
			\quad {C_{64}}:=\frac{{C_{63}}}{1-\vartheta L}.
		\end{equation}
		In particular, $\widecheck\omega$ is Dini and the unit tangent fields have modulus $\omega_*$.
	Let $v_s$ and $v_u$ be the continuous unit tangent vector fields to $\mathcal W^s$ and $\mathcal W^u$, normalized so that their first and second components, respectively, remain positive. We define the leaf expansion factors and their logarithmic Jacobians by
	\[
	J^\sigma(p)=|DG(p)v_\sigma(p)|, \quad a^\sigma(p)=\log J^\sigma(p), \quad \sigma\in\{s,u\}.
	\]
	Since $DG$ and $DG^{-1}$ are uniformly bounded, combining~\eqref{eq:saddle-DG-modulus} and~\eqref{eq:saddle-tangent-modulus} provides a constant ${C_{65}}>0$ such that
	\begin{equation}\label{eq:saddle-Jacobian-modulus}
		|a^\sigma(p)-a^\sigma(q)| \leqslant {C_{65}}\overline\omega(|p-q|), \quad \overline\omega:=\omega_*+\widecheck\omega,
	\end{equation}
	which proves in particular that $\overline\omega$ is a Dini modulus of continuity.
	
	We next prove that the holonomies are $C^1$. Consider a stable holonomy $h:U_1\to U_2$ between two unstable plaques inside a small local product box. All constants below remain uniform under perturbation of the plaques within a slightly smaller product structure. In particular, there exists ${C_{66}}>0$ such that
	\begin{equation}\label{eq:saddle-holonomy-separation}
		d_s(p,h(p))\leqslant {C_{66}}, \quad |G^kp-G^kh(p)| \leqslant {C_{66}}\rho^k, \quad k\in \mathbb{N}.
	\end{equation}
	Let $P_n$ be the projection between the complete unstable leaves
	containing $G^nU_1$ and $G^nU_2$
	along horizontal lines, preserving the
	second coordinate.
	This mapping is well defined because the leaves are global graphs
	over the vertical axis. We define the sequence of approximations $h_n(p)=G^{-n}(P_n(G^n(p)))$ on a fixed compact subplaque of $U_1$. For all sufficiently large $n$, the image of $h_n$ falls into a slightly enlarged target plaque. The uniform slope bounds on the leaves combined with~\eqref{eq:saddle-holonomy-separation} yield a constant ${C_{67}}>0$ such that
	\begin{equation}\label{eq:saddle-projection-error}
		d_u(P_n(G^np),G^nh(p))\leqslant {C_{67}}\rho^n.
	\end{equation}
	Applying the second contraction bound in~\eqref{eq:saddle-leaf-contraction} backward along unstable leaves, there exists ${C_{68}}>0$ such that
	\begin{equation}\label{eq:saddle-approx-orbits}
		d_u(G^kh_n(p),G^kh(p)) \leqslant {C_{68}}\rho^{2n-k}, \quad 0\leqslant k\leqslant n,
	\end{equation}
	which immediately gives the uniform convergence $h_n\to h$.
	
	Orienting the unstable leaves by increasing vertical coordinate, each $h_n$ preserves this orientation because the orientation signs of $G^n$ and $G^{-n}$ cancel. Denoting by $Dh_n$ the positive arclength derivative, the chain rule yields
	\begin{equation}\label{eq:saddle-hn-derivative}
		\log Dh_n(p) = \sum_{k=0}^{n-1} \bigl(a^u(G^kp)-a^u(G^kh_n(p))\bigr)+b_n(p),
	\end{equation}
	where $b_n(p)=\log D_{\rm arc}P_n(G^np)$. Setting $z=G^np$, the definition of $P_n$ gives
	\[D_{\rm arc}P_n(z) = \frac{(v_u(z))_2}{(v_u(P_nz))_2}.\]
	Since the vertical components of $v_u$ are bounded away from zero, estimates~\eqref{eq:saddle-projection-error} and~\eqref{eq:saddle-tangent-modulus} yield a constant ${C_{69}}>0$ such that
	\begin{equation}\label{eq:saddle-projection-derivative}
		|b_n(p)| \leqslant {C_{69}}\widecheck\omega({C_{69}}\rho^n).
	\end{equation}
	
	Define
	\[
	A(p)= \sum_{k=0}^{\infty} \bigl(a^u(G^kp)-a^u(G^kh(p))\bigr).
	\]
	By~\eqref{eq:saddle-Jacobian-modulus} and~\eqref{eq:saddle-holonomy-separation}, the $k$-th term of this series is bounded by ${C_{70}}\overline\omega({C_{70}}\rho^k)$ for some ${C_{70}}>0$. The monotonicity of $\overline\omega$ ensures that, for each fixed $a_0>0$,
	\begin{equation}\label{cxgj}
		\sum_{k=n}^{\infty}\overline\omega(a_0\rho^k) \leqslant \frac1{|\log\rho|} \int_0^{a_0\rho^{n-1}} \frac{\overline\omega(t)}t dt \to 0  \quad \text{as }n\to \infty,
	\end{equation}
	so that the series defining $A(p)$ converges uniformly on the plaque.
	
	To compare $\log Dh_n(p)$ with $A(p)$, we employ~\eqref{eq:saddle-approx-orbits} and~\eqref{eq:saddle-Jacobian-modulus} to find a constant ${C_{71}}>0$ satisfying
	\begin{equation}\label{eq:saddle-replacement-error}
		\sum_{k=0}^{n-1} |a^u(G^kh_n(p))-a^u(G^kh(p))| \leqslant {C_{71}}\sum_{j=n+1}^{2n} \overline\omega({C_{71}}\rho^j).
	\end{equation}
	Combining~\eqref{eq:saddle-hn-derivative}, \eqref{eq:saddle-projection-derivative}, \eqref{cxgj}, and~\eqref{eq:saddle-replacement-error}, we deduce that for a sufficiently large constant ${C_{72}}>0$,
	\begin{equation}\label{eq:saddle-holonomy-C1-convergence}
		\sup_p|\log Dh_n(p)-A(p)| \leqslant {C_{72}}\sum_{j=n}^{\infty} \overline\omega({C_{72}}\rho^j) \to 0 \quad \text{as }n\to \infty.
	\end{equation}
	Hence $Dh_n$ converges uniformly to the positive continuous limit $e^A$. Together with the uniform convergence of $h_n$, this establishes that $h$ is a $C^1$ diffeomorphism with derivative
	\[
	Dh(p)= \exp\left( \sum_{k=0}^{\infty} \bigl(a^u(G^kp)-a^u(G^kh(p))\bigr) \right).
	\]
	Because each partial sum depends continuously on both $p$ and the base plaques, and because the tail estimate~\eqref{eq:saddle-holonomy-C1-convergence} is uniform, the derivative $Dh$ varies continuously with respect to both the point and the plaque parameters. An identical argument applied to $G^{-1}$ shows that the unstable holonomies are likewise of class $C^1$, as the logarithmic Jacobian along the expanding leaves of $G^{-1}$ is given by $-a^s(G^{-1}(p))$, which remains Dini continuous.
	
	We now construct the straightening mapping. Let $\gamma_s$ and $\gamma_u$ be arclength parametrizations of the local invariant manifolds through $O$ satisfying $\gamma_s(0)=\gamma_u(0)=O$, $\gamma_s'(0)=e_s:=(1,0)^\top$, and $\gamma_u'(0)=e_u:=(0,1)^\top$. We define the grid mapping
	\[
	\Theta(\xi,\eta)= \mathcal W^u_{\rm loc}(\gamma_s(\xi)) \cap \mathcal W^s_{\rm loc}(\gamma_u(\eta)).
	\]
	For each fixed $\eta$, the curve $\xi\mapsto\Theta(\xi,\eta)$ is obtained by applying an unstable holonomy to $\gamma_s(\xi)$, whereas for each fixed $\xi$, the curve $\eta\mapsto\Theta(\xi,\eta)$ is a stable holonomy applied to $\gamma_u(\eta)$. Therefore, both partial derivatives of $\Theta$ exist and satisfy
\[
\partial_1\Theta(\xi,\eta)=c_s(\xi,\eta)v_s(\Theta(\xi,\eta)), \quad \partial_2\Theta(\xi,\eta)=c_u(\xi,\eta)v_u(\Theta(\xi,\eta))
\]
	for strictly positive continuous functions $c_s$ and $c_u$. This shows that $\Theta$ is jointly of class $C^1$. Furthermore, $\Theta(\xi,0)=\gamma_s(\xi)$ and $\Theta(0,\eta)=\gamma_u(\eta)$, which implies $\Theta(O)=O$ and $D\Theta(O)=\operatorname{Id}$. By the inverse function theorem, $\Theta$ is a local $C^1$ diffeomorphism.
	
Restricting $G$ to the invariant manifolds defines two one-dimensional mappings $f_s,f_u$ by
\[
G(\gamma_\sigma(t))=\gamma_\sigma(f_\sigma(t)),\quad \sigma\in\{s,u\}.
\]
These mappings satisfy
	\[
	f_s(0)=f_u(0)=0,\quad f_s'(0)=\lambda_1,
	\quad f_u'(0)=\lambda_2,
	\]
	and $|f_\sigma'(t)|=J^\sigma(\gamma_\sigma(t))$ for $\sigma\in\{s,u\}$. Since the parametrizations are by arclength and $a^\sigma$ is Dini continuous, each derivative $f_\sigma'$ admits a Dini modulus of continuity and maintains a constant sign near $0$. By foliation invariance and uniqueness of local intersections, the two-dimensional action decouples as $G(\Theta(\xi,\eta)) =\Theta(f_s(\xi),f_u(\eta))$.
	
	The inverse mapping $f_u^{-1}$ also has a Dini continuous derivative since $(f_u^{-1})'(y)=1/f_u'(f_u^{-1}(y))$ with $f_u'$ bounded away from zero. Applying Lemma~\ref{DINILEMMA} to $f_s$ and $f_u^{-1}$ yields local $C^1$ diffeomorphisms $h_s$ and $h_u$ satisfying
	\[
	h_s(f_s(\xi))=\lambda_1h_s(\xi),
	\quad
	h_u(f_u(\eta))=\lambda_2h_u(\eta),
	\quad
	h_s'(0)=h_u'(0)=1.
	\]
Set
\[
\Phi(x):=\begin{pmatrix}h_s(\xi)\\h_u(\eta)\end{pmatrix},
\quad \begin{pmatrix}\xi\\\eta\end{pmatrix}=\Theta^{-1}(x).
\]
Then we obtain
 a local $C^1$ diffeomorphism $\Phi$ satisfying $\Phi(O)=O$, $D\Phi(O)=\operatorname{Id}$, and $\Phi(G(x))=\Lambda\Phi(x)$.
	Since $G$ coincides with $F$ near $O$, this proves the $C^1$ assertion in Item~\ref{TH4-I}.

		We prove the additional regularity using the holonomy formula. Recall that
	$\omega_{\mathrm D}(t)=\omega_{\rm B}(t^\delta)$ is a concave modulus of continuity.
	We first establish bounds with the exact spectral rates
	$|\lambda_1|$ and $|\lambda_2|$.

		Fix $\rho_0\in(\max\{|\lambda_1|,|\lambda_2|^{-1}\},1)$ and require, when choosing the cutoff radius, that
		$\varepsilon_0:=|Df_r|_\infty$ satisfy $\varepsilon_0 C_{73}\leqslant1/2$, where
		\[
		C_{73}:=\frac1{\rho_0-|\lambda_1|}
		+\frac{\rho_0}{1-|\lambda_1|\rho_0}
		+\frac1{|\lambda_2|-\rho_0}
		+\frac{\rho_0}{|\lambda_2|\rho_0-1}.
		\]
		For an orbit segment $p_k=G^kp_0$, $0\leqslant k\leqslant N$, with
		$|p_0|,|p_N|\leqslant R$, variation of constants gives
	\[
	(p_k)_1=\lambda_1^k(p_0)_1+\sum_{j=0}^{k-1}\lambda_1^{k-1-j}(f_r(p_j))_1,\quad (p_k)_2=\lambda_2^{k-N}(p_N)_2-\sum_{j=k}^{N-1}\lambda_2^{k-1-j}(f_r(p_j))_2.
	\]
		For $w_k=\rho_0^k+\rho_0^{N-k}$, summing the four geometric series gives
		\[
		\max_{0\leqslant k\leqslant N}\frac{|p_k|}{w_k}
		\leqslant R+\varepsilon_0 C_{73}
		\max_{0\leqslant k\leqslant N}\frac{|p_k|}{w_k}.
		\]
		Consequently, $|p_k|\leqslant2R(\rho_0^k+\rho_0^{N-k})$.
		Choose $R>0$ so that $\operatorname{supp}Df_r\subset\{|p|\leqslant R\}$.
		Apply this bound between the first and last visits to that support.
		Subadditivity,~\eqref{eq:saddle-cutoff-modulus}, and the Dini condition \eqref{dinicd} yield
		\begin{equation}\label{eq:saddle-distortion-sum}
			\sum_{k=0}^{n-1}|Df_r(G^kp)|
			\leqslant2C_{58}\sum_{j=0}^{\infty}\omega(2R\rho_0^j)
			=:C_{74}<+\infty,
		\end{equation}
		uniformly in $p,n$; zero or one visit satisfies the same bound.
		
		The invariant frame $T(p)=\begin{pmatrix}1&u(p)\\s(p)&1\end{pmatrix}$ satisfies
		\[
		DG(p)T(p)=T(Gp)\operatorname{diag}(\mu_1(p),\mu_2(p)),\quad |\mu_i(p)-\lambda_i|\leqslant(1+\kappa)|Df_r(p)|,\quad i=1,2.
		\]
		Both $T$ and $T^{-1}$ are uniformly bounded. For sufficiently small $r$,
		\[
		\frac{\mu_i(p)}{\lambda_i}>0,
		\quad
		\left|\log\frac{\mu_i(p)}{\lambda_i}\right|
		\leqslant\frac{2(1+\kappa)}{|\lambda_1|}|Df_r(p)|,
		\quad i=1,2.
		\]
		Thus~\eqref{eq:saddle-distortion-sum} bounds the products of these ratios
		above and away from zero. Multiplication in this frame and integration
		along leaves give a constant $C_{75}\geqslant1$ such that
		\begin{equation}\label{eq:saddle-exact-rates}
			\begin{aligned}
				|DG^n(p)|&\leqslant C_{75}|\lambda_2|^n,
				& |DG^{-n}(p)|&\leqslant C_{75}|\lambda_1|^{-n},\\
				d_s(G^np,G^nq)&\leqslant C_{75}|\lambda_1|^nd_s(p,q),
				&q&\in\mathcal W^s(p),\\
				d_u(G^{-n}p,G^{-n}q)&\leqslant C_{75}|\lambda_2|^{-n}d_u(p,q),
				&q&\in\mathcal W^u(p).
			\end{aligned}
		\end{equation}
		
		For $x=\Theta(\xi,\eta)$, put $p_u(x)=\gamma_u(\eta)$.
		This mapping is Lipschitz on a smaller compact product box, and
		\[
		\log c_u(\Theta^{-1}(x))
		=\sum_{k=0}^{\infty}
		\bigl(a^u(G^kp_u(x))-a^u(G^kx)\bigr).
		\]
		Compare corresponding orbits using the growth bound $|\lambda_2|^k$,
		or the two points within each summand using the stable rate $|\lambda_1|^k$.
		Now, \eqref{eq:saddle-tangent-original-modulus}, \eqref{eq:saddle-Jacobian-modulus}
		and~\eqref{eq:saddle-exact-rates}, with concavity absorbing fixed factors, give
		\[
		\begin{aligned}
			\left|\log\frac{c_u(\Theta^{-1}(x))}{c_u(\Theta^{-1}(y))}\right|
			&\leqslant C_{76}\sum_{k=0}^{\infty}
			\min\{\omega(|\lambda_2|^kt),\omega(|\lambda_1|^k)\},\\
			\left|\log\frac{c_s(\Theta^{-1}(x))}{c_s(\Theta^{-1}(y))}\right|
			&\leqslant C_{76}\sum_{k=0}^{\infty}
			\min\{\omega(|\lambda_1|^{-k}t),\omega(|\lambda_2|^{-k})\},
		\end{aligned}
		\]
		where $t=|x-y|$; the second estimate follows by applying the same argument
		to $G^{-1}$.
		
		Split the first sum at
		\[
		N=\left\lfloor\frac{\log(1/t)}
		{\log(|\lambda_2|/|\lambda_1|)}\right\rfloor.
		\]
		Comparison on geometric intervals bounds it by
		\[
		\omega(|\lambda_2|^Nt)
		+\frac1{\log|\lambda_2|}\int_t^{|\lambda_2|^Nt}\frac{\omega(s)}s ds
		+\frac1{\log(1/|\lambda_1|)}\int_0^{|\lambda_1|^N}\frac{\omega(s)}s ds.
		\]
		Since
		\[
		|\lambda_2|^Nt\leqslant|\lambda_1|^N
		\asymp t^{\log(1/|\lambda_1|)/\log(|\lambda_2|/|\lambda_1|)},
		\]
		concavity and $\omega\leqslant\omega_{\rm B}$ bound this expression by
		$C_{77}\omega_{\mathrm D}(t)$.
		Exchanging $|\lambda_1|$ and $|\lambda_2|^{-1}$ gives the same bound for
		the second sum, choosing $C_{77}$ to cover both estimates.
		The formulas for $\partial_1\Theta,\partial_2\Theta$ now give
		$\Theta,\Theta^{-1}\in C_{1,\omega_{\mathrm D}}$, using the inverse argument
		of Section~\ref{PTD}. The axis mappings have derivative modulus $\omega$,
		so Lemma~\ref{DINILEMMA} gives $h_s,h_u\in C_{1,\omega_{\rm B}}$.
		Since $\omega_{\rm B}\leqslant\omega_{\mathrm D}$ near zero, composition yields
		\[
		|D\Phi(x)-D\Phi(y)|\leqslant C_{78}\omega_{\mathrm D}(|x-y|).
		\]
		The inverse has the same modulus of continuity by Section~\ref{PTD}, completing
		the additional assertion in Item~\ref{TH4-I}.

	We now prove Item~\ref{TH4-II}, assuming that~\eqref{dinicd} fails.
	Recall the perturbation $R$ from~\eqref{eq:3.3} in Section~\ref{SEC3}.
	For $\varepsilon>0$, set
\[
F_\varepsilon(x_1, x_2) := \begin{pmatrix} \lambda_1 x_1 \\ \lambda_2 x_2 + \varepsilon \lambda_2 R(x_1, x_2) \end{pmatrix}.
\]
	Lemma~\ref{pro:regularity} gives $F_\varepsilon\in C_{1,\omega}$,
	$F_\varepsilon(O)=O$, and $DF_\varepsilon(O)=\Lambda$; hence it is a local
	diffeomorphism. Since $K(0)=0$ and $K'(0)=1$,
\begin{equation}\label{eq236}
		F_\varepsilon(x,0)=\begin{pmatrix}\lambda_1x\\0\end{pmatrix},
		\quad
		\partial_2(F_\varepsilon)_2(x,0)
		=\lambda_2\bigl(1+\varepsilon\widehat\omega(|x|)\bigr).
	\end{equation}
	Suppose that a local $C^1$ linearization $\Phi$ exists, normalized by
	$\Phi(O)=O$ and $D\Phi(O)=\operatorname{Id}$.
	Differentiating its conjugacy equation on this axis and using~\eqref{eq236}
	give
	\[
	\bigl(1+\varepsilon\widehat\omega(|x|)\bigr)
	\partial_2\Phi_2(\lambda_1x,0)=\partial_2\Phi_2(x,0).
	\]
	Iteration yields
	\begin{equation}\label{cjcj}
		\partial_2\Phi_2(\lambda_1^nx,0)
		=\partial_2\Phi_2(x,0)
		\prod_{k=0}^{n-1}
		\bigl(1+\varepsilon\widehat\omega(|\lambda_1|^k|x|)\bigr)^{-1}.
	\end{equation}
	For a fixed sufficiently small $x\neq0$, geometric-interval comparison
	and Lemma~\ref{prop21} give
	\[
	\sum_{k=0}^{\infty}\widehat\omega(|\lambda_1|^k|x|)
	\geqslant\frac1{2\log(1/|\lambda_1|)}
	\int_0^{|x|}\frac{\omega(t)}t dt=+\infty.
	\]
	Since $\log(1+u)\geqslant u/2$ for small $u\geqslant0$, the product
	in~\eqref{cjcj} tends to zero. Its left-hand side instead tends to
	$\partial_2\Phi_2(O)=1$, a contradiction. This proves Item~\ref{TH4-II}.

		For sharpness under~\eqref{dinicd}, first suppose $|\lambda_1\lambda_2|\leqslant1$,
		so that
		\[
		\delta=\frac{\log|\lambda_2|}{\log(|\lambda_2|/|\lambda_1|)}.
		\]
		Using the same perturbation~\eqref{eq:3.3} with its arguments interchanged, set
\[
F_\varepsilon(x_1, x_2) := \begin{pmatrix} \lambda_1 x_1 + \varepsilon \lambda_1 R(x_2, x_1) \\ \lambda_2 x_2 \end{pmatrix}, \quad \varepsilon > 0.
\]
		For $x_2\neq0$, its first component is
		$\lambda_1x_1[1+\varepsilon\widehat\omega(|x_2|)\chi(x_1/x_2)]$.
		Lemma~\ref{pro:regularity} gives $F_\varepsilon\in C_{1,\omega}$ and
		$DF_\varepsilon(O)=\Lambda$, so the preceding sufficiency argument applies.
		
		Let $\Phi\in C_{1,\eta}$ be any normalized local linearization.
		Choose $\xi,\zeta>0$ so that the rectangle
		$|x_1|\leqslant\xi$, $|x_2|\leqslant\zeta$ lies in its conjugacy domain,
		$|\lambda_1|(1+\varepsilon\omega(\zeta))<1$, and
		$\varepsilon\omega(\zeta)\leqslant1$. Write
		\[
		p_n=(\xi,\lambda_2^{-n}\zeta),
		\quad F_\varepsilon^k(p_n)=(x_{n,k},y_{n,k}),
		\quad 0\leqslant k\leqslant n.
		\]
		These orbit segments remain in the rectangle, with
		$y_{n,k}=\lambda_2^{k-n}\zeta$ and
		$|x_{n,k}|\geqslant|\lambda_1|^k\xi$. Define
	\[
	P_n := \prod_{k=0}^{n-1} \left[1+\varepsilon\widehat\omega(|\lambda_2|^{k-n}\zeta)\chi\left(\frac{x_{n,k}}{y_{n,k}}\right)\right], \quad
	Q := \prod_{j=1}^{\infty} \left(1+\varepsilon\widehat\omega(|\lambda_2|^{-j}\zeta)\right).
	\]
		The Dini condition gives $1\leqslant P_n\leqslant Q<+\infty$, and
		\[
		F_\varepsilon^n(p_n)=(u_n,\zeta),
		\quad u_n=\lambda_1^n\xi P_n,
		\quad |u_n|\asymp|\lambda_1|^n.
		\]
		Iteration on the invariant axes and normalization give
		$\Phi_1(\xi,0)=\xi$ and $\Phi_1(0,\zeta)=0$.
		For $h(y)=\partial_1\Phi_1(0,y)$, the differentiated conjugacy equation is
		\[
		h(\lambda_2y)\bigl(1+\varepsilon\widehat\omega(|y|)\bigr)=h(y),
		\quad h(0)=1.
		\]
		Backward iteration therefore gives $h(\zeta)=Q^{-1}$, whence
		\[
		\Phi_1(u_n,\zeta)=Q^{-1}u_n+\mathcal O(|u_n|\eta(|u_n|)),
		\quad \Phi_1(p_n)=\xi+\mathcal O(|\lambda_2|^{-n}).
		\]
		Substitution into $\Phi_1(F_\varepsilon^np_n)=\lambda_1^n\Phi_1(p_n)$ yields
		a constant $C_{79}\geqslant1$, independent of $n$, such that
		\begin{equation}\label{eq:saddle-sharp-entry-exit}
			0\leqslant1-P_n/Q
			\leqslant C_{79}\bigl(|\lambda_2|^{-n}
			+\eta(C_{79}|\lambda_1|^n)\bigr).
		\end{equation}
		
		For sufficiently large $n$, let
		\[
		m_n=\left\lfloor
		\frac{n\log|\lambda_2|+\log(\xi/(2\zeta))}
		{\log(|\lambda_2|/|\lambda_1|)}
		\right\rfloor.
		\]
		For $0\leqslant k\leqslant m_n$, the ratio
		$|x_{n,k}/y_{n,k}|$ is at least $2$, so the corresponding cutoff vanishes.
		Comparing the positive products, and using Lemma~\ref{prop21}, gives
		\[
		\begin{aligned}
			\log(Q/P_n)
			&\geqslant\sum_{j=n-m_n}^{\infty}
			\log\left(1+\varepsilon\widehat\omega(|\lambda_2|^{-j}\zeta)\right)\\
			&\geqslant\frac\varepsilon4
			\sum_{j=n-m_n}^{\infty}\omega(|\lambda_2|^{-j}\zeta)\\
			&\geqslant\frac\varepsilon{4\log|\lambda_2|}
			\omega_{\rm B}(|\lambda_2|^{m_n-n}\zeta).
		\end{aligned}
		\]
		Since $1-P_n/Q\geqslant Q^{-1}\log(Q/P_n)$ and
		$|\lambda_2|^{m_n-n}\zeta\asymp|\lambda_1|^{n\delta}$,
		concavity yields a constant $C_{80}>0$, independent of $n$, satisfying
		\begin{equation}\label{eq:saddle-sharp-product-tail}
			1-P_n/Q\geqslant C_{80}^{-1}\omega_{\mathrm D}(|\lambda_1|^n).
		\end{equation}
		Concavity and positivity give $t=\mathcal O(\omega_{\rm B}(t))$, while
		\[
		\frac{|\lambda_2|^{-n}}{|\lambda_1|^{n\delta}}
		=\exp\left(-\frac{n(\log|\lambda_2|)^2}
		{\log(|\lambda_2|/|\lambda_1|)}\right)\to0.
		\]
		Thus the term $|\lambda_2|^{-n}$ can be absorbed in
		\eqref{eq:saddle-sharp-entry-exit} and \eqref{eq:saddle-sharp-product-tail}, giving
		\[
		\omega_{\mathrm D}(|\lambda_1|^n)
		\leqslant C_{81}\eta(C_{79}|\lambda_1|^n).
		\]
		Interpolation between consecutive scales $C_{79}|\lambda_1|^n$, as in the
		Jordan case, proves $\omega_{\mathrm D}(t)=\mathcal O(\eta(t))$.
		The symmetry of $\chi$ makes the argument valid for either sign of each eigenvalue.
		
		If $|\lambda_1\lambda_2|>1$, apply the construction to the pair
		$(\lambda_2^{-1},\lambda_1^{-1})$, obtaining $H$, and set
\[
F:=S\circ H^{-1}\circ S,\quad
S(x_1,x_2):=\begin{pmatrix}x_2\\x_1\end{pmatrix}.
\]
		Inversion preserves $C_{1,\omega}$, and a normalized $C_{1,\eta}$ linearization
		$\Phi$ of $F$ induces one of $H$ via $S\circ\Phi\circ S$.
		The corresponding exponent is
		\[
		\frac{\log(1/|\lambda_1|)}{\log(|\lambda_2|/|\lambda_1|)}=\delta,
		\]
		which proves the remaining case.

	This completes the proof of Theorem~\ref{TH4}.
\end{proof}

	\subsection{Proof of Theorem \ref{THNH}}
	
	\begin{proof}
		Throughout this subsection, $|\cdot|$ denotes the Euclidean norm on $\mathbb{R}^2$. Define
\[
		H(x):=(1+|x|^2)\begin{pmatrix}x_1\\x_2\end{pmatrix},\quad F(x):=\Lambda H(x).
		\]
		The mapping $H$ preserves each ray emanating from the origin $O$, acting on the radial coordinate $t\geqslant 0$ by $t\mapsto t+t^3$, which is a real-analytic bijection of $[0,\infty)$ onto itself. Consequently, $H$ is a global bijection. A direct calculation yields the Jacobian matrix
		\[
		DH(x)=(1+|x|^2)I+2xx^\top,
		\]
		whence $\det DH(x)=(1+|x|^2)(1+3|x|^2)>0$ for all $x\in\mathbb{R}^2$. By the real-analytic inverse function theorem, $H$ is a real-analytic diffeomorphism of $\mathbb{R}^2$. It follows that $F$ is a global real-analytic diffeomorphism with polynomial components of degree at most three, satisfying $F(O)=O$ and $DF(O)=\Lambda$.
		
		Let $\sigma(\Lambda)$ denote the spectrum of $\Lambda$. Since $\sigma(\Lambda)$ is finite, we may choose $r>0$ such that
		\begin{equation}\label{rxq}
			(1+r^2)|\lambda|<1,\quad\forall \lambda\in\sigma(\Lambda)\text{ with }|\lambda|<1.
		\end{equation}
		In particular, no eigenvalue of $\Lambda$ has modulus in $[(1+r^2)^{-1},1)$.
		
		We claim that there exists no nonempty compact set $K\subset\{x\in\mathbb{R}^2:0<|x|<r\}$ such that $F(K)=K$. Suppose, to the contrary, that such a set $K$ exists, and set $\delta:=\min_{x\in K}|x|>0$. For any $x\in K$, all forward iterates remain in $K$. An immediate induction gives
		\[
		F^n(x)=\left(\prod_{j=0}^{n-1}(1+|F^j(x)|^2)\right)\Lambda^n x,\quad n\in\mathbb{N}^+.
		\]
		Since $\delta\leqslant |F^j(x)|\leqslant r$ for each $j\geqslant 0$, we deduce that
		\[
		\frac{\delta}{(1+r^2)^n}\leqslant |\Lambda^n x|\leqslant \frac{r}{(1+\delta^2)^n},\quad n\in\mathbb{N}^+.
		\]
		By the Jordan canonical form of $\Lambda$, the limit
		\[
		\rho_x:=\lim_{n\to\infty}|\Lambda^n x|^{1/n}
		\]
		exists and equals the modulus of an eigenvalue of $\Lambda$. Taking $n$-th roots in the preceding inequalities and passing to the limit as $n\to\infty$ yield
		\[
		\frac{1}{1+r^2}\leqslant\rho_x\leqslant\frac{1}{1+\delta^2}<1,
		\]
		contradicting the choice of $r$ in \eqref{rxq}. This establishes the claim.
		
		On the other hand, since $\sigma(\Lambda)\cap S^1\neq\varnothing$, every neighborhood of $O$ contains a nonempty compact $\Lambda$-invariant set disjoint from $\{O\}$. Indeed, an eigenvector corresponding to the eigenvalue $1$ provides a nonzero fixed point, whereas one corresponding to $-1$ yields a two-point periodic orbit; if $\Lambda$ possesses nonreal eigenvalues of modulus $1$, it is conjugate over $\mathbb{R}$ to a rotation, thereby preserving a family of concentric ellipses disjoint from $\{O\}$. By linearity, scaling these sets by arbitrarily small positive factors yields such invariant sets in any prescribed neighborhood of $O$.
		
		Now suppose that $F$ is locally topologically linearizable at $O$, namely, there exists a homeomorphism $\Phi:U\to V$ between open neighborhoods of $O$ with $\Phi(O)=O$ such that $\Phi(F(x))=\Lambda\Phi(x)$ near $O$. Choose an open neighborhood $W$ of $O$ such that
		\[
		W\subset U\cap\{x\in\mathbb{R}^2:|x|<r\},\quad F(W)\subset U,
		\]
		and the conjugacy relation holds identically on $W$. Choose a nonempty compact $\Lambda$-invariant set $K_0\subset\Phi(W)\setminus\{O\}$ and set $K:=\Phi^{-1}(K_0)$. Then $K$ is a nonempty compact subset of $\{x\in\mathbb{R}^2:0<|x|<r\}$ satisfying $K\subset W$. The conjugacy equation then gives
		\[
		F(K)=\Phi^{-1}(\Lambda K_0)=\Phi^{-1}(K_0)=K,
		\]
		which contradicts the claim.
		
		The proof of Theorem~\ref{THNH} is now complete.
	\end{proof}

	\section*{Acknowledgements} 
	Z. Tong  was supported by the National Natural Science Foundation of China (Grant No. 12601347) and the China Postdoctoral Science Foundation (Grant No. 2025M783102). Y. Li was supported in part by the National Natural Science Foundation of China (Grant Nos. 12471183 and 12531009).

\end{document}